\documentclass[reqno,12pt]{article}
\usepackage{palatino}
\usepackage{amssymb,amsthm,amsmath,amsfonts}
\usepackage{epsfig}

\usepackage{helvet}
\usepackage{courier}
\usepackage{type1cm}

\usepackage{makeidx}         
\usepackage{graphicx}        
\usepackage{multicol}        
\usepackage[bottom]{footmisc}

\theoremstyle{plain}
\begingroup

\newtheorem{thm}{Theorem}[section]
\newtheorem{theorem}{Theorem}[section]

\newtheorem{corollary}[thm]{Corollary}
\newtheorem{lemma}[thm]{Lemma}

\newtheorem{proposition}[thm]{Proposition}

\endgroup
\theoremstyle{definition}
\newtheorem{defn}{Definition}[section]
\newtheorem{definition}{Definition}[section]
\newtheorem{remark}[defn]{Remark}

\newtheorem{example}[defn]{Example}
\numberwithin{equation}{section}
\numberwithin{figure}{section}

\DeclareMathOperator{\re}{Re} 
\DeclareMathOperator{\im}{Im}

\DeclareMathOperator*{\res}{\mathrm{Res}}

\def\I{\mathrm{i}}

\def\D{{\mathbb D}}
\def\R{{\mathbb R}}
\def\C{{\mathbb C}}

\def\Z{{\mathbb Z}}

\begin{document}

\title{Two dimensional potential theory \\with a view towards vortex motion:\\
Energy, capacity and Green functions
\textsuperscript{1}}
\author{
Bj\"orn Gustafsson\textsuperscript{2}
}
 
\date{\today}

\maketitle

\tableofcontents

\begin{abstract} 
The paper reviews some parts of classical potential theory with applications to two dimensional fluid dynamics, in particular vortex motion. 
Energy and forces within a system of point vortices are similar to those for point charges when the vortices are kept fixed, but the dynamics 
is different in the case of free vortices. Starting from Euler's and Bernoulli's equations we derive these laws. Letting the number of vortices tend to infinity 
leads in the limit to considerations of equilibrium distributions, capacity, harmonic measure and many other notions in potential theory. In particular various kinds of Green functions have a central role in the paper, and we make a distinction between electrostatic and hydrodynamic Green functions.

We also consider the corresponding concepts in the case of closed Riemann surfaces provided with a metric. From a canonically defined 
monopole Green function we rederive much of the classical theory of harmonic and analytic forms on Riemann surfaces. In the final section 
of the paper we return to the planar case, which then reappears in the form of a symmetric Riemann surface, the Schottky double. Bergman kernels, 
electrostatic  and hydrodynamic, come up naturally as well as other kernels, 
and associated to the Green function there is a certain Robin function which is important for vortex 
motion, and which also relates to capacity functions in classical potential theory.
\end{abstract}

\noindent {\it Keywords:} Vortex motion, energy, capacity, harmonic measure, balayage, 
equilibrium distribution, transfinite diameter,
hydrodynamic and electrostatic Green function, monopole Green function, Robin function, 
Bergman kernel, Schiffer kernel, Szeg\"o kernel, Garabedian kernel.

\noindent {\it MSC:} 31A15, 30F15, 76B47

 \footnotetext[1]
{Revised version. The present version provides minor corrections, additions and stylistic adjustments compared to the published version in {\it Geometric Mechanics},\\
https://www.worldscientific.com/doi/epdf/10.1142/S2972458924300018}

\footnotetext[2]
{Department of Mathematics, KTH, 100 44, Stockholm, Sweden.\\
Email: \tt{gbjorn@kth.se}}


\section{Introduction}

\subsection{General}

This paper reviews some parts of classical potential theory with applications to 
two dimensional hydrodynamics, and to some extent electrostatics. There are four main sections,
which naturally make the paper fall into two major parts. The first part, the body of which consists Sections~\ref{sec:energy and Green}
and \ref{sec:capacity}, essentially discusses classical potential theory in the plane, with concepts such as energy, capacity
and Green functions. The second part, Sections~\ref{sec:monopole Green} and \ref{sec:planar domains},
is more oriented towards geometric function theory on Riemann surfaces. But it eventually returns to  planar
domains, such a domain then being ``doubled'' to a compact Riemann surfaces, the Schottky double. 
In this second part, analytic differentials and kernel functions, such as the Bergman kernel, play central roles.

The first main part of the paper is in principle very classical, with some results dating back almost 200 years.
In some sense, the origin of potential theory is Newton's ``Principia''
(``Philosophiae Naturalis Principia Mathematica'', 1687) 
\cite{Newton-1687}, but more appropriate for the present treatment is George Green's   ``Essay on the Applications of Mathematical Analysis
to the Theory of Electricity and Magnetism'' (1828) \cite{Green-1828},
in which the word ``potential'' was coined, and from which  (later) the term ``Green function'' came into use.
A classical text in potential theory is Kellogg's book \cite{Kellogg-1967} from 1929.
Still, the present treatment has a somewhat personal slant, and some results and points of view may be new.
And in any case it leads into presently very active areas of research.

The latter is even more true for the second main part of the paper, which is a further development of several recent papers,
for example  \cite{Grotta-Ragazzo-Gustafsson-Koiller-2024}, and which also is inspired by present high activity
in geometric function theory related to physics areas such as conformal field theory and string theory. 
Much of the mathematics in this part is based on the classical text \cite{Schiffer-Spencer-1954}
by M.~Schiffer and D.~Spencer, in particular chapter~4  in
that book. However, we start out with having an independent metric available, which makes it possible to base the 
exposition on a certain monopole Green function. Such a path has also been followed in the book \cite{Lang-1988},
which is partly based on ideas of G.~Faltings \cite{Faltings-1984}, and in the paper \cite{Takhtajan-2001}.
The latter text in addition links the subject to quantum field theory and conformal field theory.

Conformal field theory is related to vortex motion in various ways. One example is the Liouville equation,
which is an equation satisfied, in the case of a simply connected fluid region, by the Kirchoff-Routh stream function, 
a potential function which was introduced in \cite{Lin-1943} (see also \cite{Lamb-1993}) 
and which in the present text appears as a Robin function. 
This function is subject to an inhomogeneous transformation law, which makes it, in the terminology of
\cite{Kang-Makarov-2013}, be a ``pre-pre-Schwarzian form''.
In modern physics the Liouville equation plays a central role in quantum gravity. A few articles in this respect are 
\cite{Alvarez-1987, Zamolodchikov-2005, Takhtajan-Teo-2006}.

To describe in some more detail the contents of the present paper, we begin by  summarizing
some notations and conventions in  Section~\ref{sec:notations}. 
After that, Section~\ref{sec:energy and Green} starts with discussing  the force between two bound 
point vortices in the plane. 
The force law in this context is similar to the two-dimensional versions of the Newton and Coulomb laws, 
with the force decaying as one over the distance, but we derive it in a somewhat special way,
namely form Bernoulli's law of hydrodynamics. This means that it comes out as a consequence, not
of Newton's law of gravity, but of Newton's general laws of mechanics, specifically his ``second law'', applied to
a continuum of particles.

We also discuss the motion of free vortices, with a derivation of the dynamics and regularization of Euler's 
equation adapted to our special approach. The energy of a system of point vortices, which actually
is a renormalized total kinetic energy of the fluid as a whole, has the same
form as the electrostatic energy for a system of point charges. But for a multiply connected fluid region,
the hydrodynamic and electrostatic points of view diverge, each of them having its own Green function. 
We also make efforts to connect the various notions of capacity, transfinite diameter and 
related Robin constants to each other. 

The above material is more exactly exposed in Section~\ref{sec:capacity}, where also concepts of equilibrium potential, 
harmonic measure of balayage are discussed, plus Hadamard's variational formula for the Green function.
At the end of Section~\ref{sec:capacity} we introduce the Bergman kernel, 
a mixed second order derivative of the Green function which
reproduces values for analytic functions in a Hilbert space setting.  
Further analysis of the Bergman and related kernels from the point of view of general  Riemann surface theory
is taken on in Sections~\ref{sec:monopole Green} and \ref{sec:planar domains}.

In Section~\ref{sec:monopole Green} we turn to curved surfaces, more precisely closed Riemann surfaces provided 
with a Riemannian metric compatible with the conformal structure. 
Here we partly redo the classical theory of harmonic and analytic differentials by
using the monopole Green function. As the name indicates it has only one pole, which is 
more precisely a logarithmic singularity, and as a substitute for the 
necessary (on a closed surface) counter-pole there is a uniform counter-sink, or counter-vorticity in our terminology.
Its density is proportional to the density of the metric, and it appears naturally in the context of Hodge decomposition of the
given source.

The regular part of the monopole Green function starts, when Taylor expanded around the pole, with a constant term which
depends on the location of the pole, and also on the local coordinate used. 
This is by definition the (coordinate) Robin function, or Kirchoff-Routh's stream function. The dependence on the
coordinate is such that the exponential of it is (after suitable scaling) a conformally invariant metric. 
The dependence on the coordinate can be avoided by  expressing the singularity of the Green function in terms of the intrinsic metric. 
Then the Robin function becomes a true function. 

Because of the uniform counter-sink the monopole Green function is not harmonic, and it might therefore seem surprising that
the classical theory of harmonic differentials on Riemann surfaces can be developed from it. 
But the problem with the Green function not being harmonic disappears after some 
differentiations, and we try to explain in detail how it works out. Similar approaches are available in
\cite{Lang-1988, Takhtajan-2001}. In this context we also re-derive some recent results
of Okikiolu, Steiner, and Grotta-Ragazzo on the  the Laplacian of the Robin function. 
The mixed second order derivatives of the Green function, or of its regular part, are fundamental in Riemann surface theory. 
These are the Schiffer kernel (meromorphic with a double pole) and the Bergman kernel (everywhere holomorphic),
the latter being reproducing kernel for global holomorphic one-forms on the surface.

In the final Section~\ref{sec:planar domains} we return to the planar case, but now in the light of the theory on closed surfaces. 
More precisely, we complete the planar surface with a backside so that it becomes a compact Riemann surface, the Schottky double. 
This is a classical technique, but usually one stays with just the conformal structure. For the vortex motion
one also needs a metric, and it is actually possible to just copy the planar Euclidean metric to backside and use it right away.
The drawback is that it becomes singular over the boundary, but it is no worse than the density of the metric becoming
just Lipschitz continuous when expressed in a coordinate which is smooth over the boundary. And the fluid will not cross this boundary 
anyway, the boundary is a streamline.

For the homology basis for the double of a planar domain of connectivity $\texttt{g}+1$ there is a natural division 
of the curves into  $\alpha$-curves and $\beta$-curves, where in our presentation the $\beta$-curves are those
which are homologous to the $\texttt{g}$ inner boundary components of the domain. These two classes of curves
represent a dichotomy of the continued development of the planar geometric function theory into {hydrodynamic} and 
{electrostatic} versions. This terminology is inspired by a book by H.~Cohn \cite{Cohn-1980}.
The ordinary Dirichlet Green function for the domain will now be renamed as being the 
electrostatic Green function, $G_{\rm electro}(z,a)$, and the specially adapted Green function for fluid dynamics
will be called the hydrodynamic Green function, $G_{\rm hydro}(z,a)$. The latter is really nothing new.  For vortex motion
it has been used in \cite{Lin-1943, Gustafsson-1979, Flucher-Gustafsson-1997, Crowdy-Marshall-2005} and in several
more recent papers. Some versions of it are also well-known from the theory of conformal mapping of multiply connected domains
\cite{Koebe-1916, Crowdy-Marshall-2006}.  See also the survey paper \cite{Crowdy-2008}.
Both types of  Green functions, as well as other domain functions, can be obtained from
the monopole Green function for the Schottky double, which we in this context denote $G_{\rm double}(z,a)$.  

From the second mixed derivatives of $G_{\rm electro}(z,a)$ and $G_{\rm hydro}(z,a)$ one obtains
Bergman reproducing kernels for Hilbert spaces of analytic functions in the planar domain,
see in general \cite{Bergman-1970}. For the purpose of the present article these will be denoted
$K_{\rm electro}(z,a)$ and $K_{\rm hydro}(z,a)$, where the first one is the ordinary Bergman kernel, briefly discussed 
also in Section~\ref{sec:Bergman}. 

The second kernel, $K_{\rm hydro}(z,a)$, which is the Bergman kernel for those analytic function which have a single-valued
integral, can be obtained also from Neumann type functions, and we give an example of how this works out. Finally, 
towards the end of the paper we mention the Szeg\"o kernel, denoted $K_\text{Szeg\"o}(z,a)$, which is a reproducing kernel with respect to arc length 
measure on the boundary. The square of the Szeg\"o kernel is squeezed between the two Bergman kernels and can be
related to a somewhat mysterious Green function, which we denote $G_\text{Szeg\"o}(z,a)$. 
It  seems to be an open question whether this Green function has any physical interpretation.
The Szeg\"o kernel itself, however, plays important roles in many applications. One over the Szeg\"o kernel is closely
related to the so-called prime form \cite{Hejhal-1972, Fay-1973, Crowdy-2010}, and in conformal field theory the Szeg\"o kernel
appears as a fermionic propagator, or two-point function \cite{Raina-1989, Hoker-2024}. 
 
Part of the last main section is devoted to summarizing, basically  from \cite{Sario-Oikawa-1969},
known results on so-called capacity functions. Specifically, we cite without proofs almost complete sets
of  estimates for such quantities, and in addition explain their relations to conformal mappings and extremal problems for analytic functions.


\subsection{Personal background to the paper}

Much of the contents of the present paper has its origin in the research bulletin \cite{Gustafsson-1979}, which was part of the author's
doctoral thesis ``Topics in geometric function theory and related questions of hydrodynamics'' (1981). That title
could as well have been the title also of the present paper, 
geometric function theory as a subject being a junction of complex analysis and potential theory, often in the context of Riemann 
surface theory. The much related paper \cite{Richardson-1980} appeared simultaneously with, but independently of, \cite{Gustafsson-1979}.
Some main sources of inspirations for the paper at hand are the books
\cite{Schiffer-Spencer-1954, Sario-Oikawa-1969, Ahlfors-1973, Cohn-1980}. 
As for physics and fluid dynamics, much of the inspiration comes from the books \cite{Arnold-1978, Arnold-Khesin-1998}.

After the work \cite{Gustafsson-1979} the author did not work much on vortex motion. Two exceptions were the short paper
\cite{Gustafsson-1990a}, inspired by a question by Avner Friedman, and a collaboration with Martin Flucher, which
resulted in the research bulletin \cite{Flucher-Gustafsson-1997}, with most of its contents also appearing in the book \cite{Flucher-1999}.  
There has also been several rewarding contacts with Darren Crowdy and his students, but only after an invitation by Stefanella Boatto to a conference
in Rio de Janeiro in 2012, and reading the inspiring paper \cite{Boatto-Koiller-2013}, more continuous
work on vortex motion was taken on. This resulted in two papers  \cite{Gustafsson-2019, Gustafsson-2022a},
and most recently collaboration with Clodoaldo Grotta-Ragazzo and Jair Koiller in \cite{Grotta-Ragazzo-Gustafsson-Koiller-2024}.

This work has benefitted much from collaborations with (among others)
Vladimir Tkachev, Ahmed Sebbar and Steven Bell. And most recently with Clodoaldo Grotta-Ragazzo and Jair Koiller. Indeed, as for the
latter two, there is a considerable overlap of the present paper with \cite{Grotta-Ragazzo-Gustafsson-Koiller-2024}, although this paper
more relies on complex analytic techniques (in combination with potential theory) and less on pure differential geometry. 

\subsection{Acknowledgements}

The author wishes to thank all the mentioned persons for fruitful collaborations and discussions. And he keeps in strong memory his
supervisor Harold S.~Shapiro (1928 -- 2021) \cite{Gustafsson-2022}, who in the 1970:s suggested three very fruitful research subjects (vortex motion is
one of them) which has kept the author busy for more than 50 years.
  

\section{Some notational conventions}\label{sec:notations}

We generally work within the frame of complex analysis in one dimension, denoting points in the complex plane
by letters such as $z=x+\I y$. Such letters also denote complex coordinates when working on a Riemann surface,
which means that they strictly speaking have a double meaning, however in a way which is customary in the subject.
Ideally, points on a Riemann surface could be denoted like $P, Q,\dots$, and thinking on a Riemann surface as a manifold
covered by patches with complex analytic variables, such as $z, \tilde{z}, w, \dots$, 
the complex coordinate value for a point $P$ would be $z(P), \tilde{z}(P), w(P), \dots$. However we shall not be that formal.

The area measure in the plane is denoted $dxdy$, and often it is convenient to turn to complex coordinates and to interpret
that expression as a wedge product:
$$
dxdy=dx\wedge dy= -\frac{1}{2\I}dz\wedge d\bar{z}= \frac{1}{2\I}d\bar{z}\wedge dz, 
$$ 
often with the $\wedge$ sign omitted. This is useful when writing Green's, Gauss' or Stokes'  formulas. With
differentials written as
$$
df=\frac{\partial f}{\partial x}dx+\frac{\partial f}{\partial y}dy=\frac{\partial f}{\partial z}dz
+\frac{\partial f}{\partial \bar{z}}d\bar{z},
$$ 
$$
d(fdx+gdy)=df\wedge dx +dg\wedge dy, 
$$
these are all are immediate consequences of Stokes' general theorem,
$$
\int_D d\omega =\int_{\partial D}\omega,
$$
for $\omega$ a differential form of a degree which matches $D$. On a Riemann surface with a metric there is an area form (volume form
in principle), and we interpret that as $2$-form, denoted ${\rm vol}$. In general we follow \cite{Frankel-2012} as for
notations in differential geometry. See also \cite{Griffiths-Harris-1978, Krantz-2004}, for more complex analytic oriented texts.

The Wirtinger derivatives are
$$
\frac{\partial}{\partial z}=\frac{1}{2}\big(\frac{\partial}{\partial x}-\I\frac{\partial}{\partial y}\big), \quad
\frac{\partial}{\partial \bar{z}}=\frac{1}{2}\big(\frac{\partial}{\partial x}+\I\frac{\partial}{\partial y}\big),$$
and the Laplace operator 
$$
\Delta=\frac{\partial^2}{\partial x^2}+\frac{\partial^2}{\partial y^2}=4\frac{\partial^2}{\partial z \partial \bar{z}},
$$ 
 in terms of any local variable $z=x+\I y$.
The Dirac measure, or point mass, at a point $a$ is denoted  $\delta_a$. In the first sections this has its usual meaning
and can be identified with the ``delta function'' $\delta(z-a)$ in the plane. However, when working on Riemann surfaces
in Section~\ref{sec:monopole Green} and later we consider it as a ``current'', with the area measure built in. Then it is to be 
identified with what is $\delta(z-a)dxdy$ in the planar case.

Derivatives are usually taken in the distributional sense. In Section~\ref{sec:Bergman} (after (\ref{residue}))
we make an attempt to explain the
meaning of that when computing the fundamental solution of $\partial/\partial\bar{z}$. 


\section{Energy and Green function in the planar case}\label{sec:energy and Green}

\subsection{Coulomb's and Bernoulli's laws}\label{sec:Bernoulli}

Coulomb's law on the force between two point charges, $q_1$ and $q_2$ on a distance $r$
in Euclidean three space,
$$
{\bf F}=-\frac{1}{4\pi \varepsilon_0}\frac{q_1q_2}{r^2}
$$ 
is usually compared with Newton's corresponding law of gravitation, 
$$
{\bf F}=G\,\frac{m_1m_2}{r^2}.
$$
In this paper we shall mainly discuss two-dimensional versions of such laws, in which case the
dependence on distance is like $1/r$.
And as a rather unusual twist we shall derive such a distance law for vortices
from Bernoulli's law in hydrodynamics.

When Newton's  second law of general mechanics, ${\bf F}=m{\bf a}$, 
is formulated for a continuum of particles, a ``fluid'',  it becomes (in absence of viscosity) {\it Euler's law}
of motion, in standard notation ($p=$ pressure, $\rho=$ density, ${\bf v}=$ fluid velocity)
\begin{equation}\label{Euler}
-\nabla p=\rho \Big(\frac{\partial{\bf v}}{\partial t}+({\bf v}\cdot \nabla){\bf v}\Big).
\end{equation}
If the flow moreover is stationary ($\partial{\bf v}/\partial t=0$) that equation can be integrated  to give {\it Bernoulli's law}, 
\begin{equation}\label{Bernoulli}
p+\frac{\rho|{\bf v}|^2}{2}={\rm constant}. 
\end{equation}

Assume now that the fluid is incompressible, $\nabla\cdot {\bf v}=0$,
with density $\rho=1$. In two dimensions, and on identifying vectors with complex numbers in the usual way,
this means that there is a stream function $\psi(z)$, $z=x+\I y$, for which
\begin{equation}\label{wpsi}
{\bf v}=-\I\big(\frac{\partial \psi}{\partial x}+\I \frac{\partial \psi}{\partial y}\big)=-2\I \frac{\partial \psi}{\partial\bar{z}}.
\end{equation}
Sometimes this is written ${\bf v}=\nabla^{\perp}\psi$, as a kind of vector notation,
but we shall usually stay with complex variable notations in the present paper.
The vorticity is the curl of ${\bf v}$, and in our two-dimensional setting this has only one
non-zero component, namely $\omega=-\Delta \psi$.

Away from regions of vorticity $\psi$ is a harmonic function, hence ${\bf v}$ is anti-holomorphic.
We shall still allow some vorticity, but only in the form of isolated point vortices.
This means that $\psi$ has logarithmic singularities and that $\partial \psi/\partial z$ is a meromorphic function.
In order to keep the flow stationary we assume initially that the point vortices are `bound', i.e. do not move.
One may for this purpose think of point vortices as being degenerate boundary components rather than being
part of the fluid region. The analysis will be extended to the case of moving vortices in Section~\ref{sec:moving vortices}.


\begin{figure}
\begin{center}
\includegraphics[scale=0.6]{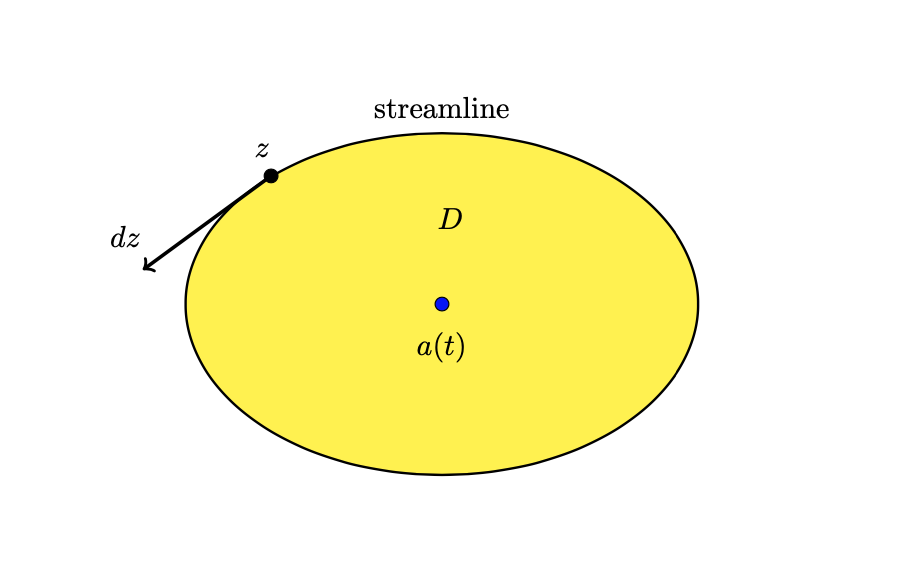}
\end{center}
\caption{Illustration of fluid patch around a point vortex and integration around the boundary.}
\label{fig:vortexD}
\end{figure} 
   

Consider first the case of an isolated point vortex of strength $\Gamma$ at a point $z=a$.
Let $D$ be a small region containing $a$ for which $\partial D$ is a stream line, so that
$d\psi=(\partial \psi/\partial z) dz+ (\partial \psi/\partial \bar{z})d\bar{z}=0$ along $\partial D$
(illustration in Figure~\ref{fig:vortexD}). 
As the function $\partial \psi/\partial z$ is meromorphic
in $D$ with just a simple pole at $a$,  it is easy to compute
the force ${\bf F}_a$ that the surrounding fluid exerts on $D$: using (\ref{Bernoulli}), (\ref{wpsi}) we have
$$
{\bf F}_a=\oint_{\partial D}p\cdot  \I dz
=-2\I\oint_{\partial D} \frac{\partial \psi}{\partial \bar{z}}\frac{\partial \psi}{\partial {z}}dz
$$
$$
=+2\I\oint_{\partial D} \frac{\partial \psi}{\partial \bar{z}}\frac{\partial \psi}{\partial \bar{z}}d\bar{z}
=4\pi\overline{\res_{z=a} \big(\frac{\partial \psi}{\partial {z}}\big)^2d{z}}.
$$
Expanding around $z=a$,
\begin{equation}\label{psi}
\psi(z)= \frac{\Gamma}{2\pi}\re\big(- \log (z-a)+ h_0 (a) + h_1(a)(z-a) +\dots\big),
\end{equation}
where the coefficients $h_0(a), h_1(a),\dots$ represent the influence from other vortices
and boundaries,  we get
\begin{equation}\label{w}
\frac{\partial \psi}{\partial z}=\frac{\Gamma}{4\pi}\big(-\frac{1}{z-a}+h_1(a)+\dots \big),
\end{equation}
\begin{equation}\label{w2}
\big(\frac{\partial \psi}{\partial z}\big)^2=\frac{\Gamma^2}{16\pi^2}\Big(\frac{1}{(z-a)^2}-\frac{2h_1(a)}{z-a}+\text{regular terms} \Big).
\end{equation}
From this we can read off the residue above to obtain
\begin{equation}\label{force}
{\bf F}_a
=-\frac{\Gamma^2}{2\pi}\,\overline{h_1(a)}.
\end{equation}

The above is a general expression for the force on the point vortex when it is bound. Assume now that the fluid region
is the entire complex plane and that there
is only one more vortex, at a point $b$ and of strength $-\Gamma$. Then the complete stream function
is (up to an additive constant)
$$
\psi(z)=\frac{\Gamma}{2\pi}\re\big(-\log  (z-a)+\log (z-b)\big)
$$
$$
=\frac{\Gamma}{2\pi}\re\big(-\log  (z-a)+\log (a-b)+\log(1+\frac{z-a}{a-b})\big)
$$
$$
=\frac{\Gamma}{2\pi}\re\big(-\log  (z-a)+\log (a-b)+\frac{z-a}{a-b}+ \mathcal{O}((z-a)^2)\big).
$$
Hence, in (\ref{psi}), 
$$
h_0(a)= \log |a-b|, \quad h_1(a)=\frac{1}{a-b}.
$$
This gives, with $r=|a-b|$, ${\bf e}_r=-\frac{a-b}{|a-b|}$, the attractive force
\begin{equation}\label{F}
{\bf F}_a=\frac{\Gamma^2}{2\pi}\,\frac{1}{\overline{a-b}}=\frac{\Gamma^2}{2\pi}\,\frac{a-b}{|a-b|^2}
=-\frac{\Gamma^2}{2\pi}\,\frac{{\bf e}_r}{r}
\end{equation}
acting on the vortex at $a$ due to the vortex at $b$. Changing the roles, the force on $b$ is 
similarly ${\bf F}_b=-{\bf F}_a$. Recall also that the vortex strengths are actually $\Gamma_a=\Gamma$
and $\Gamma_b=-\Gamma$. We see that forces between bound vortices
behave in full concordance with forces between charges in two dimensional electrostatics:
equal vortices repel each other, opposite vortices attract each other.

\begin{remark}\label{rem:affine connection}
Under conformal changes of coordinates the coefficient $h_1(a)$ changes as an affine connection,
see Lemma~\ref{lem:transformation} and Remark~\ref{rem:connections}.
It is interesting to compare this fact with views
in fundamental particle physics saying that forces, and particles mediating forces (bosons), 
often have the roles of being connections from mathematical points of view. The connections in the present context
are {\it affine connections}, initially invented as necessary correction terms 
for describing parallel transport and covariant derivatives in the tangent bundle of a manifold. See in general \cite{Frankel-2012}.
See also \cite{Grotta-Ragazzo-Barros-Viglioni-2017} for additional discussion of the force on a vortex, 
plus further differential geometric details. 
\end{remark} 


\subsection{Energy for systems of point vortices, or charges.}

If we consider $z=b$ as a variable point, the force ${\bf F}_b=-{\bf F}_a$ acting on it 
due to the vortex at $a$ represents  a conservative field:
$$
{\bf F}_b(z)=\Gamma_b\nabla V_a(z), \quad V_a(z)=-\frac{\Gamma_a}{2\pi}\log{|z-a|}.
$$ 
The work needed to move the vortex indexed by  $b$ in the field of the vortex at $a$,
from some initial position $w$ to its new position $z$, is
$$
\int_w^z {\bf F}_b\cdot d{\bf r}=-\Gamma_b (V_a(z)-V_a(w)),
$$
with ${\bf F}_b$ above, but evaluated along a curve from $w$ to $z$.

If we next have a whole system of vortices, of strengths $\Gamma_j$ and locations $z_j$ ($1\leq j\leq n$, say), then the work
needed to successively build this up from some given initial positions $w_j$ is
$$
E(z_1,\dots,z_n)=\sum_{k=1}^n\Gamma_k \sum_{j=1}^{k-1}(V_{z_j}(z_k)-V_{z_j}(w_k))
$$  
$$
=\sum_{1\leq j<k\leq n}\Gamma_kV_{z_j}(z_k)+C
=\frac{1}{2\pi}\sum_{j<k}\Gamma_j\Gamma_k \log\frac{1}{|z_j-z_k|}+C.
$$
Here the constant $C$ depends on $w_1,\dots,w_n$:
$$
C=-\frac{1}{2\pi}\sum_{j<k}\Gamma_j\Gamma_k \log\frac{1}{|w_j-w_k|}.
$$
We now define the {\it energy} of the point vortex system by simply ignoring 
the fairly arbitrary constant $C$:
\begin{equation}\label{energy}
E(z_1,\dots, z_n)
=\frac{1}{2\pi}\sum_{j<k}\Gamma_k V_{z_j}(z_k)
=\frac{1}{4\pi}\sum_{j\ne k}\Gamma_k \Gamma_j \log \frac{1}{|z_k-z_j|}.
\end{equation}

\begin{remark}\label{rem:2pi}
The factor $1/2\pi$ which appears in many formulas (as above), is natural from a physical point of view,
but it is often absent in mathematical texts on potential theory, such as
\cite{Sario-Oikawa-1969, Landkof-1972, Ahlfors-1973, Doob-1984, Ransford-1995, Saff-Totik-1997, Armitage-Gardiner-2001, Helms-2014}.
One has to keep this in mind when comparing formulas from different sources.
There is also an additional factor $\frac{1}{2}$ in expressions for energy which shows up if we sum over all pairs $(j,k)$, $j\ne k$,
as in the last term in (\ref{energy}). This factor is common in physics texts.
\end{remark}


\subsection{Moving vortices}\label{sec:moving vortices}

If vortices are not bound, but move, then the flow is no longer stationary and all quantities depend on time. 
We shall introduce two complex-valued functions, $f$ and $h$, for which Euler's equation takes the
simple formulation (\ref{dfh}) below. The function $f$, which was also used in \cite{Gustafsson-2019},
takes care of the dynamics in the sense that its imaginary part is (minus) the time derivative
of the stream function (compare (\ref{Phi}) below), while the function $h$ essentially is the vorticity times the fluid velocity. Thus $h$
reminds of what is called {\it helicity} in three dimensions, see \cite{Arnold-Khesin-1998}. 
\begin{definition}
For $z$ in the fluid region, which may be the entire complex plane or a smaller region, we define
\begin{equation}\label{f}
f(z,t)=\frac{1}{2}|{\bf v}|^2+p-\I\frac{\partial \psi}{\partial t}
=2\frac{\partial \psi}{\partial z}\frac{\partial \psi}{\partial \bar{z}}+p-\I \frac{\partial \psi}{\partial t},
\end{equation}
\begin{equation}\label{h}
h(z,t)=\frac{1}{2}(\Delta \psi) \nabla \psi 
=4\frac{\partial^2 \psi}{\partial z\partial \bar{z}}\frac{\partial \psi}{\partial \bar{z}}
=2\frac{\partial }{\partial z}\Big(\frac{\partial \psi}{\partial \bar{z}}\Big)^2,
\end{equation}
where $\nabla\psi$ in the latter equation shall be considered as a complex number.
Since $\psi$ is determined only up to a constant, which may depend on time,  $f$ contains a free
additive imaginary constant.
\end{definition}

The significance of these functions will be explained shortly. First we restate Euler's equation (\ref{Euler})
in terms of $f$ and $h$ in the case that all data are smooth.

\begin{proposition}\label{prop:dfh}
In the case of smooth vorticity distribution, Euler's equation takes the simple form
\begin{equation}\label{dfh}
\frac{\partial f}{\partial\bar{z}}=h.
\end{equation}
An equivalent statement is that
\begin{equation}\label{weak Euler}
\int_{\partial D}\Big( fdz+2\big(\frac{\partial \psi}{\partial \bar{z}}\big)^2d\bar{z}\Big)=0
\end{equation}
holds for any subregion $D$ of the fluid domain. 
\end{proposition}
Equation (\ref{dfh}) shows that $f$ is {analytic in regions free of vorticity}. In the stationary case, i.e. when $\partial{\psi}/\partial t=0$,
$f$ is real-valued and therefore necessarily constant in such regions. This gives Bernoulli's equation (\ref{Bernoulli}) again.
The equation (\ref{weak Euler}) will be useful when passing to the point vortex limit.

\begin{proof}
Euler's equation (\ref{Euler}) with $\rho=1$ becomes in complex variable notations, 
and in terms of the stream function $\psi$ (see (\ref{wpsi})), 
$$
-2\frac{\partial p}{\partial \bar{z}}=-2\I \frac{\partial^2\psi}{\partial t\partial \bar{z}}
+\re\Big(-2\I\frac{\partial\psi}{\partial \bar{z}}\cdot\overline{2\frac{\partial}{\partial \bar{z}}}\Big)(-2\I\frac{\partial \psi}{\partial \bar{z}}).
$$
By straightforward computations, using 
$$
f=2\frac{\partial \psi}{\partial z}\frac{\partial \psi}{\partial \bar{z}}+p-\I \frac{\partial \psi}{\partial t},
\quad h=4\frac{\partial^2 \psi}{\partial z\partial \bar{z}}\frac{\partial \psi}{\partial \bar{z}},
$$
as in (\ref{f}), (\ref{h}), this gives (\ref{dfh}) right away. 

Next, integrating $h$ over some flow region $D$ 
and using Stokes' formula in complex variable form gives
$$
\int_D h dx dy=\I\int_D \frac{\partial }{\partial z}\Big(\frac{\partial \psi}{\partial \bar{z}}\Big)^2 dz\wedge d\bar{z}
=\I\int_{\partial D}\Big(\frac{\partial \psi}{\partial \bar{z}}\Big)^2d\bar{z}.
$$
On the other hand, again by Stokes,
$$ 
\int_D \frac{\partial f}{\partial \bar{z}} dx dy=\frac{1}{2\I}\int_D\frac{\partial f}{\partial \bar{z}}d\bar{z}\wedge dz
=\frac{1}{2\I}\int_Dd(fdz)=\frac{1}{2\I}\int_{\partial D}fdz.
$$
Now (\ref{weak Euler}) follows.
\end{proof}

Proposition~\ref{prop:dfh} gives a handy way for computing the speed of an isolated point vortex
without introducing artificial regularizations. One only has to accept that the dynamics (\ref{weak Euler}),
derived in the case of smooth data, remains valid in the point vortex limit.

So let the flow have  a point vortex, of strength $\Gamma$ at a point $a=a(t)$, and let $D$
be a small neighborhood of $a$ containing no further vorticity. The expansion (\ref{psi}) then gives
\begin{equation}\label{dpsidt} 
\frac{\partial\psi(z,t)}{\partial t}
=\re\frac{\Gamma}{2\pi} \Big(\frac{d{a}/dt}{z-a}+ {\rm regular\ terms}\Big).
\end{equation}
Since $\partial{\psi}/\partial t$ is minus the imaginary part of $f$, and $f$ is analytic in $D\setminus\{a\}$,
it follows that
$$
f(z,t)=\frac{\Gamma}{2\pi\I} \Big(\frac{d{a}/dt}{z-a}+ {\rm regular\ terms}\Big).
$$
Hence
\begin{equation}\label{intfdz}
\int_{\partial D}fdz=\Gamma\,\frac{d a}{dt}.
\end{equation}
On the other hand, the expansion (\ref{w2}) shows that
\begin{equation}\label{intdpsisquare}
2\int_{\partial D}\big(\frac{\partial \psi}{\partial \bar{z}}\big)^2d\bar{z}
=-\frac{\Gamma^2}{2\pi\I}\,\overline{h_1(a)}.
\end{equation}
Thus (\ref{weak Euler}) gives the following well-known dynamics.
\begin{proposition}
An isolated free vortex $a=a(t)$ in a fluid governed by Euler's equation
moves according to 
\begin{equation}\label{dynamics}
 \frac{da}{d t}=\frac{\Gamma}{2\pi\I}\,\overline{h_1(a)}.
\end{equation}
\end{proposition}

This equation can be generalized to the case of a  vortex which is not free, but is subject to an external force.
To this purpose we make a more detailed analysis based on taking $\partial D$ to be an 
instantaneous stream line. This means that  
$d\psi=0$ along $\partial D$ at a given moment of time. Recalling (\ref{f}) we can then rewrite the first term in (\ref{weak Euler}) as
$$
\int_{\partial D}fdz
=\int_{\partial D}2\frac{\partial \psi}{\partial {z}}
\frac{\partial \psi}{\partial \bar{z}}dz-2
\Big(\frac{\partial \psi}{\partial z}dz+\frac{\partial \psi}{\partial \bar{z}}d\bar{z}\Big)\frac{\partial \psi}{\partial \bar{z}}
+p dz-\I\frac{\partial \psi}{\partial t}dz
$$
$$
=-\int_{\partial D}2\big(\frac{\partial \psi}{\partial \bar{z}}\big)^2d\bar{z}+\int_{\partial D}(p-\I\frac{\partial \psi}{\partial t})dz.
$$
Taking (\ref{intdpsisquare}) into account this gives
\begin{equation}\label{three terms}
\int_{\partial D}fdz=\frac{\Gamma^2}{2\pi\I}\overline{h_1(a)}+\int_{\partial D}(p-\I\frac{\partial \psi}{\partial t})dz.
\end{equation} 

For the final term in (\ref{three terms}) one has
\begin{equation}\label{dotpsi}
\int_{\partial D}\frac{\partial\psi}{\partial t}dz=0,
\end{equation}
as a consequence of $\partial D$ being a stream line. This can be easily proved in the case of 
single vortex in a bounded domain. In that case one can simply take $D$ to be the entire domain, and since 
then $\psi$ is constant on each component of $\partial D$ we have
$$
\int_{\partial D}\psi dz=0.
$$
Taking the time derivative then gives (\ref{dotpsi}). 

It follows that
\begin{equation}\label{ppsiF}
\int_{\partial D} (p -\I\frac{\partial \psi}{\partial t}) dz 
=-\I \int_{\partial D}p\cdot \I dz=-\I {\bf F}_a,
\end{equation}
where ${\bf F}_a$ is the force on $\partial D$ exerted by the surrounding fluid, see Section~\ref{sec:Bernoulli}.
This force is zero in the case of a free vortex, but we are now assuming that the vortex is subject to an externally
prescribed motion. In this case the equations (\ref{intfdz}), (\ref{three terms}), (\ref{ppsiF}) give the 
following combined dynamics, as a common generalization of (\ref{force}) and (\ref{dynamics}).
Compare more detailed investigations in \cite{Grotta-Koiller-Oliva-1994} for vortices with mass.

\begin{proposition}\label{prop:combined dynmaics}
If an isolated vortex of strength $\Gamma$ is subject to an external force ${\bf F}_a$, then the dynamics is
\begin{equation}\label{Gammadadt}
\Gamma\frac{da}{dt}=\frac{\Gamma^2}{2\pi\I}\overline{h_1(a)}-\I{\bf F}_a. 
\end{equation}
\end{proposition}
  
\begin{remark}\label{rem:Phi}
In a flow with only isolated point vortices there is a flow potential $\varphi$ and a complex potential $\Phi$,
both of them multi-valued, so that
\begin{equation}\label{Phipsi}
\Phi=\varphi +\I \psi
\end{equation}
is (multi-valued) analytic (see \cite{Marchioro-Pulvirenti-1994}).
Near a point vortex of strength $\Gamma$ we then have, by (\ref{wpsi}),
$$
\Phi(z)=\frac{\Gamma}{2\pi\I}\log (z-a)+\text{regular}.
$$
On comparison with $f(z)$ in (\ref{f}), which also is analytic, it follows that
\begin{equation}\label{Phi}
f=-\frac{\partial \Phi}{\partial t}.
\end{equation}
This gives the interpretation, in the point vortex case, of the function $f$ as being minus the time derivative of the complex potential.
\end{remark}


\subsection{The transfinite diameter.}\label{sec:transfinite diameter}

In function theoretic contexts it is of interest to take the vortex strengths (or charges, in an electrostatic picture) to be
$$
\Gamma_1=\dots =\Gamma_n=\frac{1}{n},
$$  
sometimes extended with an additional vortex of strength $\Gamma=-1$. If such an additional vortex
is not specified explicitly it will appear automatically at the point of infinity on the Riemann sphere.
On confining $z_1,\dots, z_n$ to belong to a given compact set $K\subset \C$, the geometric quantity
\begin{align}\label{deltan}
\delta_n(K)&=\max_{z_1,\dots,z_n\in K}\exp\big(-4\pi E(z_1,\dots, z_n)\big)\\
&=\max_{z_1,\dots,z_n\in K}\Big(\Pi_{j<k}|z_j-z_k|\Big)^{2/n^2}
\end{align}
(recall (\ref{energy})) leads in the limit $n\to\infty$ to the concept of transfinite diameter:
\begin{definition}
The {\it transfinite diameter} of a compact set  $K$ is
\begin{equation}\label{delta}
\delta(K)=\lim_{n\to \infty}\delta_n(K)=\inf_n \delta_n(K).
\end{equation}
\end{definition}
 
By (\ref{F}) equal vortices repel each other, like charges. With the confinement to the
compact set $K$ it is actually more natural to think of charges rather than of vortices, and
then consider $K$ to be a perfect conductor in which charges move without resistance. 
Then the minimization of the energy implicit in (\ref{deltan}), (\ref{delta}) means that the charges will rush to the boundary
of $K$ and there represent the {\it equilibrium distribution}. This limiting distribution represents a 
{\it probability measure} $\varepsilon_K$ on $K$ (actually on $\partial K$) which is of substantial potential theoretic significance.
The field it produces, via the Newton type law in (\ref{F}), vanishes in the interior of $K$, and the corresponding
potential is constant on $K$. 
Indeed, if it were not constant a smaller minimum could be reached by moving charges to places where the potential is smaller.

We see from (\ref{deltan}), (\ref{delta}) that the limiting energy $E_K=E(\varepsilon_K)$
of the equilibrium distribution relates to the transfinite diameter as
\begin{equation}\label{deltaEK}
\delta (K)=e^{-4\pi E_K}.
\end{equation}

If we move the point $a=\infty$, implicit in the above discussion, to become a finite point $a\in\C\setminus K$,
then the energy of a condensor consisting of  $K_0=\{a\}$ and $K_1=K$ involves, before passing to the limit, the expression 
\begin{equation*}
E(a; z_1,\dots, z_n)= \frac{1}{2\pi}\Big({\sum_{1\leq j<k\leq n}} \frac{1}{n^2}\log \frac{1}{|z_j-z_k|}
- \sum_{j=1}^n \frac{1}{n}\log \frac{1}{|z_j-a|}\Big)
\end{equation*}
\begin{equation}\label{Eloglog}
=\frac{1}{4\pi n^2}\Big(\sum_{j,k=1,j\ne k}^n \log \frac{|z_j-a||z_k-a|}{|z_j-z_k|}+\sum_{j=1}^n\log |z_j-a|^2\Big).
\end{equation}
This accounts for a single charge of strength $-1$ on $K_0$ and $n$ movable charges of strengths $+1/n$ on $K_1$. 
There is a corresponding transfinite diameter,
$$
\delta(K,a)=\lim_{n\to\infty}\max_{z_1,\dots,z_n\in K}\exp\big(-4\pi E(a;z_1,\dots,z_n)\big),
$$
again related to the limiting energy $E_{K,a}=E(\varepsilon_{K,a})$ by
\begin{equation}\label{deltaKaEKa}
\delta(K,a)=e^{-4\pi E_{K,a}}.
\end{equation}

As $n\to \infty$ the second term in (\ref{Eloglog}) becomes negligible compared to the corresponding part of the first term.    
And in both cases, $\delta(K)$ and $\delta(K,a)$, it is natural to adjust the dependence on $n$ (before passing to the
limit and without affecting the limit itself) to arrive at the following standard definitions.

\begin{definition}\label{def:transfinite diameter}
The {\it transfinite diameter} of a compact set $K\subset\C$, with respect to the point of infinity, respectively
a finite point $a\in\C\setminus K$, is 
$$
\delta(K)=\delta(K,\infty)=\lim_{n\to \infty}\max_{z_1,\dots,z_n\in K}\Big(\Pi_{j<k}|z_j-z_k|\Big)^{2/n(n-1)},
$$
$$
\delta(K,a)=\lim_{n\to\infty}\max_{z_1,\dots,z_n\in K}\Big(\Pi_{j<k}\frac{|z_j-z_k|}{|z_j-a||z_k-a|}\Big)^{2/n(n-1)}.
$$
\end{definition}

\begin{remark}
The close connection between energy and transfinite diameter depends on
the nice feature of two dimensional potential theory that when potentials are added
the results collect into elementary analytic expressions, essentially logarithms of rational functions. 
On the most basic level we have, 
$$
\log|z-a|-\log|z-b|=\log|\frac{z-a}{z-b}|.
$$
In general terms, this is a consequence of harmonic functions in two dimensions being just real parts of analytic functions,
and for analytic functions all kinds of algebraic operations apply. {Function theoretic aspects} of two dimensional potential theory are much 
discussed in \cite{Tsuji-1975, Ransford-1995}. The transfinite diameter was first introduced by {Fekete} \cite{Fekete-1923} in a context of 
{\it polynomial approximation}, a subject for which potential theory plays an important role.

See Figures~\ref{fig:Fekete1}, \ref{fig:Fekete2} for illustrations of {\it Fekete points}, finite approximations of configurations for the transfinite diameter 
and equilibrium distribution. 
\end{remark}


\begin{figure}
\begin{center}
\includegraphics[scale=0.7]{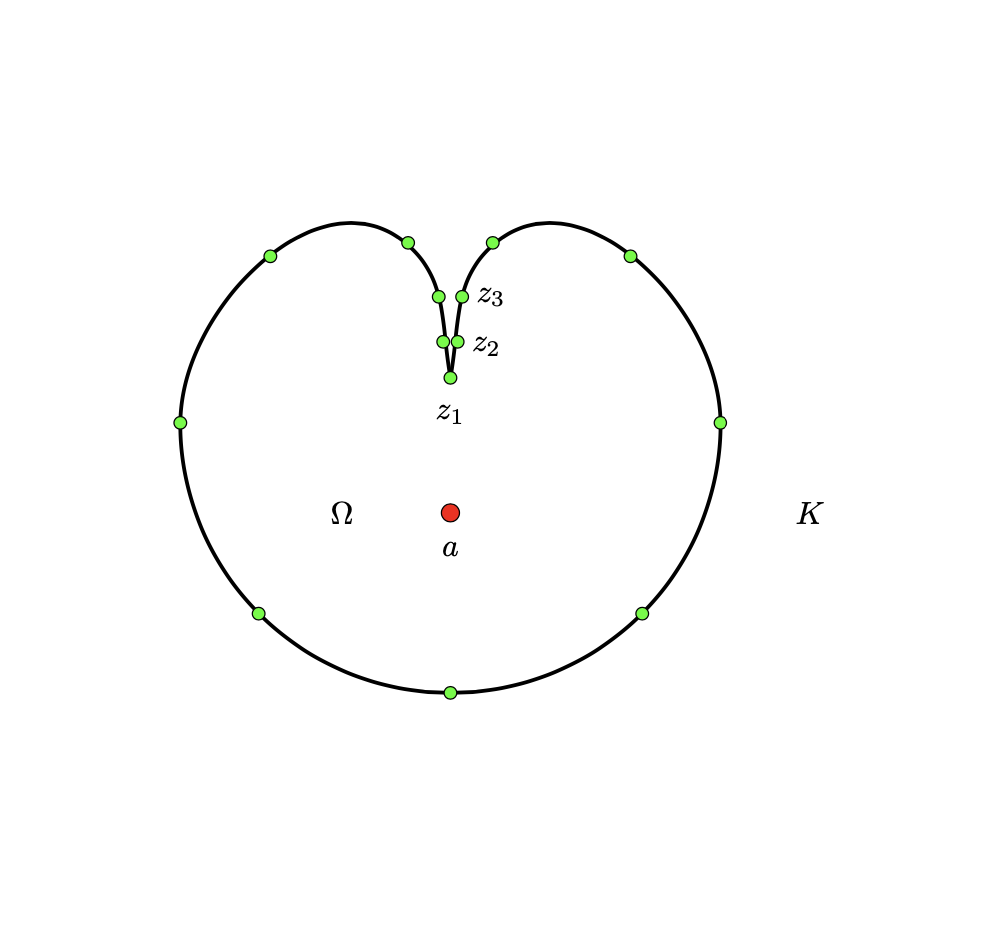}
\end{center}
\caption{Fekete points $z_1,z_2,z_3,\dots$ clustering near an inward cusp relative to $\Omega$. 
In the limit they represent the equilibrium distribution which minimizes the energy $E_{K,a}$. Brownian motion started
at $a$ has a high probability of first reaching the boundary close to the cusp.}
\label{fig:Fekete1}
\end{figure}


\begin{figure}
\begin{center}
\includegraphics[scale=0.6]{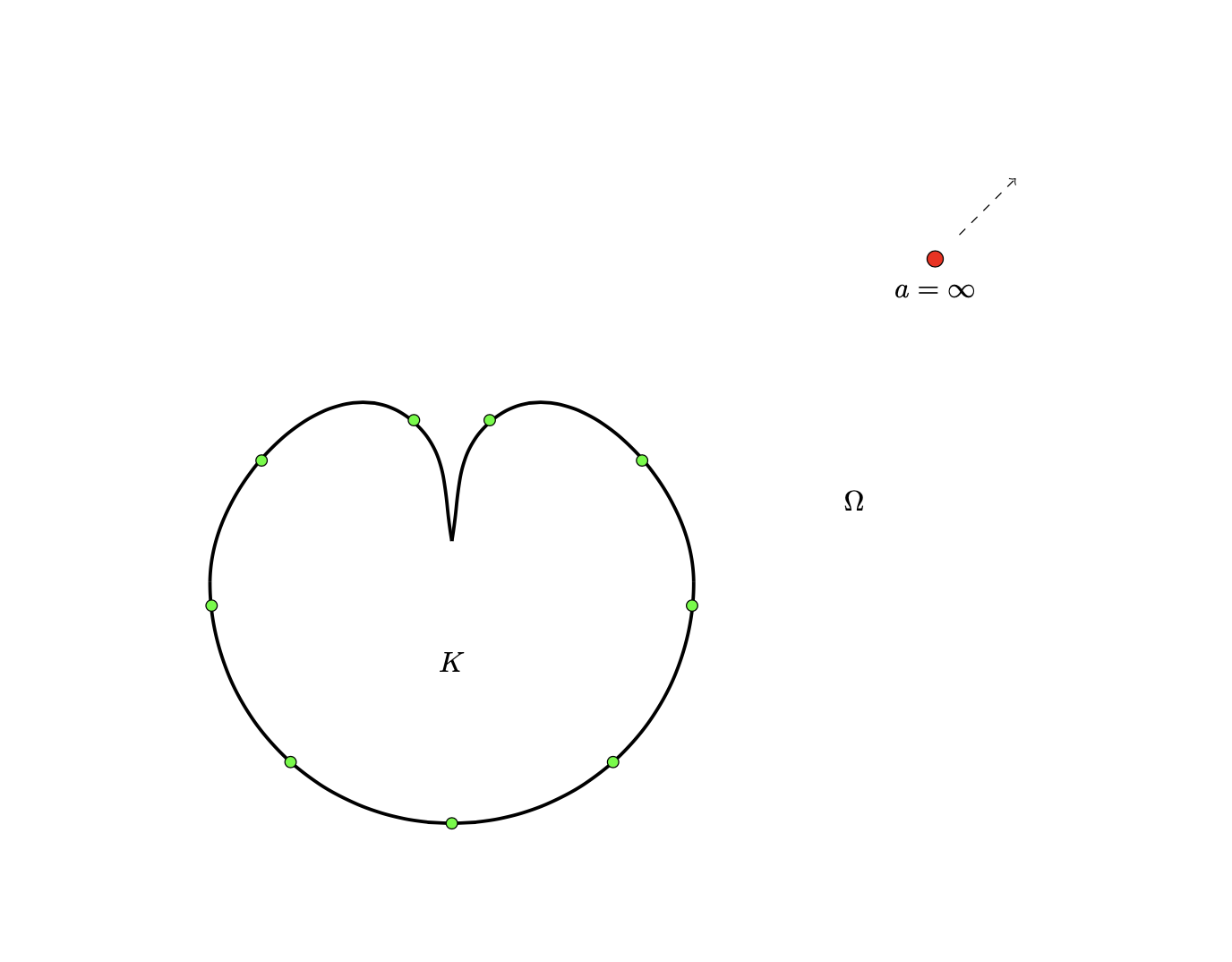} 
\end{center}
\caption{Fekete points in reverse geometry, where cusp goes into $K$.}
\label{fig:Fekete2}
\end{figure}



\subsection{Equilibrium potential and Robin's constant.}\label{sec:equilibrium potential}

Both in the vortex picture and in the electrostatic picture, with $K$ a compact set as above, 
it is really in the external region $\Omega=\C\setminus K$
that the flow and electric fields live and are active. In $K$ the potentials stay at constant values
(possibly different constants on different components).

The potential of the equilibrium measure
$\varepsilon_K$, namely the  {\it equilibrium potential} 
\begin{equation}\label{epsilonK}
V_K(z)=\frac{1}{2\pi} \int_K \log \frac{1}{|z-w|} d\varepsilon_K(w),
\end{equation}
is harmonic in $\Omega$ and has the asymptotic behavior
$$
V_K(z)=-\frac{1}{2\pi}\log |z|+\mathcal{O}(|z|^{-1}) \quad \text{as}\,\,|z|\to\infty.
$$
Note the absence of a constant term. What could have been a constant term pops up instead in the boundary 
behavior, and constant value on $K$, essentially the Robin constant:
\begin{definition}\label{def:robin}
The {\it Robin constant} $\gamma=\gamma_K$ for the compact set $K$ is given by
\begin{equation}\label{VK}
V_K(z)= {\rm constant}=\frac{\gamma_K}{2\pi} \quad \text{on}\,\, K.
\end{equation}
\end{definition}

The function 
\begin{equation}\label{external Green}
G(z,\infty)=\frac{\gamma_K}{2\pi}-V_K(z)=\frac{1}{2\pi}\Big(\log |z|+\gamma_K+\mathcal{O}(|z|^{-1}\Big) \quad \text{as}\,\,|z|\to\infty
\end{equation}
then is harmonic in $\C\setminus K$ and has zero boundary values on $\partial K$. In view of its behavior at infinity
it is exactly the {Green function} of $\Omega= (\C\cup\{\infty\})\setminus K$ with pole at infinity. 
In fact, changing the coordinate to $w=1/z$, the behavior  becomes (keeping $G$, as for notation)
$$
{G}(w,0)= -\frac{1}{2\pi}\log |w|+\text{harmonic}, \quad G(w,0)=0 \,\, \text{on }\,\partial\Omega, 
$$
which is exactly what is required by a Green function, see Section~\ref{sec:Green and Robin} below.
By the maximum principle, $G\geq 0$ in $\Omega$, so the equilibrium potential satisfies
$V_K\leq \gamma_K/2\pi$ everywhere, $V_K=\gamma_K/2\pi$ on $K$.

The Robin constant $\gamma_K$ is related to the transfinite diameter as follows.
\begin{lemma}\label{lem:gammadelta}
$$
e^{-\gamma_K}=\delta(K)=\lim_{n\to \infty}\max_{z_1,\dots,z_n\in K}\Big(\Pi_{j<k}|z_j-z_k|\Big)^{2/n(n-1)}.
$$ 
\end{lemma}
\begin{proof}
The second equality is just the definition of the transfinite diameter, which is related to energy 
via (\ref{deltaEK}). Thus the first equality relates the energy $E_K$ of the equilibrium distribution $\varepsilon_K$
to the level $\gamma_K$ of  the equilibrium potential $V_K$, 
in fact it simply says that 
\begin{equation}\label{EKgammaK}
4\pi E_K=\gamma_K.
\end{equation}
That equality is actually far from trivial,
it is the main result in the much celebrated doctoral thesis \cite{Frostman-1935} of Otto Frostman,
in which the ideas of Gauss were made mathematically rigorous in a general setting. See \cite{Doob-1984}
for historical comments and complete mathematical details.
Some other sources are \cite{Sario-Oikawa-1969, Ahlfors-1973, Ransford-1995}.   
\end{proof}

However the idea for the mentioned equality is simple, and in fact implicit in our previous discussion.
First we introduce the notion of potential of a measure $\mu$.
\begin{definition} 
If $\mu$ is a, possibly signed, measure with compact support in the plane its  {\it (logarithmic) potential} is 
\begin{equation*}
U^\mu(z)=\frac{1}{2\pi}\int \log \frac{1}{|z-w|}d\mu(w) \quad (z\in \C).
\end{equation*} 
This means that that $-\Delta U^\mu=\mu$ in the sense of distributions.
\end{definition}
 
Now, to informally finish the previous proof,  we pass 
from the discrete case to the limiting equilibrium distribution $\varepsilon_K$, 
which is a probability measure the energy of which is obtained from (\ref{energy})
as
$$
E_K=\frac{1}{4\pi}\int_K\int_K \log\frac{1}{|z-w|}d\varepsilon_K(z)d\varepsilon_K(w)
=\frac{1}{2}\int_K U^{\varepsilon_K}d\varepsilon_K.
$$ 
Using (\ref{VK}) this becomes
\begin{equation}\label{EK}
E_K=\frac{1}{2}\int_K V_K d\varepsilon_K=\frac{\gamma_K}{4\pi},
\end{equation}
which is the desired equality (\ref{EKgammaK}).


\subsection{Green and Robin functions}\label{sec:Green and Robin}

Now we reverse the geometry, or at least allow it to be more general.
Let $\Omega$ be any domain in the Riemann sphere having a nontrivial (in terms of capacity)
complement $K$, and let $a\in\Omega$. The most straight-forward definition of the Green function
of $\Omega$ is the following. 
\begin{definition}\label{def:Green}
The {\it Green function} $G(\cdot,a)=G_\Omega (\cdot,a)$ 
with respect to $a$ is the unique solution of the Dirichlet problem
\begin{equation}\label{DeltaG}
-\Delta G (\cdot,a)=\delta_{a} \quad \text{in }\Omega,
\end{equation}
\begin{equation}\label{G0}
G(\cdot,a)=0 \quad \text{on }\partial\Omega.
\end{equation}
\end{definition}

In the case $a=\infty$ the above definition agrees with the one given in (\ref{external Green}).
It is a PDE definition which seemingly requires  $\partial\Omega$ to be reasonably regular.
But this can be relaxed (see below). 
Later, in Section~\ref{sec:electro-hydro}, the above Green function will be identified as the {\it electrostatic} Green function, 
to be distinguished
from a {\it hydrodynamic} Green function which, as for boundary conditions, is more adapted to problems in fluid dynamics. 

In potential theory one often prefers definitions that avoid
any regularity assumptions whatsoever, and on writing
\begin{equation}\label{GlogH}
G(z,a)=\frac{1}{2\pi}\big(- \log {|z-a|}+H(z,a)\big)
\end{equation}
one may to this purpose directly define $H(\cdot,a)$ as being the smallest superharmonic function 
(meaning that $-\Delta H(\cdot,a)\geq 0$) in $\Omega$ which satisfies $H(z,a)\geq \log|z-a|$ for $z\in\Omega$. 
See \cite{Doob-1984}. Or the other way around,  from below as a suitable supremum of subharmonic functions,
see \cite{Ahlfors-1973, Farkas-Kra-1992} for details.
With a definition of the latter type one can show that every {\it simply connected Riemann surface}, except the Riemann sphere
and the entire complex plane, admits a Green function with (\ref{G0}) holding in a certain weak sense. 
The exponential of (minus) the analytic completion
of that Green function, namely $f(z)=\exp{(-G(z,a)-\I G^*(z,a))}$, then defines a 
conformal map of the Riemann surface onto the unit disk, thus proving the famous {\it uniformization theorem} of
Koebe \cite{Koebe-1918}. 

Returning to (\ref{DeltaG}), (\ref{G0}) one may introduce an everywhere defined potential $V_{K,a}$ by setting
\begin{equation}\label{VGK}
V_{K,a}=
\begin{cases}
-G_\Omega (\cdot, a) \quad&\text{in }\Omega,\\
\quad 0 \quad &\text{on } K=(\C\cup\{\infty\})\setminus \Omega.
\end{cases}
\end{equation}
Compare  (\ref{external Green}).
The Laplacian of $V_{K,a}$ then gets a distributional contribution on $\partial \Omega$,
which is the equilibrium measure, or equilibrium distribution, in the present context:
\begin{definition}
The {\it equilibrium measure} $\varepsilon_{K,a}$ for the complement $K$ of a general domain $\Omega$ in the
Riemann sphere, and with respect to a point $a\in\Omega$, is the probability measure on $\partial \Omega$ arising via
\begin{equation}\label{DeltaVK}
-\Delta V_{K,a}= \varepsilon_{K,a}-\delta_a. 
\end{equation}
\end{definition}
Since $V_{K,a}$ vanishes on $K$ this equality says that $\varepsilon_{K,a}$ is the measure  obtained by balayage (``sweeping'')
of $\delta_a$ onto $\partial \Omega$ ($=\partial K$ if minor regularity is assumed).  See further Section~\ref{sec:balayage}. 
Thus we have 
\begin{equation}\label{VUU}
V_{K,a}= U^{\varepsilon_{K,a}}-U^{\delta_a}
\end{equation}
everywhere, and  $V_{K,a}=-G(\cdot,a)$ in $\Omega$.

Keeping the notations appearing in (\ref{psi}) we expand the regular part $H(z,a)$ (see (\ref{GlogH})) of the Green function
in a Taylor series as
$$
H(z,a)=h_0(a)+\frac{1}{2}h_1(a)(z-a)+\frac{1}{2}\overline{h_1(a)}(\bar{z}-\bar{a})+\mathcal{O}(|z-a|^2)
$$
\begin{equation}\label{TaylorH}
=\re \big(h_0(a)+h_1(a)(z-a)+\mathcal{O}((z-a)^2)\big).
\end{equation}
Clearly $h_0(a)=H(a,a)$, and in addition $h_1(a)=2\frac{\partial}{\partial z}|_{z=a}H(z,a)$.
We then have the following relations.
\begin{lemma}\label{lem:hgammadelta}
\begin{align*}\label{eh0delta}
e^{-h_0(a)}&=\delta(K,a),\\
 h_1(a)&=\frac{\partial h_0(a)}{\partial a}.
\end{align*}
\end{lemma}

The first equation says that $h_0(a)$ plays the same role, for a finite $a\in\Omega$, 
as that which $\gamma_K$ plays in Lemma~\ref{lem:gammadelta} when $a=\infty$. Thus we may set
$$
\gamma_K(a)=h_0(a).
$$
This quantity is the {\it Robin function} for $\Omega$. We shall however keep the complementary set
$K$ as for the notation. A more precise terminology would be to say that $\gamma_K(a)$ is a ``coordinate Robin function''.
Indeed, in a Riemann surface setting it depends on the choice of coordinates in such a way that 
\begin{equation}\label{dsegamma}
ds=e^{-\gamma_K(a)}|da|
\end{equation}
becomes an invariantly defined metric on $\Omega$. See more precisely Lemma~\ref{lem:transformation} below,
and Section~\ref{sec:Robin} for the natural extension to Riemann surfaces.

We now turn to the proof of the lemma.

\begin{proof}
In view of (\ref{deltaKaEKa}) the first equality amounts to the statement $h_0(a)=4\pi E_{K,a}$.
The proof of this is similar to that of Lemma~\ref{lem:gammadelta}, but we have to resort on the more complicated
expression (\ref{Eloglog}) for the energy. 
In that expression the last term can be discarded (like in Definition~\ref{def:transfinite diameter}).

First note that
$$
G(z,a)=U^{\delta_a}(z)-U^{\varepsilon_{K,a}}(z)=\frac{1}{2\pi}\log \frac{1}{|z-a|}-U^{\varepsilon_{K,a}}(z).
$$
This shows that the regular part $H(z,a)$ of the Green function, 
\begin{equation}\label{HU}
H(z,a)=-2\pi U^{\varepsilon_{K,a}}(z),
\end{equation}
is essentially the potential of the equilibrium distribution for a point charge at the point $a$. In particular,
$$
h_0(a)=-2\pi U^{\varepsilon_{K,a}}(a).
$$
Moreover, from Fubini we have the general symmetry $\int U^\mu d\nu=\int U^\nu d\mu$, by which
$$
\int_K U^{\delta_a}d\varepsilon_{K,a}
=U^{\varepsilon_{K,a}}(a).
$$ 

Starting from the continuous, and limiting, version of (\ref{Eloglog}) we now get
$$
E_{K,a}=\frac{1}{4\pi}\int_K\int_K \log \frac{|z-a||w-a|}{|z-w|}d\varepsilon_{K,a}(z)d\varepsilon_{K,a}(w)
$$
$$
=\frac{1}{4\pi}\int_K\int_K \log \frac{1}{|z-w|}d\varepsilon_{K,a}(z)d\varepsilon_{K,a}(w)-2\cdot \frac{1}{2} \int_KU^{\varepsilon_{K,a}}(a)d\varepsilon_{K,a}(w)
$$
$$
=\frac{1}{4\pi}\int_K\int_K \log \frac{1}{|z-w|}d\varepsilon_{K,a}(z)d\varepsilon_{K,a}(w)-U^{\varepsilon_{K,a}}(a)
$$
$$
=\frac{1}{2}\int_K U^{\varepsilon_{K,a}}d\varepsilon_{K,a}-U^{\varepsilon_{K,a}}(a)
=\frac{1}{2}\int_K (V_{K,a}+ U^{\delta_a})d\varepsilon_{K,a}-U^{\varepsilon_{K,a}}(a)
$$
$$
=\frac{1}{2}\int_K U^{\delta_a}d\varepsilon_{K,a}-U^{\varepsilon_{K,a}}(a)=\frac{1}{2}U^{\varepsilon_{K,a}}(a)-U^{\varepsilon_{K,a}}(a)
$$
$$
=-\frac{1}{2}U^{\varepsilon_{K,a}}(a)=\frac{1}{4\pi}h_0(a).
$$ 
Referring to the first line of the proof, this is exactly what we wanted to prove.

The second assertion of the lemma is immediate from the symmetry of $H(z,a)$:
\begin{equation}\label{h0H}
h_1(a)=2\frac{\partial}{\partial z}\big|_{z=a}H(z,a)=\frac{\partial}{\partial z}\big|_{z=a}H(z,z)=\frac{\partial}{\partial a}h_0(a).
\end{equation}
\end{proof}

For later use we observe that taking the $\partial/\partial \bar{a}$ derivative of (\ref{h0H}) gives
\begin{equation}\label{h00H}
\frac{\partial^2}{\partial\bar{a}\partial a}h_0(a)=2\{\frac{\partial^2}{\partial z\partial \bar{a}}H(z,a)\}\big|_{z=a}.
\end{equation}


\subsection{Behavior under conformal maps}\label{sec:conformal}

As already mentioned, see (\ref{dsegamma}), $ds=e^{-h_0(z)}|dz|$ is naturally considered as a metric on $\Omega$,
in fact a {\it conformally invariant} metric (this expressed as (\ref{ehzehw}) below). 
See in general \cite{Ahlfors-1973, Krantz-2004} for this concept.
\begin{definition}\label{def:Gaussioan curvature}
The {\it Gaussian curvature} of any metric $ds =\lambda(z)|dz|$ is
\begin{equation}\label{kappageneral}
\kappa=-\frac{\Delta \log \lambda}{\lambda^2}.
\end{equation}
When the metric is written on the form $ds=e^{-\gamma(z)}|dz|$ this becomes
\begin{equation}\label{kappa}
\kappa=e^{2 \gamma}\Delta \gamma.
\end{equation}
\end{definition}

The following proposition expresses some instances of conformal invariance in our context.
\begin{proposition}\label{lem:transformation}
Let $f:\tilde{\Omega}\to\Omega$ be a conformal map. Denoting by $\tilde{z}$ the complex variable in $\tilde{\Omega}$, by 
$z$ that in $\Omega$, and with a tilde on quantities referring to $\tilde{\Omega}$, we have
\begin{equation}\label{h0h0log}
h_0(z)=\tilde{h}_0(\tilde{z})+\log |f^\prime(\tilde{z})|, 
\end{equation}
equivalently
\begin{equation}\label{ehzehw}
e^{-h_0(z)}|dz|=e^{-\tilde{h}_0(\tilde{z})}|d\tilde{z}|.
\end{equation}
The curvature transforms as a function:
\begin{equation}\label{invcurvature}
\kappa (z)=\tilde{\kappa}(\tilde{z}),
\end{equation}
and the coefficient $h_1$ transforms as an affine connection:
\begin{equation}\label{h1}
h_1(z)f^\prime (\tilde{z})=\tilde{h}_1(\tilde{z})+\frac{f''(\tilde{z})}{2f^\prime(\tilde{z})}.
\end{equation}
\end{proposition}

\begin{proof}
The lemma is proved by straight-forward computations using that the Green function itself transforms as a function:
$G(z_1,z_2)=\tilde{G}(\tilde{z}_1,\tilde{z}_2)$  ($z_j=f(\tilde{z}_j)$).
\end{proof}

\begin{remark}\label{rem:connections}
Equations (\ref{h0h0log}) and (\ref{h1}) say that $h_0$ and $h_1$ transform as different kinds of connections. 
In the terminology of \cite{Kang-Makarov-2013}, $h_0$ is a {\it pre-pre-Schwarzian} form, and $h_1$ is a {\it pre-Schwarzian}
form, the latter governed by an affine connection. Besides these two, there exist also {\it Schwarzian} forms, 
governed by Schwarzian, or {\it projective}, connections. Possibly, also such connections are  relevant for vortex motion, if so
for describing motion of vortex dipoles (see brief discussions in \cite{Gustafsson-2019}).  In a certain sense there
are no other connections, mediating inhomogeneous transformation laws in a consistent way, besides the three mentioned types, 
see \cite{Gunning-1966, Gunning-1967}.
\end{remark}

A perhaps astonishing consequence of the above proposition is
\begin{corollary}\label{cor:curvature}
Any conformally invariant metric in a simply connected domain has constant curvature.
\end{corollary}

\begin{proof}
The conformal group acts transitively in simply connected domains $\Omega$: for any two points
$a,b\in\Omega$ there is a conformal map $\Omega\to\Omega$ which takes $a$ to $b$. To see this
it is enough to consider the model with M\"obius transformations in the unit disk. Applying then (\ref{invcurvature}) with
$\tilde{z}=a$, $z=b$ gives the conclusion.
\end{proof}

The corollary does not extend to multiply connected domains because the conformal group does not act transitively on such domains.
Still the Poincar\'{e} metric has constant curvature, simply because that metric is pulled down from the universal covering
surface, which is simply connected by definition.

\begin{example}\label{ex:disk}
For the disk $D_R=\{z\in\C: |z|< R\}$ we have
$$
G(z,a)=-\frac{1}{2\pi}\log \Big|\frac{R(z-a)}{R^2-z\bar{a}}\Big|,
$$
hence, using (\ref{TaylorH}),
$$
H(z,a)=\log |\frac{R^2-z\bar{a}}{R}|
=\log|R^2-|a|^2-\bar{a}(z-a)|-\log R
$$
$$
=\re\log (1-\frac{\bar{a}(z-a)}{R^2-|a|^2})+\log(R^2-|a|^2)-\log R
$$
$$
=\log \frac{R^2-|a|^2}{R}-\re \Big(\frac{\bar{a}(z-a)}{R^2-|a|^2}+\mathcal{O}((z-a)^2)\Big),
$$
and so
\begin{equation}\label{h0h1disk}
h_0(a)=\log \frac{R^2-|a|^2}{R}, \quad h_1(a)=-\frac{\bar{a}}{R^2-|a|^2}.
\end{equation}

Choosing $R=1$ we see that
\begin{equation}\label{Poincare}
ds=e^{-h_0(z)}|dz|= \frac{|dz|}{1-|z|^2}
\end{equation}
coincides with the {\it Poincar\'{e} metric} in the case of the unit disk, and by definition so for every 
simply connected domain (other that the entire plane). For multiply connected domains $\Omega$ 
(and more general Riemann surfaces) one defines the Poincar\'{e} metric via the {\it universal covering surface}, 
which in most cases is conformally equivalent to the unit disk. See \cite{Ahlfors-1973, Beardon-Minda-2007}.

To elaborate, there exists a covering map, a holomorphic local homeomorphism $f:\D\to\Omega$,
such that the points in the preimage $f^{-1}(a)$ of any point $a\in\Omega$ are related by 
self-maps of $\D$. Then $\Omega$ becomes presented as $\D/\Gamma$ where
$\Gamma$ is a discrete group of M\"obius transformations in $\D$. 
This is a very useful point of view, which in particular makes it 
easy to to define the Poincar\'{e} metric on $\Omega$: one simply pushes down the Poincar\'{e} metric (\ref{Poincare}) from 
$\D$, which makes sense since that metric is invariant under M\"obius transformations.

Looking in the other direction, the covering map $f:\D\to\Omega$ 
has local inverses, ``liftings'', presenting $\varphi=f^{-1}$ as a multi-valued function with infinitely many branches
(unless $\Omega$ is simply connected).  These branches are related by M\"obius maps preserving $\D$, and
it is easy to see that this ensures that the combination $|\varphi'(z)|/(1-|\varphi(z)|^{2})$ is single-valued, despite $\varphi(z)$ itself is not.
Accordingly, the Poincar\'{e} metric may be defined via a variable transformation of (\ref{Poincare}) as
$$
ds=\frac{|\varphi'(z)||dz|}{1-|\varphi(z)|^2}.
$$
Eventually this means that also the Poincar\'{e} is of the form (\ref{ehzehw}), but then with the Robin function $h_0(z)$ defined in terms of
the Green function for the universal covering surface for $\Omega$.  

Besides having constant curvature, the Poincar\'{e} has the property of being complete, i.e.
the boundary being infinitely far away from any interior point. This is expressed in (\ref{Poincare})
by saying that the denominator vanishes on the boundary, or that $h_0(z)\to -\infty$ as $z$ tends to the boundary. 
Such a statement is expressed in terms of estimates for the Robin function in Proposition~\ref{prop:estimates}
in the next section. Being complete and having a specified curvature actually characterizes a metric
on a domain with smooth boundary because these conditions amount to an elliptic 
partial differential equation (nonlinear though) in the interior, plus some form of boundary data.
The definition (\ref{kappa}) of the curvature thus can be rephrased as saying that $\Delta h_0=\kappa e^{-2h_0}$,
and when $\kappa= -4$, as for the Poincar\*{e} metric, this becomes the {\it Liouville equation},
\begin{equation}\label{Liouville}
-\Delta h_0=4e^{-2h_0}.
\end{equation}

Returning briefly to the Green function, on computing
$$
\frac{\partial G(z,a)}{\partial z}=-\frac{1}{4\pi}\frac{R^2-|a|^2}{(z-a)(R^2-z\bar{a})}
$$
one finds, for analytic functions $f$ in $D_R$, that
$$
\int_{\partial D_R} f(z)\frac{\partial G(z,a)}{\partial z}dz =\frac{1}{2\I}f(a).
$$
This gives, for any harmonic function $u$, taken to be on the form $u=\re f$ with $f$ analytic,
\begin{equation}\label{Poisson}
u(a)=\re f(a)
=\frac{1}{2\pi}\int_0^{2\pi} u(Re^{\I\theta})\frac{R^2-r^2}{R^2-2Rr\cos({\theta-\varphi})+r^2}d\theta.
\end{equation}
Here we have set $z=Re^{\I\theta}$, $a=re^{\I\varphi}$ with $0\leq r<R$.
For $R=1$, the kernel in the integrand is known as the {\it Poisson kernel}.
\end{example}


\section{Capacity}\label{sec:capacity}

\subsection{Capacity in dimension two: the logarithmic capacity}\label{sec:logarithmic capacity}

Capacity in general, from a physical point of view, refers to a condensor's ability to store
charges while keeping the energy, or potential difference, within strict limits. In principle the definition is
\begin{equation}\label{capacity}
{\rm capacity}=\frac{\rm charge}{\rm potential}=\frac{({\rm charge})^2}{\rm energy}=\frac{\rm energy}{({\rm potential})^2}.
\end{equation}
In terms of units we have, for example, ${\rm joule}={\rm coulomb}\times {\rm volt}$.

The above works perfectly well in higher dimension, while the two dimensional case requires some adaptations.
More precisely, in higher dimensions it is possible to handle condensors consisting of just a single conductor,
while in two dimensions this causes problems.

In the situation of (\ref{epsilonK}), (\ref{VK}) there is the equilibrium distribution $\varepsilon_K$
for which the potential  takes the value $\gamma_K/2\pi$ on the conductor $K$. 
But one should really consider the difference of
potential between the two pieces of the condensor, which in the present case are $K$ and the point of infinity.
And this difference is infinite, meaning that the capacity actually should be zero. This is indeed reasonable
because the point of infinity (like any other single point) has no capacity to store any charges at a finite energy, 
or with finite potential.

Ignoring then the point of infinity and using just the Robin constant $\gamma_K/2\pi$ as a value for the potential, the capacity 
of the single conductor $K$ could, tentatively, be defined as
$$
{\rm tentative\,capacity\,of\,}K=\frac{2\pi}{\gamma_K}.
$$
However this expression is problematic because $\gamma_K$
may be zero or negative. Therefore one uses in two dimensional potential theory another
decreasing function of $\gamma_K$ to measure the capacity of a single conductor: 
\begin{definition}
The {\it logarithmic capacity} of a compact set $K$ is defined by
\begin{equation}\label{logcap}
{\rm logcap}(K)= e^{-\gamma_K}.
\end{equation}
We should remark that the notation ${\rm logcap}(K)$ is not standard in the literature, but it was used in \cite{Gustafsson-2004}.
\end{definition}

By Lemma~\ref{lem:gammadelta} the logarithmic capacity is the same as the {transfinite diameter} $\delta(K)$ of $K$.
So to summarize the somewhat messy situation in two dimensions that replaces the sequence of equalities
(\ref{capacity}) we may write, in a highly informal way and ignoring some constants,
\begin{align*}
{\rm logarithmic\,\, capacity}=& \exp({-{\rm equilibrium\,\,potential \,\,(at \,\,unit \,\,charge)}})\\
=&\exp(-{\rm Robin's \,\,constant})\\
=&\exp({-{\rm energy \,\, of \,\, equilibrium\,\, distribution}})\\
=&{\rm transfinite \,\,diameter}.
\end{align*}
In terms of exactly defined mathematical quantities this becomes 
$$
{\rm logcap}(K)=e^{-\gamma_K}= e^{-4\pi E(\varepsilon_K)}=\delta(K).
$$

\begin{remark}
The factor $4\pi$ in front of the energy above is absent in many mathematical texts, which have more ``clean'' definitions 
of energy. Compare Remark~\ref{rem:2pi}.
\end{remark}

\begin{example}
With $K$ a closed disk of radius $R$, or its boundary, we have
$$
{\rm logcap}(K)=R.
$$
And for $K$ a straight line segment of length $\ell$, 
$$
{\rm logcap}(K)=\ell/4.
$$
See \cite{Ransford-1995} for further examples.
\end{example}


\subsection{Geometric estimates of capacities and Robin functions}\label{sec:estimates}

The transfinite diameter, being a geometric quantity, is a useful tool for deriving estimates for capacities and Robin functions.
For the capacity and the Robin constant this is classical, see for example \cite{Ahlfors-1973, Ransford-1995}. As is obvious 
from Lemma~\ref{lem:gammadelta},
\begin{equation}
K_1\subset K_2\quad\text{implies}\quad\delta (K_1)\leq \delta (K_2), \,\, \gamma_{K_1}\geq \gamma_{K_2}.
\end{equation}
More refined estimates can be obtained by symmetrization techniques in combination with explicit examples,
as will be shortly be explained.

As for the Robin function $h_0(a)=\gamma_K(a)$ for the open set $\Omega=\C\setminus K$, 
Definition~\ref{def:transfinite diameter} again gives simple estimates as for inclusions: 
$$
a\in\Omega_1\subset \Omega_2 \quad\text{implies}\quad \delta(K_1,a)\geq \delta(K_2,a), \,\,
\gamma_{K_1}(a)\leq \gamma_{K_2}(a).
$$ 
Thus $h_0(a)$ increases with the domain. This can also be proved using a form of the maximum principle which generalizes the
Schwarz-Pick lemma and which says that any holomorphic map is distance decreasing for the metric $ds =e^{-h_0(z)}|dz|$: 
if $f:\Omega\to \tilde{\Omega}$ then
\begin{equation}\label{eh0dz}
e^{-\tilde{h}_0(f(z))}|f^\prime (z)||dz|\leq e^{-h_0(z)}|dz|,\,\, z\in\Omega.
\end{equation}
With $\Omega\subset\tilde{\Omega}$, $f(z)=z$, the mentioned monotonicity of $h_0$ follows.
The inequality (\ref{eh0dz}) is related to what is sometimes called {\it Lindel\"of's principle}, see 
\cite{Nevanlinna-1953, Gustafsson-1979}, 
and along with other similar distance shrinking properties for metrics \cite{Krantz-2004, Beardon-Minda-2007} it
generalizes the {\it principle of hyperbolic mass} \cite{Minda-Wright-1982}.

In terms of the distance to the boundary, and taking into account also the other direction,
the main estimates are as follows.
\begin{proposition}\label{prop:estimates}
The Robin function is subject to the following bounds in terms of the distance
$d(z)$ from $z\in \Omega$ to the boundary $\partial\Omega$. As a general estimate we have
$$
\log d(z)\leq h_0(z) \leq \log d(z) +A \quad(z\in\Omega)
$$
for some constant $0<A<\infty$.
Next,  depending on the geometry the upper bound can be made more precise as follows.
\begin{enumerate}
\item If $\Omega$ is convex, then
$$ 
h_0(z)\leq \log 2d(z).
$$

\item If $\Omega$ is simply connected, then
$$
h_0(z)\leq\log 4d(z).
$$

\item If $\Omega$ is not simply connected, 
let $K_1,\dots, K_m$ denote the components of $K=\C\setminus \Omega$ (each assumed to 
consist of more than one point) and set
$$
d_j(z)=\inf\{ |z-\zeta|: \zeta\in K_j\},
$$
$$
D_j(z)=\sup\{ |z-\zeta|: \zeta\in K_j\}.
$$
Then
$$
h_0(z)\leq \min_{j=1,\dots,m} \log\frac{4d_j(z)}{1-\frac{d_j(z)}{D_j(z)}}.
$$
\end{enumerate} 
\end{proposition}

\begin{proof}
The lower bound follows on comparing $h_0(z)$ with the corresponding quantity for the largest disk with center $z$
and contained in $\Omega$, hence having radius $d(z)$. Equation (\ref{h0h1disk}) with $z=a=0$, $R=d(z)$ shows
that for that disk, $h_0(z)=\log d(z)$. Since $h_0$ increases with the domain the statement on lower bound follows.

In the convex case the domain is contained in a half-plane on distance $d(z)$ from $z$, and  
the upper bound follows on comparing $h_0(z)$ for $\Omega$ with the same quantity for the half-plane. The
details (omitted) are very easy.

The last two estimates require symmetrization, or rearrangement techniques, in combination with the expression
for $h_0(z)$ in terms of the transfinite diameter as in Lemma~\ref{lem:hgammadelta}.
For the simply connected case the idea is as follows:  let $w\in\partial \Omega$ let be a closest point to $z\in\Omega$. 
We may assume that $z=0$ and $w=R>0$ for simplicity of notation.
Then the disk $\D(0,R)$
is contained in $\Omega$, but there is no explicit larger domain that we can compare with (as we could in the convex
case). 

However we see from Definition~\ref{def:transfinite diameter}, with $a=0$, that $\delta(K,0)$ does not increase if we project circularly
each point $z_j$ onto the corresponding point $|z_j|$ on the positive real axis. Indeed, the distances $|z_j-a|=|z_j|$ to $a$
do not change, while
the distances $|z_j-z_k|$ can only become smaller.  It follows that $h_0(0)$ decreases, or at least does not increase, if the complement $K$
of $\Omega$ is circularly projected onto the positive real axis.

To reach an explicit estimate we therefore need to compute $h_0(0)$ for a domain of type $D=(\C\cup \{\infty\})\setminus [a,b]$ 
where  $0<a<b$. The point of infinity actually makes no difference here, but we include it to make $D$ simply connected. This allows 
mapping $D$ conformally onto the upper half-plane by
$$
z\mapsto w=\sqrt{\frac{z-a}{z-b}}.
$$
Then $z=0$ goes to $w=\sqrt{a/b}$. The version of the Green function for the upper half-pane $U$
which is relevant in our case is
$$
G_{U}(w,\sqrt{a/b})=-\frac{1}{2\pi}\big(\log |w-\sqrt{a/b}\,|-\log |w+\sqrt{a/b}\,|\big).
$$
Thus the coefficient $h_0$ becomes, in the $w$-plane,
$$
h_0(\sqrt{a/b})=\log |\sqrt{a/b}+\sqrt{a/b}\,|.
$$

Transforming back to the symmetrized domain $D$ in the $z$-plane gives,  
by straight-forward computations using Lemma~\ref{lem:transformation},
$$
h_0(0)=\log |2\sqrt{a/b}\,|+\log |\frac{dz}{dw}|_{w=\sqrt{a/b}}=\log \frac{4ab}{b-a}.
$$

When $\Omega$ is simply connected the desired result in the lemma follows on letting $a=d(0)$, $b\to \infty$.
In the multiply connected case we take $a=d_j(0)$, $b=D_j(0)$. This finishes the proof 
of the proposition. Some further details can be found in \cite{Gustafsson-1979}.
\end{proof}

\begin{remark}
In {univalent function theory}, see \cite{Pommerenke-1975, Duren-1983}, 
one works with one-to-one, named {\it univalent} or {\it schlicht}, analytic functions in the unit disk, usually normalized according to
\begin{equation}\label{funivalent}
f(\zeta)=\zeta+a_2\zeta^2+a_3\zeta^3+\dots,
\end{equation}
that is with the first two Taylor coefficients being $a_0=0$, $a_1=1$. The book \cite{Duren-1983} by Peter Duren
appeared with an exceptionally unlucky timing: the main question in the area was {\it Bieberbach's conjecture} from 1916,
stating that $|a_n|\leq n$ for all $n$, and this conjecture was settled (as being true) by Louis de Branges 
\cite{deBranges-1985} immediately after the appearance  of the book.  

The {\it one-quarter theorem} of Koebe says that the image of the unit disk 
under any mapping as in (\ref{funivalent}) covers at least the disk of radius $1/4$ around the origin. 
Under a general variable transformation from the unit disk  to a domain $\Omega$ in the $z$-plane,
the Green functions, with poles $\zeta=0$ and $z=a$, respectively, can be directly compared with each other to give
$$
\log|\zeta|=\log |z-a|-h_0(a)+\mathcal{O} (|z-a|).
$$
This shows that, as a function of $\zeta$, the left member in 
$$
e^{-h_0(a)}(z-a)=\zeta +a_2\zeta^2+\dots .
$$
is a normalized univalent function, as in (\ref{funivalent}). 
On the boundary it is, on one hand, at least $1/4$ by Koebe,
and on the other hand at least $e^{-h_{0}(a)}d(a)$ by definition of the distance function $d(a)$.
This connects the estimate for simply connected domains in Proposition~\ref{prop:estimates}
with the Koebe theorem.

The quantity $e^{h_0(a)}$ is known in this context as the {\it mapping radius}, and it is known to
be concave (as a function of $a$) when $\Omega$ is convex, see \cite{Minda-Wright-1982}
for an elegant proof using (\ref{eh0dz}). 
Consequently, $h_0$ itself is convex. The convexity of $h_0$ is independently
proved in many places, including \cite{Caffarelli-Friedman-1985, Kawohl-1985, Gustafsson-1990a}. See also 
\cite{Bandle-Flucher-1998, Flucher-1999} for overviews.
\end{remark}


\subsection{Harmonic measure and balayage}\label{sec:balayage}

Closely related to the equilibrium distribution are concepts of {harmonic measure}  and {balayage}. 
\begin{definition}\label{def:harmonic measure}
Given a point $a$ in a domain $\Omega$, the {\it harmonic measure} with respect to $a$ is the uniquely determined measure
$\eta_a$ on the boundary $\partial \Omega$ having the property that
\begin{equation}\label{hdeta}
h(a)=\int_{\partial \Omega} h\, d\eta_a
\end{equation}
for every harmonic function $h$ in $\Omega$ smooth up to the boundary. 
Necessarily, $\eta_a$ has total mass one (and is positive).
\end{definition}

Since, in terms of the Green function $G(z,a)$ for $\Omega$, we also have 
\begin{equation}\label{hdGdn}
h(a)=-\int_{\partial \Omega} h(z)\frac{\partial G(z,a)}{\partial n} ds,
\end{equation}
we can immediately identify harmonic measure as being given by
\begin{equation}\label{dGdn}
d\eta_a=-\frac{\partial G(z,a)}{\partial n} ds.
\end{equation}
Indeed, referring to (\ref{VGK}), (\ref{DeltaVK}) we conclude that $\eta_a=\varepsilon_{K,a}$.
In the case of a disk the normal derivative of the Green function above is fully explicit
and can be identified with the {\it Poisson kernel}. See precisely (\ref{Poisson}).

We see that the right member of  (\ref{hdeta}) reproduces values of a harmonic function  from a probability measure
on the boundary, and it is an important tool in potential theory. Particularly interesting are its
roles in probabilistic potential theory. One interpretation of the harmonic
measure is that, given a segment $E\subset \partial \Omega$ of the boundary,  
$\eta_a(E)$ equals the probability for a Brownian motion particle started at $a\in\Omega$
to first reach the boundary  on exactly that segment.
See \cite{Doob-1984, Bass-1995} in general, and \cite{Grotta-Ragazzo-2024} 
for an interpretation in a vortex related context.

Much related to harmonic measure, and also to equilibrium measure, is the concept of {balayage}.
The precise meaning of {\it balayage}, french for ``cleaning'', is that a measure 
(the `dust' in the cleaning picture) sitting in a domain is swept to the boundary
in such a way that its potential outside the domain is not affected. 
The sweeping operation is $\mu\mapsto \nu$, where, assuming $\mu$ sits in 
a domain $\Omega$, $\nu$ has its support on the boundary $\partial\Omega$ and is determined by
\begin{equation}\label{UUK}
U^\nu=U^\mu \quad \text{on}\,\,K. 
\end{equation}
Here we have used $K$ for the entire complement of $\Omega$ in the Riemann sphere, hence it is still a compact set.
We see that the difference between the potentials is constantly equal to zero on $K$.
An intuitive notation is as follows. 

\begin{definition}\label{def:BalK}
{\it Classical balayage} (sweeping) of a measure $\mu$ in $\Omega$, out from $\Omega$
and to the complement $K$, is denoted
\begin{equation}\label{notationBal}
\nu={\rm Bal} (\mu, K).
\end{equation}
The new measure $\nu$ is determined by (\ref{UUK}) and actually sits on the boundary $\partial\Omega\subset K$.
\end{definition}
In the case of a point mass, for example $\mu=\delta_a$, $a\in\Omega$, 
it follows from the definition of the Green function that 
$$
U^{\delta_a}-U^{{\rm Bal}(\delta_a, K)}=G(\cdot,a) \quad \text{in } \Omega.
$$
This identity remains valid up to $\partial \Omega$, where both members vanish.
The right member is not defined outside
$\Omega\cup\partial\Omega$, but the left member is, and it vanishes identically there.  
Therefore ${\rm Bal}(\delta_a,K)$ in fact equals the previously introduced of equilibrium measure $\varepsilon_{K,a}$.
In summary:
\begin{lemma}
Balayage of a point mass, harmonic measure, and equilibrium distribution all result in the same measure:
\begin{equation}\label{BaldeltaKeta}
{\rm Bal}(\delta_a, K)=\eta_a=\varepsilon_{K,a}.
\end{equation}
\end{lemma}

Returning to the geometry with $K$ compact in $\C$ and $a=\infty$, 
it follows that also the initially defined equilibrium measure $\varepsilon_K=\varepsilon_{K,\infty}$
(see before (\ref{deltaEK}))
is simultaneously an instance of harmonic measure and of balayage, with $a=\infty$: 
$$
\varepsilon_K=\eta_\infty={\rm Bal}(\delta_\infty, K).
$$

\begin{remark}\label{rem:phil}
Balayage techniques were invented and used by Poincar\'{e} to solve the Dirichlet problem.
The idea was that in order to construct a harmonic function with given boundary values one starts with some, 
more or less arbitrary, superharmonic function $u$ having the right boundary values, and then ``sweeps'', in small and
explicit steps, the exceeding mass $-\Delta u\geq 0$ to the boundary. Eventually all that mass is swept away
and $u$ has become the actual solution. As Perron noticed (see \cite{Doob-1984} for historical information), 
this solution can simply be 
characterized as the smallest superharmonic function satifying the boundary conditions.

This contrasts, and complements, Riemann's method to solve the Dirichlet problem by energy minimization, which
later evolved into techniques of orthogonal projection and Hodge decomposition.
The two methods represent different sides of a fundamental duality in potential theory which has been made
explicit by Moreau \cite{Moreau-1962}.  See also \cite{Gustafsson-2004} for a general overview, with focus on balayage.
\end{remark} 


\subsection{The Hadamard variational formula}\label{sec:Hadamard}

The classical Hadamard formula describes the change of the Green function under a small
(infinitesimal) deformation of the boundary (assumed smooth) of the domain.
Traditionally it  reads as follows.

\begin{theorem}
Under an infinitesimal variation $\delta n$ of the boundary $\partial \Omega$ in the outward normal 
direction, the Green function changes according to
\begin{equation}\label{Hadamard}
\delta G(a,b)=\int_{\partial \Omega}\frac{\partial G(\cdot,a)}{\partial n}\frac{\partial G(\cdot,b)}{\partial n}\delta n ds.
\end{equation}
\end{theorem}
If one wants to avoid infinitesimals one can simply divide both sides with $\delta t$ and interpret $\delta n/\delta t$
as the speed of the boundary in the normal direction. The left member then becomes the derivative
$\frac{d}{dt}G_{\Omega(t)}(a,b)$. 
The proof of (\ref{Hadamard}) is simple: in (\ref{hdGdn})
one simply chooses $h(z)=-\frac{\partial G(z,b)}{\partial n}$, where after multiplication by $\delta n$
the left member can be interpreted as  $\delta G(a,b)$ after a displacement 
of the boundary in the normal direction. See  \cite{Nehari-1952, Garabedian-1964} for details.
The latter reference also contains a formulation of the Hadamard in terms of a stress-energy
tensor consisting of derivatives of the Green function in the interior of the domain. This version of the 
Hadamard variation is further developed in \cite{Gustafsson-Sebbar-2022}.   

There are many other types of domain variations, for example the {\it Schiffer variation} 
\cite{Ahlfors-1973, Duren-1983, Wolpert-2018},
which is less elementary than the Hadamard variation but has the advantage of
requiring no smoothness of the boundary.
It was invented as a tool for proving one of the estimates 
($n=4$) for the Bieberbach conjecture mentioned in Section~\ref{sec:estimates} above.

Recalling next (\ref{GlogH}) and noticing that the singularity in the right member is not affected by any variation of  
the boundary, the Hadamard formula (\ref{Hadamard}) gives
\begin{equation}\label{HadamardH}
\delta H(a,b)=2\pi \int_{\partial \Omega}\frac{\partial G(\cdot,a)}{\partial n}\frac{\partial G(\cdot,b)}{\partial n}\delta n ds.
\end{equation}
Here we can choose $a=b$ to obtain
\begin{equation}\label{Hadamardh}
\delta h_0(a)=2\pi \int_{\partial \Omega}\Big(\frac{\partial G(\cdot,a)}{\partial n}\Big)^2\delta n ds.
\end{equation}

Now consider a variation $\delta x$ in the $x$-direction.
This gives a variation in the normal direction by the $x$-component of the unit normal vector:
$\delta n=\delta x \cdot n_x=\delta x \,({\bf e}_x\cdot {\bf n})$.
With this choice $\delta h_0(a)$ can be identified with a differential,
$$
\delta h_0(a)=-\frac{\partial h_0(a)}{\partial {a_x}}\delta a_x,
$$
the minus sign arising because displacing $a$ by $da_x$ is equivalent to moving the entire domain
in the opposite direction. Proceeding similarly for $\delta n=\delta y \cdot n_y=\delta y \,({\bf e}_y\cdot {\bf n})$ 
and taking finally complex linear combinations and using the definition of the 
Wirtinger derivatives, see Section~\ref{sec:notations}, we obtain
\begin{equation*} 
h_1(a)=\frac{\partial h_0(a)}{\partial {a}}
=\frac{1}{2}\big(\frac{\partial h_0(a)}{\partial {a_x}}-\I \frac{\partial h_0(a)}{\partial {a_y}}\big)
\end{equation*}
\begin{equation*}
=-\pi \int_{\partial \Omega}\Big(\frac{\partial G(z,a)}{\partial n}\Big)^2(\, {\bf e}_x-\I \,{\bf e}_y)\cdot {\bf n} \,ds.
\end{equation*}
\begin{equation*}
=-\pi \int_{\partial \Omega}\Big|2\frac{\partial G(z,a)}{\partial z}\Big|^2\,\I d\bar{z} 
=-4\pi \I\int_{\partial \Omega}\frac{\partial G(z,a)}{\partial z}\frac{\partial G(z,a)}{\partial \bar{z}}\, d\bar{z}
\end{equation*}
$$
=4\pi\I \int_{\partial \Omega}\Big(\frac{\partial G(z,a)}{\partial z}\Big)^2\, d{z}.
$$
Here we used partial integration in the last step.

Of course, the above formula is most easily obtained directly by computing the final integral by residues, as was indeed done
in deriving the formula (\ref{force}) for the force on a vortex. Still it is worth to single out the result as statement of its own:
\begin{lemma}
The coefficient of the linear term in the expansion (\ref{GlogH}), (\ref{TaylorH}) of the Green function is given by
\begin{equation}
h_1(a)=4\pi\I \int_{\partial \Omega}\Big(\frac{\partial G(z,a)}{\partial z}\Big)^2\, d{z}.
\end{equation}
\end{lemma} 

\begin{remark}\label{rem:Laplacian growth}
It is difficult to resist letting $\delta n$ in the Hadamard formula (\ref{Hadamard})
be proportional to the normal derivative of another Green function, $G(\cdot,c)$. 
Then we get a remarkable formula, which we write as
\begin{equation}\label{integrability}
\nabla(c) G(a,b)= \int_{\partial \Omega}\frac{\partial G(\cdot,a)}{\partial n}
\frac{\partial G(\cdot,b)}{\partial n} \frac{\partial G(\cdot,c)}{\partial n} ds.
\end{equation}
Here the differential operator $\nabla(c)$ represents the infinitesimal generator for {\it Laplacian growth} with a 
sink at $c$, 
a moving boundary  problem which has a fluid dynamic interpretation in terms of so-called {\it Hele-Shaw flow} 
(there are many other interpretations as well). This problem is of considerable importance,
and has interesting probabilistic interpretations, for example as DLA, {\it diffusion limited aggregation}.
See \cite{Gustafsson-Teodorescu-Vasiliev-2014} for further information, and references.

The fact that the expression (\ref{integrability}) is completely symmetric in $a,b,c$ expresses a certain {\it integrability
of the Dirichlet problem}, or ``zero curvature'', as it also has been called.
See \cite{Mineev-1990, Wiegmann-Zabrodin-2000, Kostov-Krichever-Mineev-Wiegmann-Zabrodin-2001, 
Krichever-Marshakov-Zabrodin-2005, Alekseev-Mineev-2017}, a selection from a long series of
highly interesting papers on integrable hierachies with applications to various areas of mathematical physics
(including fluid dynamics).
\end{remark}


\subsection{Capacity in higher dimensions}\label{sec:capacityRn}

In higher dimensions the capacity for a single conductor is less problematic than in two dimensions.
In any number of dimensions a good general setting is to consider a pair
$K\subset \Omega\subset\R^n$, where $K$ is compact and $\Omega$ open and bounded,
the condensers then being $\partial K$ and $\partial \Omega$. 
Let $u$ be the harmonic function  in $\Omega\setminus K$ which has boundary values $u=1$
on $\partial K$ and $u=0$ on $\partial \Omega$. 
It represents a unit potential jump, hence the capacity shall be the corresponding 
energy, namely (up to a factor) 
\begin{equation}\label{Enablau}
E=\int_{\Omega\setminus K}|\nabla u|^2 dx,
\end{equation}
where $dx$ represents the $n$-dimensional Euclidean measure (Lebesgue measure).

A good formal approach involves the Sobolev space $H_0^1(\Omega)$, for which the norm is the 
Dirichlet integral representing energy. In terms of an obstacle problem 
the capacity can then be defined as (see \cite{Treves-1975, Kinderlehrer-Stampacchia-1980})
\begin{equation}\label{cap}
{\rm Cap}\,(K)={\rm Cap}\,(K,\Omega)=\inf \{\int_\Omega |\nabla u|^2 dx: u\in H_0^1(\Omega), u\geq \chi_K\},
\end{equation}
$\chi_K$ denoting the characteristic  function of $K$ ($\chi_K=1$ on $K$, $\chi_K=0$ in $\Omega\setminus K$).

If the ambient space $\Omega$ is kept fixed then ${\rm Cap}(K)$ increases as $K$ increases
(meaning that the two condensers $\partial K$ and $\partial\Omega$ come closer to each other). 
Similarly, it decreases if $K$ is kept fixed and $\Omega$ increases. In dimension $n\geq 3$ 
it in fact decreases to a positive limit as $\Omega\to \R^n$. However not so in two dimensions,
then the limit of ${\rm Cap}(K,\Omega)$ becomes zero when $\Omega$ increases beyond bounds.

On the other hand, a good property in two dimensions is that the capacity is conformally invariant 
as depending on $\Omega\setminus K$.  
This is obvious because the integral to be minimized in (\ref{cap}) only depends on the conformal structure
(in dimension $n=2$). This can be seen by writing the Dirichlet integral in (\ref{cap}) on the form
$\int_{\Omega\setminus K} du\wedge *du$, where the star is the Hodge star. That star depends in two dimensions
only on the conformal structure when acting on one-forms because it then is equivalent to just a rotation of
the coordinates (for example, $*dx=dy$, $*dy=-dx$).

\begin{example}
For balls $K=\{|x|\leq r\}$, $\Omega=\{|x|<R\}$, $r<R$, one gets, for the potential $u$ discussed above,
$$
u(x)=A\log |x|+B= \frac{\log R-\log |x|}{\log R-\log r}\quad (n=2), 
$$
$$
u(x)=A|x|^{2-n}+B=\frac{|x|^{2-n}-R^{2-n}}{r^{2-n}-R^{2-n}} \quad (n\geq 3).
$$
Then $\int_\Omega  du\wedge\star du =\int_{\Omega\setminus K}  du\wedge\star du=\int_{\partial K}\star du$, which gives
\begin{equation}\label{cap2}
{\rm Cap}\,(K,\Omega)
=\frac{2\pi}{\log R-\log r}\quad (n=2),
\end{equation}
$$
{\rm Cap}\,(K,\Omega)=\frac{(n-2)|S^{n-1}|}{r^{2-n}-R^{2-n}}\quad (n\geq 3).
$$

Here one sees clearly that in two dimensions the capacity depends only on the quotient $R/r$, confirming that it is conformally invariant. 
This quotient exactly corresponds to the conformal type of the annulus $\Omega\setminus K$.  And in dimension $n\geq 3$ one sees
that the limiting capacity as $R\to\infty$ is
$$
{\rm Cap}(K)=(n-2)r^{n-2}|S^{n-1}|>0.
$$
When $n=3$ this gives ${\rm Cap}(K)=4\pi r$.
\end{example}


\subsection{The Bergman kernel for a planar domain}\label{sec:Bergman} 

By taking second derivatives of the Green function one reaches the Bergman
kernel and some related kernels. We shall be careful to take all derivatives in the sense of distributions (see below), 
and that will make our treatment look slightly different 
from what one usually sees in text books, for example 
\cite{Bergman-1970, Nehari-1952, Epstein-1965, Hedenmalm-Korenblum-Zhu-2000, Duren-Schuster-2004, Bell-2016}.

The singularity of the Green function $G(z,a)$ of a domain $\Omega\subset\C$ 
can be decomposed into an analytic and an anti-analytic part as
$$
\log|z-a|^2=\log(z-a)+\log(\bar{z}-\bar{a}).
$$
This gives
\begin{equation}\label{residue}
\frac{\partial}{\partial z}\log |z-a|^2=\frac{1}{z-a},
\end{equation}
where the right member is a locally integrable function, and there are  no distributional contributions (so far).
Taking next the distributional derivative of the right member, 
either with respect to $\bar{z}$ or with respect to $\bar{a}$, and with the other variable kept fixed, one gets
$$
\frac{\partial }{\partial \bar{z}}\frac{1}{z-a}=-\frac{\partial}{\partial \bar{a}}\frac{1}{z-a}={\pi} \delta (z-a).
$$
Here $\delta(z)$ denotes the ordinary Dirac distribution in the plane, and we have used that $1/({\pi z})$ is a {\it fundamental solution}
of $\partial/ \partial\bar{z}$.  The latter means that if $\varphi$ is a smooth test function with compact support in the plane, then
$$
\int_\C \big(\frac{\partial}{\partial \bar{z}}\frac{1}{\pi z}\big)\varphi (z)dxdy
=-\int_\C \frac{1}{\pi z}\frac{\partial\varphi}{\partial \bar{z}}dx\wedge dy
=-\lim_{\varepsilon\to 0}\frac{1}{2\pi \I}\int_{\{|z|>\varepsilon\}} \frac{1}{z}\frac{\partial\varphi}{\partial \bar{z}}d\bar{z}\wedge dz
$$
$$
=-\lim_{\varepsilon\to 0}\frac{1}{2\pi \I}\int_{\{|z|>\varepsilon\}} d\Big(\frac{1}{z}\varphi(z) dz\Big)
=\lim_{\varepsilon\to 0}\frac{1}{2\pi \I}\oint_{|z|=\varepsilon} \varphi(z)\frac{dz}{z}=\varphi (0).
$$
The first equality represents the definition of a {\it distributional derivative}, 
and the further equalities illustrate the use of Stokes' formula in a complex variable setting.
See \cite{Treves-1975} for more details on fundamental solutions. It follows from the above that
\begin{equation}\label{d2log}
\frac{\partial^2 \log |z-a|^2}{\partial z \partial \bar{z}}=-\frac{\partial^2 \log |z-a|^2}{\partial z \partial \bar{a}}={\pi}\delta (z-a).
\end{equation}

Again with $\varphi$ a smooth test function we have, since $G(\cdot,a)=0$ on $\partial\Omega$ for all $a\in\Omega$,
$$
0=\frac{\partial}{\partial a}\int_{\partial \Omega}\varphi (z)G(z,a)dz=
\int_{\partial \Omega}\varphi (z)\frac{\partial G(z,a)}{\partial a}dz
$$
$$
=\int_{\Omega}d\Big(\varphi (z)\frac{\partial G(z,a)}{\partial a}dz\Big)
=\int_{ \Omega}\frac{\partial \varphi}{\partial \bar{z}}\frac{\partial G(z,a)}{\partial a}d\bar{z}dz
+\int_\Omega \varphi(z){\frac{\partial^2 G(z,a)}{\partial \bar{z}\partial {a}}} d\bar{z}dz.
$$
Now, if $\varphi$ is analytic in $\Omega$ the first term disappears, and changing notation from
$\varphi(z)$ to $f(z)$ we thus have
\begin{equation}\label{intfG}
\int_\Omega f(z){\frac{\partial^2 G(z,a)}{\partial \bar{z}\partial {a}}} d\bar{z}dz=0
\end{equation}
for functions $f(z)$ analytic in $\Omega$ and, say, smooth up to $\partial\Omega$.
On decomposing $G(z,a)$ as in (\ref{GlogH}), and using (\ref{d2log}) this gives 
$$
\int_\Omega f(z){\frac{\partial^2 H(z,a)}{\partial \bar{z}\partial {a}}} d\bar{z}dz=-\I{\pi}f(a).
$$
It follows that
\begin{equation}\label{Bergman}
K(z,a)=-\frac{2}{\pi}{\frac{\partial^2 H(z,a)}{\partial {z}\partial \bar{a}}}
\end{equation}
is the Bergman kernel of $\Omega$, which is characterized by its reproducing property:
\begin{definition}\label{def:Bergman}
The {\it Bergman kernel} of $\Omega$ is the unique square integrable analytic function in $\Omega$ 
having the reproducing property that
\begin{equation}\label{Kreproducing0}
\int_\Omega f(z)\overline{K(z,a)}dxdy=f(a)
\end{equation}
for all square integrable analytic functions $f$ in $\Omega$.
\end{definition}

If we specialize (\ref{Bergman}) to the diagonal $z=a$ 
and use (\ref{h00H}) to express the result in terms of $h_0(z)=H(z,z)$ (see (\ref{TaylorH})) we obtain
the identity
\begin{equation}\label{Deltah0K}
-\Delta h_0(z)=4\pi K(z,z). 
\end{equation}

\begin{remark}\label{rem:Fredholm}
Kernels, or integral kernels in general, trace back  to the work of Ivar Fredholm, who developed a theory
for integral equations, see \cite{Fredholm-1903, Courant-Hilbert-1943}.
 Later David Hilbert raised the theory to a more abstract level for which the kernels
represent linear operators in a Hilbert space. The Bergman kernel can be viewed in this context as 
representing the orthogonal projection of the full Lebesgue space $L^2(\Omega)$ of a domain
onto its subspace $L^2_a(\Omega)$ of analytic functions. When restricted to
$L^2_a(\Omega)$ it simply becomes the identity operator, thus 
reproducing the values of a function. This gives the defining property (\ref{Kreproducing0}).

The property of representing the identity operator also shows up in the expansion of the Bergman kernel along 
an arbitrary orthonormal basis $\{e_n\}$:
$$
K(z,a)=\sum_{n=1}^\infty e_n(z)\overline{e_n(a)}.
$$ 
Here each term represents the orthogonal projection onto a one-dimensional subspace (that generated by the vector $e_n$
in question). Later we shall identify the above (standard) Bergman kernel as the {\it electrostatic Bergman kernel}, to distinguish
it from a {\it hydrodynamic Bergman kernel} with a reproducing property on a certain subspace, and hence represented 
with a shorter sum above. See Section~\ref{sec:hydrodynamic Bergman}. The missing terms in the sum will turn out to represent
a Bergman kernel relevant on closed Riemann surfaces, see Sections~\ref{sec:monopole Green}-\ref{sec:planar domains} 
below. Our terminology  (electrostatic/hydrodynamic) is not standard in other texts, but is inspired by \cite{Cohn-1980} and has been used in 
\cite{Grotta-Ragazzo-Gustafsson-Koiller-2024}.

\end{remark}


\section{The monopole Green function on closed surfaces}\label{sec:monopole Green}

\subsection{Definition of the monopole Green function}\label{sec:definition Green}

On a closed Riemann surface (compact Riemann surface without boundary) $M$ there is no Green function with just a single pole because the
strengths of the logarithmic poles must add up to zero. Traditionally, given a desired pole like $-\frac{1}{2\pi\I}\log|z-a|$
at a given point $a\in M$, one introduces a counter-pole $+\frac{1}{2\pi\I}\log |z-b|$ at some other point $b\in M$ to obtain
a function with singularity structure 
\begin{equation}\label{V}
V(z)=\frac{1}{2\pi}\big(-\log |z-a|+\log |z-b|\big)+{\rm harmonic}.
\end{equation}
After a normalization, requiring $V(z)$ to vanish at a third point $z=w$, $V(z)=V(z,w;a,b)$ becomes uniquely determined
and has  natural symmetry properties (see for example \cite{Gustafsson-Sebbar-2012}). 
This is illustrated by its explicit form in the case of the Riemann sphere, for which it
(up to a factor) becomes the logarithm of the modulus of the {\it cross ratio}:
\begin{equation}\label{Vcross ratio}
V(z,w;a,b)=-\frac{1}{2\pi}\log\big| \frac{(z-a)(w-b)}{(z-b)(w-a)}\big|=-\frac{1}{2\pi}\log|(z:w:a:b)|.
\end{equation}

With $V(z,w;a,b)$ as a building block, in the case of a general compact Riemann surface, 
practically all basic harmonic and analytic functions,
as well as differentials of different sorts, can be easily constructed. This was indeed done in the classical era by Riemann, Klein
and their followers, see for example \cite{Weyl-1964, Schiffer-Spencer-1954}. 

If the Riemann surface is provided with a Riemannian metric compatible with the conformal structure, namely of the form
\begin{equation}\label{dsdz}
ds^2=\lambda(x,y)^2(dx^2+dy^2)=\lambda(z)^2 |dz|^2 \quad (\lambda>0),
\end{equation}
then there is, besides the mentioned classical procedure, another route to construct the basic harmonic and analytic objects.
This goes via the {\it monopole Green function}, and has previously been used in \cite{Lang-1988, Takhtajan-2001}.
First of all, the metric (\ref{dsdz}) gives rise to its corresponding two-dimensional volume form (area form) which,
following the notations of \cite{Frankel-2012}, is 
\begin{equation}\label{volume form}
{\rm vol}=\lambda(x,y)^2 dx\wedge dy= \frac{\lambda (z)^2}{2\I}\,d\bar{z}\wedge dz. 
\end{equation}
Clearly, this contains exactly the same information as the metric.

There is also the {\it Hodge star} operator defined on basic forms by
$*1={\rm vol}$, $*dx=dy$, $dy=-dx$, $*{\rm vol}=1$.
For complex basic one-forms this gives $*dz=-\I dz$, $*d\bar{z}=\I d\bar{z}$.
The Riemann surface itself will be assumed to be closed (compact) of genus $\texttt{g}$, and provided with a canonical 
homology basis $\alpha_1, \dots, \alpha_\texttt{g}$, $\beta_1, \dots, \beta_\texttt{g}$. This means that 
$\alpha_k$ does not intersect $\beta_j$ if $k\ne j$ and that $\alpha_k$ and $\beta_k$ intersect like the $x$-axis
and $y$-axis in an ordinary Cartesian coordinate system.
See \cite{Forster-1981, Farkas-Kra-1992} for relatively recent introductory texts on Riemann surfaces.

Now, in the presence of the metric (\ref{dsdz}) there is the option to replace, in the potential $V(z,w;a,b)$, 
the counter-pole at $z=b$ with an extended sink proportional to ${\rm vol}$.
One can view this replacement as a form of balayage, more precisely partial balayage \cite{Gustafsson-2004, Gustafsson-Roos-2018}. 
We shall work more exactly with the so obtained {monopole Green function}. 
Let $V={\rm vol}(M)$ denote the volume (area) of $M$.

\begin{definition}
The {\it monopole Green function} $G(z,a)$ is the unique solution of
\begin{equation}\label{ddGdelta}
-d*dG(\cdot,a) =\delta_a -\frac{1}{V}\,{\rm vol}
\end{equation}
subject to the normalization 
\begin{equation}\label{normalization}
\int_M G(\cdot,a)\, {\rm vol}=0.
\end{equation}
\end{definition}

Note that everything is written in the formalism of differential forms, and in particular
the Dirac measure $\delta_a$ is considered as the two-form current (the area form thus built in)
determined by 
\begin{equation}\label{varphideltaa}
\int_M \varphi\wedge  \delta_a=\varphi(a)
\end{equation}
for any test function $\varphi$ on $M$.

The counter pole at $z=b$ has been replaced by a sink distributed all over the surface, 
more precisely a multiple of ${\rm vol}$,  and (\ref{normalization}) substitutes the 
previous normalization, which  asked the potential $V$ in (\ref{V}) to vanish at  $z=w$.
The requirements (\ref{ddGdelta}), (\ref{normalization}) determine $G(\cdot, a)$ uniquely, and by partial integration the identity
\begin{equation}\label{mutual energyG}
G(a,b)=\int_M dG(\cdot,a)\wedge *dG(\cdot,b),
\end{equation}
follows directly. It presents $G(a,b)$ as the mutual energy between two point charges/vortices, one at $a$ and one at $b$,
and it has the symmetry
$$
G(a,b)=G(b,a)
$$  
automatically built in. 

The equation (\ref{mutual energyG}) might seem dubious from the point of view of dimensional analysis, but one has to 
keep in mind that there is an implicit unit charge (in electrostatic language) in the right member of (\ref{ddGdelta}).
This is to be multiplied with the left member in (\ref{mutual energyG}), which then is given dimension
charge times potential, which is energy. 
The right member is a standard energy expression, a Dirichlet integral, like in (\ref{Enablau}). See further below. 

To elaborate the above one can think of $dG$ as a field strength and $*dG$ as the corresponding current (or displacement field,
or induction, depending on the context).  This is more obvious in a three dimensional interpretation, where $dG$ is a
one-form and $*dG$ a two-form. The latter can be integrated over a surface which then gives the current flowing through
that surface. An even more refined picture can be obtained by distinguishing between ordinary differential forms and {twisted forms}
(or pseudo forms), as is done in \cite{Burke-1983, Burke-1985, Frankel-2012}.  Then $dG$ is an ordinary from,
representing a field strength, while $*dG$ is a {\it twisted form} representing some kind of ``quantity'' (amount of current). 
Also the volume (or area) form ${\rm vol}$ is a twisted form in this language. However, we shall not go further into this direction.
 
In view of (\ref{ddGdelta}) the potential $G(\cdot,a)$ is not harmonic anywhere, still it can be used to construct
all fundamental harmonic and analytic objects on $M$.
From a more general perspective $G(\cdot,a)$ can be viewed as the result of a Hodge decomposition of the two-form $\delta_a$.
Following \cite{Gustafsson-2022a, Grotta-Ragazzo-Gustafsson-Koiller-2024} we can write
\begin{equation}\label{Hodgedeltaa}
\delta_a= d({\rm something})+{\rm harmonic \ form},
\end{equation}
and in the fluid picture this is exactly (\ref{ddGdelta}), with the harmonic form then being the distributed counter vorticity. 
One should notice that since the only harmonic functions on a compact surface are the constants, and since the Hodge star
of a constant is a multiple of the volume form, the harmonic two-forms are exactly the multiples of this volume form.
The above procedure gives the Green function  (\ref{ddGdelta}), (\ref{normalization})
by what is effectively an orthogonal projection. Compare Remark~\ref{rem:phil}.

Slightly more generally, any vorticity distribution $\omega$ can be decomposed as
\begin{equation}\label{omegaddG}
\omega=-d*d G^\omega +c\cdot {\rm vol}, \quad c=\frac{1}{V}\int_M\omega.
\end{equation}
The {\it Green potential} $G^\omega$ is defined similarly as $G^{\delta_a}$ in (\ref{ddGdelta}) and with a normalization as in (\ref{normalization}).
In the Hodge decomposition (\ref{omegaddG}), the last term, the counter vorticity, can somewhat in analogy with dark matter and 
dark energy in cosmology be thought of as  a {\it dark vorticity} which permeates all of $M$. It is present whenever $\int_M \omega\ne 0$,
and it is ``dark'' for example in the sense that it possesses no energy, see (\ref{Evolvol}) below. 

The kinetic energy of a flow one-form is naturally expressed in terms of the inner product, representing {\it mutual energy}
for flow fields:
\begin{equation}\label{mutual energy}
(\nu_1,\nu_2)=\int_M \nu_1\wedge*\nu_2.
\end{equation}
In terms of potentials this becomes the Dirichlet integral $(du_1,du_2)$. For vorticity two-forms it defines the mutual energy as follows.
\begin{definition}\label{def:mutual energy}
The {\it mutual energy} between two vorticity distributions is
$$
\mathcal{E} (\omega_1, \omega_2)=(dG^{\omega_1}, dG^{\omega_2})=\int_M G^{\omega_1}\wedge \omega_2.
$$  
\end{definition}  
Choosing $\omega_1=\omega_2={\rm vol}$ and using (\ref{normalization}) gives that
\begin{equation}\label{Evolvol}
\mathcal{E}({\rm vol}, {\rm vol})=0,
\end{equation}
as claimed. It is also interesting to notice that (\ref{mutual energyG}) can be expressed as
$$
G(a,b)=\mathcal{E}(\delta_a, \delta_b).
$$

Clearly, the four variable potential mentioned in the beginning of this section is, in terms of $G(z,a)$,
\begin{equation}\label{VG}
V(z,w;a,b)=G(z,a)-G(z,b)-G(w,a)+G(w,b).
\end{equation}


\subsection{Local properties}\label{sec:local properties}

Decomposing next, in a local coordinate $z$ and keeping $a$ fixed in the same coordinate patch,  
the monopole Green function into its singular and regular parts as
\begin{equation}\label{GHlog}
G(z,a)=\frac{1}{2\pi}\Big(-\log |z-a|+H(z,a)\Big),
\end{equation}
one can expand the regular part in a Taylor series around $z=a$ as
\begin{equation}\label{HTaylor1}
H(z,a)=h_0(a)+\frac{1}{2}\Big(h_1(a)(z-a)+\overline{h_1(a)}(\bar{z}-\bar{a})\Big)+
\end{equation}
\begin{equation}\label{HTaylor2}
+\frac{1}{2}\Big(h_2(a)(z-a)^2+\overline{h_2(a)}(\bar{z}-\bar{a})^2\Big)+h_{11}(a)(z-a)(\bar{z}-\bar{a})+\mathcal{O}(|z-a|^3). 
\end{equation}
This expansion differs from (\ref{TaylorH}) by the presence of non-harmonic terms, 
most importantly the one with coefficient $h_{11}(a)$. 
Note that $H(z,a)$, unlike $G(z,a)$, is only locally defined, namely near the diagonal with $z$ close to $a$, and it moreover
depends on the choice of local coordinate since the balance between the two terms in (\ref{GHlog}) changes under a coordinate
transformation. 
However, certain derivatives of $H(z,a)$ are still meaningful global objects as being coefficients of suitable double differentials.

One may, to this purpose, start from (\ref{ddGdelta}), 
which rephrased in terms of $H(z,a)$ and the present local variables becomes, in view of (\ref{d2log}),
$$
\frac{\partial^2 G(z,a)}{\partial z \partial \bar{z}}
=-\frac{1}{2\pi}\frac{\partial^2}{\partial z\partial \bar{z}}\log |z-a|
+\frac{1}{2\pi}\frac{\partial^2 H(z,a)}{\partial z \partial \bar{z}}
$$
\begin{equation}\label{dzzG}  
= -\frac{1}{4}\delta (z-a)+\frac{\lambda (z)^2}{4V},
\end{equation}
where the last term comes from comparison of the left member with (\ref{ddGdelta}).
Here $\delta (z-a)$ is to be interpreted as the usual Dirac distribution in terms of the coordinate $z$,
evaluating the value at $z=a$ upon integrating functions 
with respect to the area form $dxdy$. 

For the second term it follows that
\begin{equation}\label{Hlambda}
\frac{\partial^2 H(z,a)}{\partial z \partial \bar{z}}=\frac{\pi\lambda (z)^2}{2V},
\end{equation} 
in particular that the left member of (\ref{Hlambda}) is independent of $a$: 
\begin{equation}\label{Hazz}
\frac{\partial }{\partial {a}}\frac{\partial^2 H(z,a)}{\partial z\partial \bar{z}}
=\frac{\partial }{\partial \bar{a}}\frac{\partial^2 H(z,a)}{\partial z\partial \bar{z}}=0.
\end{equation}
Reversing the order of differentiation,
\begin{equation}\label{Hzza}
\frac{\partial^2 }{\partial z\partial \bar{z}}\frac{\partial H(z,a)}{\partial a}
=\frac{\partial^2 }{\partial z\partial \bar{z}}\frac{\partial H(z,a)}{\partial \bar{a}}=0,
\end{equation}
hence the $a$-derivatives of $H(z,a)$ are harmonic with respect to $z$. 
In the same vein we have
\begin{equation}\label{Hzaz}
\frac{\partial }{\partial \bar{z}}\frac{\partial^2 H(z,a)}{\partial z\partial a}
=\frac{\partial}{\partial \bar{z}}\frac{\partial^2 H(z,a)}{\partial z\partial \bar{a}}=0.
\end{equation}
Thus the mixed second order derivatives of $H(z,a)$ in (\ref{Hzaz}) are analytic as functions of $z$.
Moreover, these mixed derivatives do not depend on the metric.

However, as previously remarked, $H(z,a)$ is not really a function on $M$ since the logarithmic term in (\ref{GHlog}) 
only makes sense when $z$ and $a$ are within the same coordinate patch.
On the other hand, $G(z,a)$ as a whole is a well-defined function on $M$, and we have
\begin{equation}\label{GHza}
-4\frac{\partial^2 G(z,a)}{\partial z\partial {a}}
=\frac{1}{\pi (z-a)^2}-\frac{2}{\pi}\frac{\partial^2 H(z,a)}{\partial z\partial {a}}
\end{equation}
close to the diagonal and in terms of any local coordinate (the same for $z$ and $a$).  
The point with this formula is that it exhibits the singularity structure of the left member in a clear way.
One sees that  quantity in (\ref{GHza}), as depending on $z$ and for any fixed $a$, 
is the coefficient of a basic {\it Abelian differential of the second kind} on $M$,
``second kind'' referring to having only residue free poles (namely the one visible in the right member above). 
  
As for  the derivative $\partial^2 H(z,a)/\partial z\partial \bar{a}$, 
the logarithmic term in $G(z,a)$ contributes only with a point mass at $z=a$, as in (\ref{d2log}), 
and this can be isolated and be given an independent meaning. On writing
\begin{equation}\label{HGza}
-4\frac{\partial^2 G(z,a)}{\partial z\partial \bar{a}}
=-\delta(z-a)-\frac{2}{\pi} \frac{\partial^2 H(z,a)}{\partial z\partial \bar{a}},
\end{equation}
the last term is therefore analytic/anti-analytic on all of $M\times M$ when interpreted as the coefficient of a $dzd\bar{a}$-double differential.
Clearly, such a thing must be a fundamental quantity.  This leads to
\begin{definition}\label{def:Bergman and Schiffer}
The {\it Bergman kernel} for the global holomorphic one-forms on a closed Riemann surface $M$ is
\begin{equation}\label{K}
K(z,a)dzd\bar{a}=-\frac{2}{\pi}\frac{\partial^2 H(z,a)}{\partial z \partial\bar{a}}dzd\bar{a}.
\end{equation}
The accompanying kernel (\ref{GHza}),
\begin{equation}\label{L}
L(z,a)dzda=-4\frac{\partial^2 G(z,a)}{\partial z\partial {a}}dzda
\end{equation}
is in \cite{Schiffer-Spencer-1954} just called the ``$L$-kernel'', but following several other sources, 
including \cite{Takhtajan-2001, Wolpert-2018}, we shall use the name {\it Schiffer kernel} for it. 
\end{definition}

We note from (\ref{GHza}) the singularity structure 
\begin{equation}\label{singularityL}
L(z,a)dzda=\frac{dzda}{\pi(z-a)^2}+\text{regular}
\end{equation}
of the Schiffer kernel, while the Bergman kernel itself is completely regular. 
Note also that $L(z,a)dzda$ is symmetric in $z$ and $a$, while for the Bergman kernel we have
a Hermitean symmetry: 
\begin{equation}\label{symmetryK}
K(a,z)dad\bar{z}=\overline{K(z,a)dzd\bar{a}}.
\end{equation}

Usually one needs to specify period requirements for Abelian differentials, but we
emphasize that the Bergman and Schiffer kernels as defined above for a closed Riemann surface
are completely canonical.  They
neither depend on any period requirements or on any underlying metric, the latter despite they were defined using the monopole Green function,
which does depend on the metric. But one could equally well have defined them using the potential $V(z,w;a,b)$ in (\ref{VG}), and in that case the
definitions above become identical with the corresponding definitions in beginning of Section~4.10 of \cite{Schiffer-Spencer-1954}.

The Bergman and Schiffer kernels in principle live on separate closed surfaces, but they are connected by their 
relations to the Green function. This implies that their periods are linked, see (\ref{gammaLK}) below. 
One may also provide the two surfaces with opposite conformal structures and then view the total disconnected surface as
a double of each of the pieces. Such points of views are put forward in \cite{Schiffer-Spencer-1954}, where even ''topological surgery''
is discussed, like cutting small holes in the surfaces and connecting them by electric wires. See Section~4.2 and Chapter~7 
in \cite{Schiffer-Spencer-1954}.

When discussing doubles of planar domains in Section~\ref{sec:planar domains}, the Bergman and Schiffer kernels will become 
even more tight, and they are in certain combinations  glued along the boundary. 
See Proposition~\ref{prop:KKL} and (\ref{LKboundary}) below (next section).


\subsection{Harmonic one-forms}\label{sec:harmonic forms}

Stepping down from second order to first order derivatives we have, by
differentiating the monopole Green function with respect to $z$ and taking Hodge stars, 
$$
dG(z,a)=\frac{\partial G(z,a)}{\partial z}dz+\frac{\partial G(z,a)}{\partial \bar{z}}d\bar{z}=2\re \big(\frac{\partial G(z,a)}{\partial z}dz\big),\qquad 
$$
\begin{equation}\label{stardG0}
*dG(z,a)=-\I\frac{\partial G(z,a)}{\partial z}dz+\I\frac{\partial G(z,a)}{\partial \bar{z}}d\bar{z}=2\im \big(\frac{\partial G(z,a)}{\partial z}dz\big).
\end{equation}
For the combined complex differential this gives the pole structure
$$
dG(z,a)+\I *dG(z,a) =2\frac{\partial G(z,a)}{\partial z}dz
$$
\begin{equation}\label{dGistardG}
=\frac{1}{2\pi}\Big(-\frac{dz}{z-a}+2\frac{\partial H(z,a)}{\partial z}dz\Big),
\end{equation}
for $z$ close to $a$. 
Here the last term is in general not harmonic with respect to $z$ because of the identity (\ref{Hlambda}).
However, remarkably, it is harmonic with respect to $a$ as a consequence of the identity (\ref{Hzza}) with 
interchanged roles between $z$ and $a$. Recall that $H(z,a)$ is symmetric in $z$ and $a$ (by (\ref{GHlog})),
but it is not harmonic in any of these variables. 

Concerning periods around cycles, the differential $dG(\cdot,a)$ is exact and has no periods,
while $*dG(\cdot,a)$ does have periods. 
As a consequence of what just have been said about harmonicity,
these periods depend harmonically on $a$, as long as this point stays away from the cycle itself. 
More precisely, let $\gamma$ be a cycle (closed oriented curve)
in $M\setminus\{a\}$. Then integrating $*dG(\cdot,a)$ along $\gamma$ defines a function 
\begin{equation}\label{Ugamma}
U_\gamma(a)=\oint_\gamma *dG(\cdot,a)=-\I\oint_\gamma \big(dG(\cdot,a)+\I*dG(\cdot,a)\big)
\end{equation}
$$
=-2\I \oint_\gamma\frac{\partial G(z,a)}{\partial z}dz
=\frac{1}{2\pi\I}\oint_\gamma\Big(-\frac{dz}{z-a}+2\frac{\partial H(z,a)}{\partial z}dz\Big),
$$
which away from $\gamma$ is harmonic. This harmonic part may be recognized as being the
{\it basic harmonic integral associated to the cycle $\gamma$}.
As for the (non-harmonic) behavior on $\gamma$ itself,
it is easy to see from the residue term in the last expression
that $U_\gamma (a)$ steps up by one unit when $a$ crosses $\gamma$
from the left hand side of it to the right hand side. 

It follows that the restriction of $U_\gamma$ to $M\setminus \gamma$ has indefinite
harmonic  extensions across $\gamma$ and then becomes a multi-valued harmonic
function on $M$, with different branches differing by integers. Classically one simply
accepts this multi-valuedness and works with it without problems. This was also done in 
\cite{Gustafsson-2022a, Grotta-Ragazzo-Gustafsson-Koiller-2024}.  
Or else one notices that the function becomes single-valued on the {\it universal covering surface}
(see \cite{Ahlfors-1973}) of $M$. 
Thus the multi-valuedness of the harmonic part of $U_\gamma$ is no real problem, and it disappears 
under differentiation. For example, the multi-valued functions $U_\gamma$ can be used directly in 
Definition~\ref{def:omegagamma} below.

Still we prefer in this paper to work with $U_\gamma$ as defined almost everywhere on 
$M$ by (\ref{Ugamma}) and to take derivatives of it in the sense of distributions. 
This means that
the resulting differential $dU_\gamma$ will contain a singular contribution along $\gamma$.  
Like for the Green function itself, with its decomposition (\ref{GHlog}) into regular and singular parts,
the distributional one-form $dU_\gamma$ therefore decomposes into a regular harmonic one-form, 
to be denoted $\eta_\gamma$, 
and a distributional contribution on $\gamma$, which we shall denote $\delta_\gamma$. This $\delta_\gamma$ is really
a one-form ``current'', a differential form with distributional coefficients. 
See in general \cite{deRham-1984, Federer-1969, Griffiths-Harris-1978} for the terminology. 
If for example $\gamma$ is the $x$-axis (in a local coordinate), then 
 $\delta_\gamma=d\chi_{\{y<0\}}=-\delta(y)dx$ in the present situation, $\delta(y)$ denoting the ordinary Dirac
distribution in the real variable $y$.

In conclusion, we have the decomposition
\begin{equation}\label{deltaeta}
dU_\gamma= {\rm current \, along\,\,}\gamma +{\rm harmonic \,form}
=\delta_\gamma+\eta_\gamma.
\end{equation}
When integrating $dU_\gamma$ along another closed curve $\sigma$ which intersects $\gamma$, once
from the right to the left, the definition of $\delta_\gamma$ means that $\oint_\sigma \delta_\gamma=-1$.
Since necessarily $\oint_\sigma \delta_\gamma+\oint_\sigma \eta_\gamma=\oint_\sigma dU_\gamma=0$
it follows that $\oint_\sigma \eta_\gamma=+1$. The latter can also be expressed on the form
$$
\int_M \eta_\gamma \wedge \eta_\sigma=+1.
$$
Indeed, since $U_\gamma$ is harmonic in $M\setminus\gamma$ and jumps down one unit as $\sigma$ 
crosses the two-sided boundary $\partial_\pm (M\setminus\gamma)$ , the left member equals
$$
\int_M \eta_\gamma\wedge \eta_\sigma
=\int_{M\setminus \gamma} dU_\gamma\wedge \eta_\sigma
=\int_{M\setminus \gamma} d(U_\gamma \eta_\sigma)
$$
$$
=\int_{\partial_\pm (M\setminus \gamma)}U_\gamma\eta_\sigma
=\oint_\gamma(-1)\eta_\sigma=\oint_\sigma\eta_\gamma.
$$ 

More generally, assuming that  $\sigma$ is any cycle, which may have several 
crossings with $\gamma$, and defining the {\it intersection number} 
$\gamma\times\sigma=-\oint_\gamma \eta_\sigma$ 
as the number of such crossings, it follows that
\begin{equation}\label{gammacrosssigma1}
\int_M \eta_\gamma\wedge \eta_\sigma=\gamma\times\sigma.
\end{equation}
It is easy to see that $\eta_\gamma$ depends only on the homology class of $\gamma$, namely that
$\eta_\gamma=0$ whenever $\gamma =\partial D$ with $D$ is a subdomain of $M$.
One then concludes that the map $\gamma\mapsto \eta_\gamma$ is an isomorphism,
preserving also product structure as in (\ref{gammacrosssigma1}), 
of the first homology group of $M$  onto the first
de Rham cohomology group as represented uniquely by harmonic forms. See \cite{deRham-1984, Warner-1983}
for the terminology. 

Completing next $dU_\gamma$ in (\ref{deltaeta}) with a corresponding imaginary part gives in principle
$$
dU_\gamma+\I *dU_\gamma =(\delta_\gamma +\I*\delta_\gamma)+ (\eta_\gamma+\I*\eta_\gamma).
$$
The meaning of the first term in the right member is best explained in terms of the example just above, 
where $\delta_\gamma=-\delta(x)dy$.
In that case we will have $\delta_\gamma+\I *\delta_\gamma=-\delta(y)(dx+\I*dy)=-\delta(y) dz$.
However, it is really the holomorphic part $\eta_\gamma+\I*\eta_\gamma$ that we are interested in. 
To identify it we differentiate the last expression in (\ref{Ugamma}) for $a\notin\gamma$ to obtain 
\begin{equation}\label{dUstardUH}
(\eta_\gamma +\I*\eta_\gamma)(a)=2\frac{\partial U_\gamma(a)}{\partial a}da
=-4\I\oint_\gamma\frac{\partial^2G(z,a)}{\partial z\partial a}dzda
\end{equation}
\begin{equation}\label{dUstardUH2}
=+4\I\oint_\gamma\frac{\partial^2G(z,a)}{\partial \bar{z}\partial a}d\bar{z}da
=-\frac{2}{\I\pi}{\oint_\gamma} \frac{\partial^2 H(z,a)}{\partial \bar{z}\partial {a}}d\bar{z}d{a}.
\end{equation}
We summarize in terms of the Bergman and Schiffer kernels:
\begin{definition}\label{def:omegagamma} 
The {\it basic holomorphic and harmonic one-forms}, $\omega_\gamma$ and $\eta_\gamma$, associated to a given cycle $\gamma$
in $M$ are  
\begin{align*}
\eta_\gamma&=\text{the regular (harmonic) part of } \, dU_\gamma,\\
\omega_\gamma&=\text{the regular (holomorphic) part of }\,  dU_\gamma+\I *dU_\gamma.
\end{align*}
and they are related to the Bergman kernel (\ref{K}) and Schiffer kernel (\ref{L}) via
(for any $a\in M$)
\begin{align}\label{omegaeta}
\omega_\gamma(a)=\eta_\gamma(a)+\I*\eta_\gamma(a)
&=\I \oint_\gamma L(z,a)dzda\\ 
&=-\I \oint_\gamma \overline{K(z,a)dzd\bar{a}}.
\end{align}
In particular the Bergman and Schiffer kernels are linked by
\begin{equation}\label{gammaLK}  
 \oint_\gamma L(z,a)dzda+\oint_\gamma \overline{K(z,a)dzd\bar{a}}=0,
\end{equation}
holding for any closed curve $\gamma$
(integration with respect to $z$).
\end{definition}

Of course, (\ref{gammaLK}) just expresses that the combined integrand is (essentially) exact.
More precisely we have, from (\ref{HGza}), (\ref{K}), (\ref{L}), that 
$$
d\big(\frac{\partial G(z,a)}{\partial a} \big)=L(z,a)dz+ \overline{K(z,a)}d\bar{z}+\frac{1}{4}\delta(z-a)d\bar{z},
$$
and the Dirac contribution does not affect the line integral in (\ref{gammaLK}). 


\subsection{Period matrices}\label{sec:period matrices}

Next we proceed to choose $\gamma=\alpha_j,\beta_j$ in the above construction. This gives bases for the harmonic and analytic differentials.  
The harmonic differentials are naturally ordered as
$\{-\eta_{\beta_1},\dots, -\eta_{\beta_\texttt{g}}, \eta_{\alpha_1}, \dots, \eta_{\alpha_\texttt{g}} \}$,
matching the chosen homology basis in the sense that
\begin{equation}\label{etaalpha}
\oint_{\alpha_k} (-\eta_{\beta_j})=\delta_{kj}, \quad \oint_{\alpha_k} \eta_{\alpha_j}=0,
\end{equation}
\begin{equation}\label{etabeta}
\ \oint_{\beta_k} (-\eta_{\beta_j})=0, \qquad \oint_{\beta_k} \eta_{\alpha_j}=\delta_{kj}. 
\end{equation}

It is convenient to summarize formulas like this in block matrix form. We then arrange cycles and harmonic forms into columns
and use superscript $T$ to denote transpose. Involving also the Hodge star we then have, for example,
\begin{equation}\label{block}
\alpha=
\left(\begin{array}{c}
{\alpha_1}\\
\vdots\\
{\alpha_\texttt{g}}
\end{array}\right),
\quad
\eta_\alpha=
\left(\begin{array}{c}
\eta_{\alpha_1}\\
\vdots\\
\eta_{\alpha_\texttt{g}}
\end{array}\right),
\quad
*\eta_\alpha^T=
\left(\begin{array}{ccc}
*\eta_{\alpha_1}&
\dots &
*\eta_{\alpha_\texttt{g}}
\end{array}\right).
\end{equation}
In such a notation (\ref{etaalpha}), (\ref{etabeta}) become the single matrix equation
\begin{equation}\label{periods}
\left( \begin{array}{cc}
-\oint_{\alpha} \eta_{\beta}^T& -\oint_{\beta} \eta_{\beta}^T\\
\oint_{\alpha} \eta_{\alpha}^T & \oint_{\beta} \eta_{\alpha}^T 
\end{array} \right)
=\left( \begin{array}{cc}
                                            I & 0  \\
                                             0 & I  \\
\end{array} \right).
\end{equation}

The conjugate periods also form a period matrix, which we write as
\begin{equation}\label{PRRQ}
\left( \begin{array}{cc}
                                            P & R  \\
                                              R^T & Q  \\
\end{array} \right)=
\left( \begin{array}{cc} 
                                             - \oint_{\beta}*\eta_{\beta}^T & \oint_{\beta}*\eta_{\alpha}^T  \\
                                               \oint_{\alpha}*\eta_{\beta}^T & ( -\oint_{\alpha}*\eta_{\alpha}^T  \\
\end{array} \right).
\end{equation}
Spelled out in detail this is
$$
P_{kj}=-\oint_{\beta_k}\eta_{\beta_j}, \quad Q_{kj}=-\oint_{\alpha_k}\eta_{\alpha_j}, \quad R_{kj}=\oint_{\beta_k}\eta_{\alpha_j}.
$$
Using the general identity (see Proposition~III.2.3 in \cite{Farkas-Kra-1992})
\begin{equation}\label{general identity}
\int_M \sigma\wedge\tau=\sum_{j=1}^\texttt{g} \Big( \int_{\alpha_j}\sigma \int_{\beta_j}\tau -\int_{\alpha_j}\tau\int_{\beta_j}\sigma\Big),
\end{equation}
which holds for arbitrary closed one-forms $\sigma$ and $\tau$ on $M$,
the period matrix takes the form of an energy tensor (kinetic energy in the fluid picture):
\begin{equation}\label{ointdU}
\left( \begin{array}{cc}
                                            P & R  \\
                                              R^T & Q  \\
\end{array} \right)
=
\left( \begin{array}{cc}
                                             \int_M \eta_{\beta}\wedge *\eta_{\beta}^T &  -\int_M \eta_{\beta}\wedge *\eta_{\alpha}^T  \\
                                              - \int_M \eta_{\alpha}\wedge*\eta_{\beta}^T &   \int_M \eta_{\alpha}\wedge *\eta_{\alpha}^T  \\
\end{array} \right).
\end{equation}

The matrices $P$ and $Q$ are symmetric and positive definite, like the block matrix in (\ref{PRRQ}) or (\ref{ointdU}) as a whole.
Since, on the other hand, $R$ is not symmetric we may carefully write it out as
$$
R=-\int_M \eta_\beta\wedge *\eta_\alpha^T
=\left(\begin{array}{ccccc}
\oint_{\beta_1}*\eta_{\alpha_1}& \oint_{\beta_1}*\eta_{\alpha_2}&\dots & \oint_{\beta_1}*\eta_{\alpha_{\texttt{g}}}\\
\oint_{\beta_2}*\eta_{\alpha_1}& \oint_{\beta_2}*\eta_{\alpha_2}&\dots & \oint_{\beta_2}*\eta_{\alpha_\texttt{g}}\\
\vdots &\vdots&\dots &\vdots\\
\oint_{\beta_\texttt{g}} *\eta_{\alpha_1}& \oint_{\beta_\texttt{g}}*\eta_{\alpha_2}&\dots & \oint_{\beta_\texttt{g}}*\eta_{\alpha_\texttt{g}}
\end{array}\right).
$$

As mentioned, the column matrix consisting of $\{-\eta_\beta, \eta_\alpha\}$ defines a basis of the harmonic forms.
Another basis is provided by the corresponding Hodge starred row matrix $\{-*\eta_\beta, *\eta_\alpha\}$.

\begin{lemma}\label{lem:PQR}
The two bases $\{  -\eta_{\beta},  \eta_{\alpha}\}$ and $\{  -*\eta_{\beta},   *\eta_{\alpha}\}$ are related by 
\begin{equation}\label{RQPR} 
\left( \begin{array}{cc}
                                         -  *\eta_{\beta}  \\
                                            *\eta_{\alpha} \\
\end{array} \right)=
\left( \begin{array}{cc}
                                            -R & P  \\
                                              -Q & R^T \\
\end{array} \right)
\left( \begin{array}{cc}
                                           -   \eta_{\beta}  \\
                                              \eta_{\alpha} \\                                            
\end{array} \right).
\end{equation}
In addition we have
\begin{equation}\label{RPQR2}
\left( \begin{array}{cc}
                                            -R & P  \\
                                              -Q & R^T \\
\end{array} \right)^2
=-\left( \begin{array}{cc}
                                            I & 0  \\
                                              0 & I \\
\end{array} \right).
\end{equation}
Equivalently,  $RP$ and $QR$ are symmetric matrices and 
\begin{equation}\label{PQIR}
PQ=I+R^2.
\end{equation}
\end{lemma}

\begin{proof}
One checks, using (\ref{etaalpha}), (\ref{etabeta}), (\ref{ointdU}), 
that the two members in (\ref{RQPR}) have the same periods with respect to the homology basis  $\{\alpha_j, \beta_j\}$.
The identity (\ref{RPQR2}) is a consequence of the fact that the Hodge star acting twice on a one-form gives the same
one-form back with a minus sign. See also \cite{Farkas-Kra-1992} in this respect.
\end{proof}


\subsection{Holomorphic one-forms}\label{sec:capacity matrices}

The space of holomorphic one-forms on $M$, the {\it Abelian differentials of the first kind},
has complex dimension $\texttt{g}$, with one basis consisting of the differentials in (\ref{omegaeta})
for the cycles $\gamma=\beta_1,\dots, \beta_{\texttt{g}}$, another one with those for
$\gamma=\alpha_1,\dots,\alpha_\texttt{g}$. 
In terms of the notation in Definition~\ref{def:omegagamma} we set up the following two bases
of holomorphic one-forms, written in the form of column vectors, or  ($\texttt{g}\times 1$)-matrices.
\begin{definition} 
Two different bases of the space of holomorphic differentials on $M$ are
\begin{align}\label{omegabeta} 
-\omega_\beta&=-(\eta_\beta+\I*\eta_\beta),\\
\label{omegaalpha}
\omega_\alpha&=\eta_\alpha+\I *\eta_\alpha.
\end{align}
\end{definition}

The relation between the two bases can be figured out from Lemma~\ref{lem:PQR}, and comes out as follows. 
\begin{lemma}\label{lem:UQRU}
The bases $\omega_\beta$ and $\omega_\alpha$ are related according to
\begin{align*}
-\omega_\beta&=Q^{-1}(R^T+\I I )\,\omega_\alpha,\\
\qquad\omega_\alpha&=P^{-1}(R-\I I)(-\omega_\beta).
\end{align*}
\end{lemma}

\begin{proof}
A clue for the appearance the linear operators in the lemma comes from
factorization of the identity (\ref{RPQR2}) when written in the form with all terms moved to one side.
Using that the right member of (\ref{PQIR}) factorizes as $I+R^2=(I+\I R)(I-\I R)$ the identity (\ref{RPQR2})
then can be  expressed as
$$
\left( \begin{array}{cc}
                                            -\I I-R & P  \\
                                              -Q & -\I I+R^T \\
\end{array} \right)
\left( \begin{array}{cc}
                                          \I I  -R & P  \\
                                              -Q &\I I+ R^T \\
\end{array} \right)
=
\left( \begin{array}{cc}
                                            0 &0 \\
                                             0  &0 
\end{array} \right).
$$
Obviously the matrices in the left member have to be singular, 
in fact each of them has rank $\texttt{g}$ (half of the maximal rank), and the holomorphic basis 
$\{-\omega_\beta, \omega_\alpha\}$ then represents the null space
of the matrix in the second factor:
$$
\left( \begin{array}{cc}
                                          \I I  -R & P  \\
                                              -Q &\I I+ R^T \\
\end{array} \right)
\left( \begin{array}{cc}
                                     -\omega_\beta \\
                                      \omega_\alpha \\
\end{array} \right)
=
\left( \begin{array}{c}
                                            0  \\
                                             0   \\
\end{array} \right).
$$
This explains the first equation in the lemma. The second equation is explained similarly.

So much for the statements of the lemma.  The actual proof of it is obtained by writing the statements as
$$
Q\omega_\beta+(R^T+\I I )\,\omega_\alpha=0,
\quad P\omega_\alpha+(R-\I I)\omega_\beta=0,
$$
then decompose into real and imaginary parts, and finally use Lemma~\ref{lem:PQR}.
We omit the simple details.
\end{proof}


\subsection{Properties of Bergman and Schiffer kernels}\label{sec:Bergman and Schiffer}

If $f(z)dz$ is a holomorphic one-form on $M$ then, for any $a\in M$,
$$
0=\int_M d\big(f(z)G(z,a)dz\big)=-\int_M f(z)dz\wedge \frac{\partial G(z,a)}{\partial \bar{z}}d\bar{z}.
$$
The singularity of $G(z,a)$ at $z=a$ causes no problem here because it is quite mild. The
above expression is integrable, and if one wants to remove a small disk around the singularity one can 
do so and let the radius tend to zero without getting additional contributions.

Taking the derivatives with respect to $a$ and $\bar{a}$ above gives
\begin{equation}\label{Kprincipal}
\int_M f(z)dz \wedge\overline{\frac{\partial^2 G(z,a)}{\partial z\partial \bar{a}}dzd\bar{a}}=0,
\end{equation}
\begin{equation}\label{Lprincipal}
\int_M f(z)dz \wedge\overline{\frac{\partial^2 G(z,a)}{\partial z\partial {a}}dzd{a}}=0.
\end{equation}
These equalities should be subject to careful interpretation because now the singularities are less innocent.
For (\ref{Kprincipal}), the second factor is (up to a factor and a conjugation) the 
Bergman kernel together with a Dirac
distribution, see (\ref{K}), (\ref{HGza}). Therefore, evaluation of the identity (\ref{Kprincipal})
leads to the {\it reproducing property} of the Bergman kernel:
\begin{equation}\label{Kreproducing}
\frac{\I}{2}\int_M f(z)dz \wedge \overline{K(z,a)dzd\bar{a}}=f(a)da,
\end{equation}
holding for any holomorphic one-form $f(z)dz$ in $M$.
The factor $\I/2$ appears because the definition of the Bergman kernel is adapted to its use for planar domains,
where the area measure relates to complex differentials as in Section~\ref{sec:notations}.
 
For (\ref{Lprincipal}), the second factor is (up to a constant) the Schiffer kernel, as introduced in (\ref{GHza}), (\ref{L}),
so the identity can be written
\begin{equation}\label{Lprincipal1}
\int_M f(z)dz \wedge \overline{L(z,a)dz}=0. 
\end{equation}
The singularity is significant in the sense that the second factor is not
absolutely integrable. But the integral can still be evaluated as a principal value
integral with no extra contribution from the point $z=a$ itself. 
See discussions in \cite{Schiffer-Spencer-1954}.

In a certain sense the properties (\ref{Kreproducing}) and (\ref{Lprincipal1}) replace period requirements
on $K(z,a)dzd\bar{a}$ and $L(z,a)dzda$. Still one can see from (\ref{K}), (\ref{L}) that
the Bergman and Schiffer kernels can alternatively be characterized by having purely imaginary periods over all cycles,
this because $G(z,a)$ and $H(z,a)$ are real-valued functions. Compare the discussion after (\ref{upsilon}).

Next we expand the Bergman kernel along the natural bases of holomorphic differentials.
\begin{proposition}\label{prop:GUU} 
In terms of the matrices $P$, $Q$ we have 
\begin{align}\label{KPQ} 
K(z,a)dzd\bar{a}
&=\sum_{k,j=1}^\texttt{g} (P^{-1})_{kj}\omega_{\beta_k}(z)\overline{\omega_{\beta_j}(a)}\\
&=\sum_{k,j=1}^\texttt{g} (Q^{-1})_{kj}\omega_{\alpha_k}(z)\overline{\omega_{\alpha_j}(a)}.
\end{align}
\end{proposition}

\begin{proof}
Since $K(z,a)dzda$ is holomorphic in $z$,
anti-holomorphic in $a$, there must be expansions of the form in the lemma, for some choice of coefficients.
It only remains to check that the coefficients in the proposition are the right ones. 

On keeping $a$ fixed this is confirmed by showing that either all integrals
$\oint_{\alpha_j}dz$, or all integrals $\oint_{\beta_j}dz$, 
come out the same when applied to the left and right members above. 
If one of these sets of integrals come out correctly, then the other set must too, because an everywhere holomorphic 
differential on $M$ is determined by its periods around any one of these sets of cycles.

The check is indeed straight-forward. Combining (\ref{periods}) and (\ref{PRRQ}) we have  
$$
\oint_\alpha \omega_{\alpha}^T =-\I Q, \quad \oint_\alpha \omega_\beta^T =\I R^T -I,
$$
$$
\oint_\beta \omega_\alpha^T =\I R+I, \quad \oint_\beta \omega_\beta^T =-\I P.
$$
This gives, in view of (\ref{omegaeta}) and with integration with respect to $z$, 
$$
\oint_{\alpha}\omega_\alpha(z)^T Q^{-1}\overline{\omega_\alpha(a)} 
=-\I Q^T Q^{-1}\overline{\omega_\alpha(a)}=-\I\overline{\omega_\alpha(a)}  
= \oint_{\alpha}K(z,a)dzd\bar{a}, 
$$
$$ 
\oint_{\beta}\omega_{\beta}^T(z)P^{-1}\overline{\omega_\beta(a)} 
=-\I P^T P^{-1}\overline{\omega_\beta(a)}=-\I\overline{\omega_\beta(a)}  
=\oint_{\beta}K(z,a)dzd\bar{a},
$$ 
as required.
\end{proof}
 
Since the matrices $P$ and $Q$ (and their inverses)
are positive definite one can easily orthogonalize the bases for holomorphic one-forms
and write (\ref{KPQ}) on the form
\begin{equation}\label{Kee}
K(z,a)dz  d\bar{a}=\sum_{i=1}^\texttt{g} e_i(z) \overline{e_i(a)},
\end{equation}
where the $e_j$ are holomorphic one-forms. Slightly more precisely, $P^{-1}$
has a positive square root $\sqrt{P^{-1}}$ so that
$(P^{-1})_{kj}=\sum_i (\sqrt{P^{-1}})_{ki}(\sqrt{P^{-1}})_{ij}$, and in terms of this (\ref{KPQ})
can be written
$$
K(z,a)dzd\bar{a}
=\sum_{i,k,j=1}^\texttt{g} (\sqrt{P^{-1}})_{ki}\omega_{\beta_k}(z)\overline{(\sqrt{P^{-1}})_{ji}\omega_{\beta_j}(a)},
$$
which is of the form (\ref{Kee}) with $e_i=\sum_{k=1}^\texttt{g}(\sqrt{P^{-1}})_{ki}\omega_{\beta_k}$.

\begin{remark}
We shall later apply (\ref{Kee}) in the case that the Riemann surface is the double of a planar domain,
and we shall then use it with an additional factor $2$ in the left member. This appears because the
orthogonality is to be adapted to the planar domain (half of the closed surface). See Section~\ref{sec:orthogonal}.
\end{remark}
 
Having discussed the Abelian differentials of the first and second kind, 
see (\ref{GHza}), we may also
mention {\it Abelian differentials of the third kind}, that is, those holomorphic one-forms
which have residue poles. A basic such differential, with residue $\pm 1$ poles at points $a$ and $b$, 
can most easily be obtained using the four variable potential in (\ref{VG}) as
\begin{equation}\label{upsilon}
\upsilon_{a-b}(z)=-{4\pi}\Big(\frac{\partial G(z,a)}{\partial z}dz-\frac{\partial G(z,b)}{\partial z}dz\Big)
=\frac{dz}{z-a}-\frac{dz}{z-b}+{\rm regular}.
\end{equation}
Here one can think of $a-b$ as a polar divisor.  

The differential $\upsilon_{a-b}$ is besides this divisor
characterized by having purely imaginary periods. Indeed,
$\upsilon_{a-b}= -2\pi (dV+\I *dV)$ with $V$ as in (\ref{VG}), so the real part of $\upsilon_{a-b}$ is exact.
By adjusting with linear combinations of the holomorphic
one-forms (\ref{omegabeta}), (\ref{omegaalpha}) one can construct bases of Abelian differentials normalized 
instead to have all $\alpha$-periods, alternatively all $\beta$-periods, being zero. Such normalizations have
the advantage that the dependence on the points $a$ and $b$ in the divisor becomes holomorphic. 
Indeed, having vanishing periods along certain cycles is a ``holomorphic'' kind of restriction, while requiring
just the real parts to vanish breaks holomorphicity. This affects only the regular term in (\ref{upsilon}).
See equation (\ref{exupsilonab}) in Example~\ref{ex:kernels2} below for an explicit example. Compare also
the decomposable differential defined by (\ref{LLL}) below (and more generally in \cite{Hawley-Schiffer-1966}).

On taking derivatives of (\ref{upsilon}) with respect to $a$ and $\bar{a}$ one meets again the Schiffer and Bergman kernels, for example
$$
\frac{\partial}{\partial a}\upsilon_{a-b}(z)da=-4\pi \frac{\partial^2 G(z,a)}{\partial z \partial a}dzda=\pi L(z,a)dzda,
$$
where now the dependence on all variables is holomorphic.
 

\subsection{The Robin function}\label{sec:Robin} 

Returning to the Green function in (\ref{ddGdelta}) and the decomposition (\ref{GHlog}),
we are, in the same vein as in (\ref{GlogH}) and Lemma~\ref{lem:hgammadelta}, 
led to 
\begin{definition}\label{def:coordinaterobin}
The {\it coordinate Robin function} is
$$
\gamma(a)=h_0(a)=H(a,a).
$$
\end{definition}
From another point of view this is a {\it capacity function} (or {\it capacity}) in the sense of \cite{Sario-Oikawa-1969}, 
and it is actually not a true function because it  transforms under conformal changes of coordinates in such a way that
$ds=e^{-\gamma(z)}|dz|$ is an invariantly defined metric. 

In the presence of an independent metric $ds=\lambda(z)|dz|$ and on adding $\log \lambda$ to $\gamma$,
one gets a true function,  the {\it Robin function},
\begin{equation}\label{Rh}
R(z)=\frac{1}{2\pi}\big(\gamma(z)+\log \lambda(z)\big).
\end{equation}
This has been used in  \cite{Koiller-Boatto-2009, Grotta-Ragazzo-Barros-Viglioni-2017, Grotta-Ragazzo-2024}, for example,
and it is the constant term when the Green function is expanded as
\begin{equation}\label{GlogdR}
G(z,a)=-\frac{1}{2\pi} \log d(z,a)+R(a)+ \mathcal{O}(d(z,a))\quad (z\to a)
\end{equation}
with $d(z,a)$ denoting the metric distance between $z$ and $a$.

In the Taylor expansion (\ref{HTaylor1}), which is written in terms of a local coordinate, we have
$$
\frac{\partial^2 H(z,a)}{\partial z\partial\bar{z}}\Big|_{z=a}=h_{11}(a),
$$
$$
\frac{\partial^2 H(z,a)}{\partial z\partial\bar{a}}\Big|_{z=a}=\frac{1}{2}\frac{\partial h_1(a)}{\partial \bar{a}}-h_{11}(a).
$$
Since $h_1(a)=\partial h_0(a)/\partial a$ (like in the planar case, see Lemma~\ref{lem:hgammadelta})
we conclude, using also (\ref{Hlambda}), that 
$$
\Delta \gamma(a)=4\frac{\partial^2h_0(a)}{\partial a\partial\bar{a}}=8h_{11}(a)
+8\frac{\partial^2 H(z,a)}{\partial z\partial\bar{a}}\Big|_{z=a}
=\frac{4\pi}{V}\lambda(a)^2-4\pi K(a,a),
$$
$K(z,a)$ referring to the Bergman kernel (\ref{K}). Introducing the notation
$$
K_{\rm diag}(a)=K(a,a)
$$
for the Bergman kernel restricted to the diagonal we can write
\begin{equation}\label{Deltagammadouble}
\Delta \gamma=4\pi\Big(\frac{\lambda^2}{V}- K_{\rm diag}\Big),
\end{equation}
and similarly
\begin{equation}\label{DeltaR}
\Delta R =\Big(\frac{2}{V}-\frac{\kappa}{2\pi}\Big) \lambda^2-2K_{\rm diag},
\end{equation}
where we used the definition (\ref{kappageneral})  of Gaussian curvature. Here one can insert
the expressions in Proposition~\ref{prop:GUU}, or (\ref{Kee}), to obtain explicit formulas
exactly like those obtained by Okikiolu, Steiner, and Grotta-Ragazzo, see
 \cite{Okikiolu-2009, Steiner-2005, Grotta-Ragazzo-2024}. 

We remark that also the Bergman kernel itself defines a metric, the {\it Bergman metric},
\begin{equation}\label{Bergman metric}
ds^2=K(z,z)|dz|^2=K_{\rm diag}(z)|dz|^2,
\end{equation}
in general not identical with the previously mentioned metrics.
See (\ref{inequalities}) below for comparisons and for some more metrics,
and \cite{Jost-2001, Krantz-2004, Grotta-Ragazzo-2024} for some further metrics. 

\begin{remark}\label{rem:robin}
The term ``Robin function'' is also used in  other meanings, related to mixed boundary value problems for planar
domains. See \cite {Duren-Schiffer-1991}.
\end{remark}


\subsection{Examples in genus zero and one}\label{sec:examples}

\subsubsection{The sphere}\label{sec:sphere}

First a brief treatment of the genus zero case. For the sphere of radius one, with the metric 
$$
ds^2=\frac{4|dz|^2}{(1+|z|^2)^2},
$$
obtained via stereographic projection, the monopole Green function is
$$
G(z,a)=-\frac{1}{4\pi}\Big(\log \frac{|z-a|^2}{(1+|z|^2)(1+|a|^2)}+1\Big).
$$
This gives the Bergman and Schiffer kernels
$$
K(z,a)=0, \quad L(z,a)dzda=\frac{dzda}{\pi (z-a)^2}.
$$

As for the coefficients in the expansion (\ref{GHlog}), (\ref{HTaylor1}) of the Green function we have
$$
h_0(a)=\log (1+|a|^2)-\frac{1}{2},\quad
h_1(a)=\frac{\bar{a}}{1+|a|^2},\quad
h_2(a)=\frac{\bar{a}^2}{2(1+|a|^2)^2},
$$
$$
e^{-2h_0(a)}=\frac{e}{(1+|a|^2)^2},\quad
\lambda(a)^2=8h_{11}(a)= \frac{4}{(1+|a|^2)^2}. 
$$


\begin{figure}
\begin{center}
\includegraphics[scale=0.7]{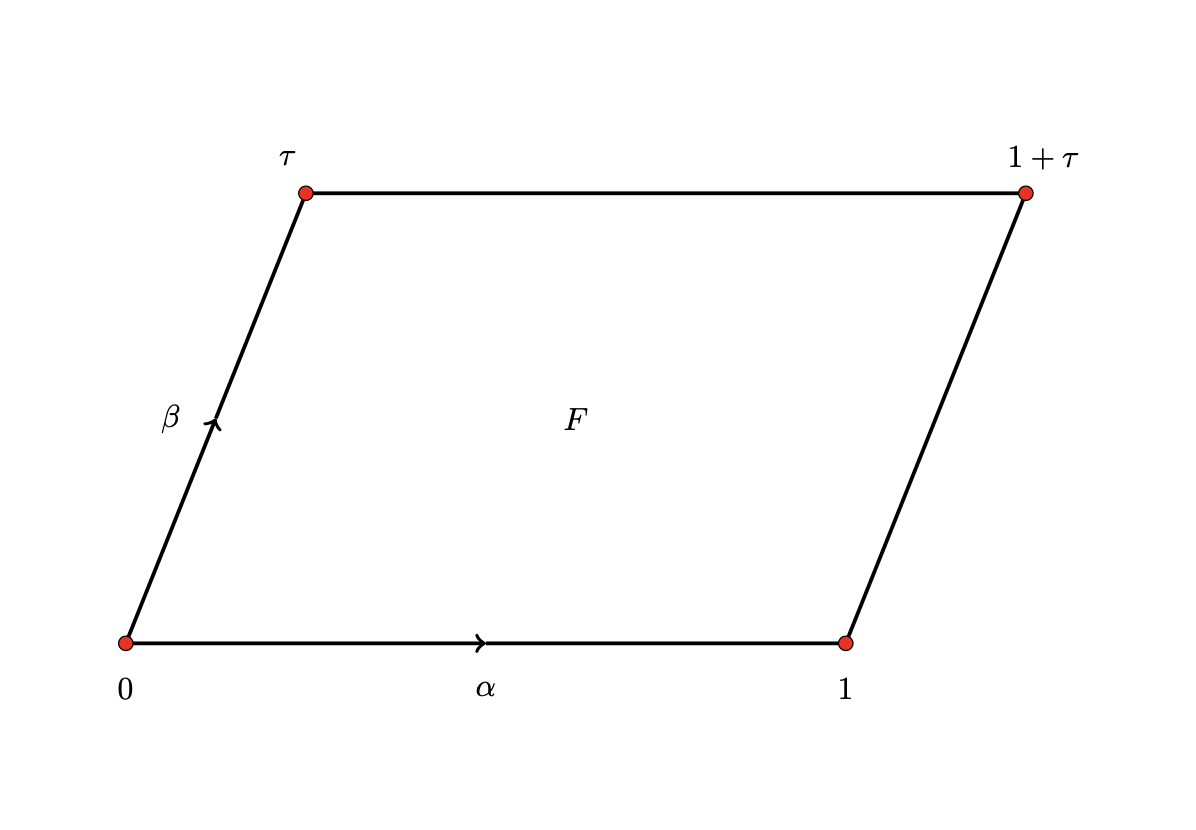}
\end{center}
\caption{Period parallelogram representing a torus. Notations as in text.}
\label{fig:torus1}
\end{figure}


\subsubsection{The torus}\label{sec:torus}

Next we turn to genus one.
Let $M$ be the flat torus represented as $M=\C/(\Z\times \tau \Z)$, where $\im \tau>0$, and with the Euclidean metric
$ds^2=dx^2+dy^2$. It suffices to work in the fundamental domain $F$ generated by the vectors $1$ and $\tau$. The area of $F$
is $V=\im \tau$, and as canonical homology basis
we choose $\alpha$ and $\beta$ (no indices are needed) to be the straight lines from the origin to $1$ and to $\tau$, 
respectively, see Figure~\ref{fig:torus1}.

The monopole Green function $G(z,a)$ is to satisfy
\begin{equation}\label{Gtorus}
-4\frac{\partial^2}{\partial z\partial \bar{z}} G(z,a) =\delta(z-a)-\frac{1}{\im\tau} \quad (z,a\in F), 
\end{equation}
with boundary conditions which allow for the natural periodic continuation of this relation.
In principle, $G(z,a)$ can be represented in terms of elliptic functions, but the expressions are somewhat complicated 
and do not add much of theoretical understanding. See however \cite{Courant-Hilbert-1968, Crowdy-Marshall-2007} for
formulas and series developments useful for numerical purposes.

The regular part $H(z,a)$ of the Green function is equally complicated, but some of its derivatives
are simple and can be obtained by other means. 
For example, the Schiffer kernel  (\ref{L}), (\ref{GHza}) is, because of the singularity structure, necessarily of the form
$$
L(z,a)dzda=-4\frac{\partial^2 G(z,a)}{\partial z \partial a}dzda=\frac{1}{\pi}\big( \wp(z-a)+C\big)dzda
$$
for some constant $C$,
 and where $\wp(z)$ is the Weierstrass $\wp$-function \cite{Ahlfors-1966}:
$$
\wp(z)=\frac{1}{z^2}+\sum\Big(\frac{1}{(z-m-n\tau)^2}-\frac{1}{(m+n\tau)^2}\Big),
$$
with summation over all $(m,n)\in\Z\times \Z \setminus \{(0,0)\}$.
 
The constant $C$ is to be determined so that 
\begin{equation}\label{C}
\int_M L(z,a)dzd\bar{z}=0,
\end{equation}
because this is what (\ref{Lprincipal1}) says on choosing $f(z)dz=dz$, up to a constant factor the only holomorphic
one-form on the torus. 

Let, in traditional notation, $\zeta(z)$ be a primitive function of minus $\wp(z)$, so that
$\zeta^\prime(z)=-\wp({z})$. Then, for any $a\in F$, $\zeta(z-a)$ is well-defined in $F$ and has a simple pole at $z=a$ with
residue one. However, $\zeta(z-a)$ is not doubly periodic, and so is not single-valued on $M$, but one can
still work with it within $F$ to obtain the following identities.

\begin{lemma}\label{lem:torus}
Defining the periods $\eta_1, \eta_2$ of $\wp(z)$ so that $\zeta(z+1)=\zeta(z)+\eta_1$, $\zeta(z+\tau)=\zeta(z)+\eta_2$ 
we have, for any $a\in F$, the identities
$$
\int_\alpha \wp(z-a)dz=-\eta_1, \quad \int_\beta \wp(z-a)dz=-\eta_2,
$$
$$
2\pi \I=\int_{\partial F}\zeta(z-a)dz=\eta_1 \tau-\eta_2,
$$
$$
-\int_F\wp(z-a)dzd\bar{z}=\int_{\partial F}\zeta(z-a)d\bar{z}=\bar{\tau}\eta_1-\eta_2.
$$
\end{lemma} 

The first relations are simply the definition of $\eta_1$ and $\eta_2$, and the second relation 
is an instance of the well-known {\it Legendre identity} \cite{Ahlfors-1966}. In the last relation the left member shall be
interpreted as a {\it principal value} integral (i.e., one removes a small disk $\D(a,\varepsilon)$ from $F$ and takes the
limit as $\varepsilon\to 0$).

\begin{proof}
We only prove the last identity (the proof of the Legendre identity is similar). 
The integrand in the first term has a singularity with leading part $1/(z-a)^2$, 
which can be handled by removing a small disk
around $a$ and letting the radius tend to zero. In the limit there will be no contribution,
so the singularity can just be ignored.

We therefore get, letting $t$ parametrize the various parts of $\partial F$,
$$
\int_F\wp({z-a})dzd\bar{z}=-\int_F \zeta^\prime(z-a) dzd\bar{z}=-\int_{\partial F} \zeta(z-a)d\bar{z} 
$$
$$
=-\int_0^1\zeta(t-a)dt-\int_0^1 \zeta(1+t\tau-a)\bar{\tau}dt
$$
$$
-\int_0^1\zeta(1+\tau-t-a)(-dt)-\int_0^1\zeta(\tau+(-\tau) t-a)(-\bar{\tau}dt)
$$
$$
=-\int_0^1\zeta(t-a)dt-\int_0^1 \zeta(1+t\tau-a)\bar{\tau}dt
$$
$$
+\int_0^1\zeta(t+\tau-a)dt+\int_0^1\zeta(t\tau -a)\bar{\tau}dt
=\eta_2-\bar{\tau}\eta_1.
$$
\end{proof}

We proceed with identifying the two basic harmonic differentials $\eta_\alpha$ and $\eta_\beta$
(not to be confused with the periods $\eta_1$, $\eta_2$ above).
They  are determined via the periodicity properties (\ref{etaalpha}), (\ref{etabeta}). This gives
$$
\eta_\alpha=\frac{1}{\im \tau}dy, \quad \eta_\beta=-dx+\frac{\re \tau}{\im \tau}dy,
$$
$$
*\eta_\alpha=-\frac{1}{\im \tau}dx, \quad *\eta_\beta=-\frac{\re \tau}{\im \tau}dx-dy,
$$
with the period matrix (see (\ref{PRRQ}) for the general expression) given by
\begin{equation*}
\left( \begin{array}{cc}
                                            P & R  \\
                                              R^T & Q  \\
\end{array} \right)=
\left( \begin{array}{cc} 
\frac{|\tau|^2}{\im\tau} & -\frac{\re \tau}{\im\tau}  
\vspace{2mm}\\
                                             -\frac{\re \tau}{\im\tau}   & \frac{1}{\im\tau} \\
\end{array} \right).
\end{equation*}
The two versions of a basic holomorphic one-form (see (\ref{omegabeta}), (\ref{omegaalpha})) are
$$
\omega_\alpha=-\frac{\I \,dz}{\im \tau},\quad 
\omega_\beta=-\frac{\I\bar{\tau}dz}{\im\tau},
$$
the relations in Lemma~\ref{lem:UQRU} thus being confirmed with
$$
Q^{-1}(R^T+\I I)= -\re\tau+\I\im\tau=-\bar{\tau},
$$
$$
P^{-1}(R-\I I)=-\frac{{\tau}}{|\tau^2|}=-\frac{1}{\bar{\tau}}.
$$

Turning to the Bergman kernel, this can be read off  from either of the two formulas in Proposition~\ref{prop:GUU},
for example the first one, then coming out as 
\begin{equation}\label{Ktorus}
K(z,a)dzd\bar{a}=\frac{\im \tau}{|\tau|^2}\cdot (-\frac{\I \bar{\tau}dz}{\im \tau})\cdot (+\frac{\I \tau d\bar{a}}{\im\tau})
=\frac{dzd\bar{a}}{\im\tau}.
\end{equation}

Here we can confirm that (see (\ref{omegaeta}))
$$
\omega_\alpha(a)=-\frac{\I da}{\im \tau}=-\I\oint_\alpha K(a,z)d\bar{z}da=\I \oint_\alpha L(z,a)dzda,
$$
$$
\omega_\beta(a)=-\frac{\I \bar{\tau}da}{\im \tau}=-\I\oint_\beta K(a,z)d\bar{z}da=\I \oint_\beta L(z,a)dzda.
$$

Combining the two identities in Lemma~\ref{lem:torus} and comparing with (\ref{C}) gives
$$
L(z,a)=\frac{1}{\pi}\big(\wp(z-a)+ \eta_1\big)-\frac{1}{ \im \tau}.
$$
We summarize as

\begin{proposition}\label{prop:Lwp}
The Bergman and Schiffer kernels, and the holomorphic one-forms, for the torus in a period parallelogram 
in $\C$ spanned by $\{1,\tau\}$, are
\begin{align*}
K(z,a)dzd\bar{a}&=\frac{dzd\bar{a}}{\im \tau},\\
L(z,a)dzda&=\frac{1}{\pi}\big(\wp(z-a)+ \eta_1\big)dzda-\frac{dzda}{ \im \tau},\\
\omega_\alpha(a)&=-\frac{\I da}{\im \tau},\\
\omega_\beta(a)&=-\frac{\I \bar{\tau}da}{\im \tau}.
\end{align*}
Here the cycles $\alpha, \beta$ correspond to the oriented segments $[0,1]$ and $[0,\tau]$,
respectively.
\end{proposition}


\section{Planar domains by Schottky double}\label{sec:planar domains}

\subsection{The Schottky double}\label{sec:Schottky}

Now we return to the case of a planar domain $\Omega\subset \C$, which we for simplicity assume is bounded, 
with boundary consisting of $\texttt{g}+1$ analytic curves. 
The latter assumption is actually very innocent: every domain of finite connectivity for which none of the boundary components
consists of just a single point is conformally equivalent  to a domain with analytic boundary, in fact even
to a domain bounded by perfect circles, see \cite{Koebe-1918, Nehari-1952}. This is often used for computational
purposes and for representation of multiply connected domains in terms of Schottky groups of reflection maps, as in
\cite{Crowdy-Marshall-2006}.

The Schottky double of a planar domain, first described in \cite{Schottky-1877},
is the compact Riemann surface $M=\Hat{\Omega}$ obtained by completing $\Omega$ with a ``backside'' $\tilde{\Omega}$
having the opposite conformal structure and glueing the two along their common boundary $\partial\Omega$. Thus 
$\Hat{\Omega}=\Omega \cup\partial\Omega\cup \tilde{\Omega}$ in a set theoretic sense, and the conformal structure becomes
smooth over $\partial\Omega$, as can be seen from well-known reflection principles (see further below). If $z$ is a point in $\Omega$,
then we denote by $\tilde{z}$ the corresponding (reflected) point on $\tilde{\Omega}$. The map $z\mapsto \tilde{z}$,
together with $\tilde{z}\mapsto z$, defines an anti-conformal involution $J:M\to M$ making $M=\Hat{\Omega}$ 
become a {\it symmetric Riemann surface}.
The boundary $\partial\Omega$ can be recovered as the set of fixed points of $J$, and the domains $\Omega$ and $\tilde{\Omega}$ are 
the two components of $M$ when this fixed point set is removed.

There also exist symmetric Riemann surfaces $(M,J)$ for which the fixed point set of $J$ does not divide $M$ into two
components. One example is the Riemann sphere with involution $z\mapsto -1/\bar{z}$. This involution has no fixed points at all.
Such involutions come up in the context of non-orientable Riemann surfaces, or Klein surfaces. See for example
\cite{Schiffer-Spencer-1954, Alling-Greenleaf-1971}, and in the context of vortex motion \cite{Balabanova-Montaldi-2022}.

In order to be relevant for physical problems, such as vortex motion, the Schottky double
need to be provided with a Riemannian metric, preferably one which is invariant with respect to $J$.
For the unit disk there are three metrics of constant curvature which come up naturally, namely
\begin{itemize}
\item the spherical metric $ds=\frac{2|dz|}{1+|z|^2}$,  
\item the Euclidean metric $ds=|dz|$,
\item  the hyperbolic metric $ds=\frac{2|dz|}{1-|z|^2}$.
\end{itemize}
They are here normalized so that the curvatures are $+1$, $0$, $-1$.  

Identifying the Schottky double of the disk with the (Riemann) sphere, the first metric above becomes the ordinary round metric 
on the sphere,
which is good in all respects. The Euclidean metric is of course good in itself, but when extended symmetrically to the double it
becomes singular on the boundary. In fact, all the curvature becomes concentrated there, the total Gaussian curvature of $4\pi$ 
(for any topological sphere) 
being uniformly distributed on the boundary circle of length $2\pi$. It is natural think in this case of the double as a ``pancake''.

The final case, with the hyperbolic metric, is 
however no good for the Schottky double. The metric is complete in itself, and this means that the boundary is
infinitely far away and cannot be incorporated.

For the future discussion we shall work with the Schottky double 
$M=\Hat{\Omega}$ with each of $\Omega$  and $\tilde{\Omega}$
being provided with the Euclidean metric. This gives the pancake metric on $M=\Hat{\Omega}$, namely
\begin{equation}\label{metricM}
ds=
\begin{cases}
|dz|, \quad z\in \Omega,\\
|d\tilde{z}|, \quad \tilde{z}\in \tilde{\Omega}.
\end{cases}
\end{equation}
To see how this behaves across
$\partial \Omega$ we need a holomorphic coordinate defined in a full neighborhood of $\partial\Omega$ in $M$. 
A natural candidate can be
defined in terms of the {\it Schwarz function} $S(z)$ for $\partial\Omega$, 
a function which is defined by its properties of being holomorphic 
in a two-sided neighborhood of $\partial\Omega$ in the complex plane and by satisfying
\begin{equation}\label{Schwarz}
S(z)=\bar{z} \quad \text{on }\partial\Omega.
\end{equation} 
See \cite{Davis-1974, Shapiro-1992} for details about $S(z)$. We remark that 
$z\mapsto \overline{S(z)}$ is the locally defined anti-conformal reflection map in $\partial\Omega\subset\C$,
thus
\begin{equation}\label{SSz}
\overline{S(\overline{S(z)})}=z.
\end{equation}
The derivative satisfies
\begin{equation}\label{SprimeT}
S'(z)=T(z)^{-2},
\end{equation}
where $T(z)$ is the positively oriented and holomorphically extended unit tangent vector on $\partial\Omega$. 

The complex coordinate $z$ in $\Omega$ extends, as a holomorphic function, to a full neighborhood of $\Omega\cup\partial\Omega$,
both when this neighborhood is considered as a subset of $\C$ and when it is considered as a subset of $M$.
In the first case this is obvious, while in the second case it depends on $\partial\Omega$ being analytic. 
In terms of the Schwarz function this second extension is given by
\begin{equation}\label{extension}
\phi(z)=
\begin{cases}
z\quad &\text{for }z\in\Omega\cup \partial\Omega,\\
\overline{S({\tilde{z}})}\quad&\text{for }\tilde{z}\in \tilde{\Omega},\text{ close to }\partial\Omega,
\end{cases}
\end{equation} 
where, in the latter expression, the variable $\tilde{z}$ is to be interpreted as a complex number.

The function $\phi(z)$ defined by (\ref{extension}) is a holomorphic coordinate on a neighborhood of $\Omega\cup\partial\Omega$
in the Schottky double $M$. It is to be combined with a corresponding coordinate on $\tilde{\Omega}\cup\partial\Omega$,
and that can similarly be taken to be
\begin{equation}\label{extensionbackside}
\tilde{\phi}(\tilde{z})=
\begin{cases}
\overline{\tilde{z}}\quad &\text{for }\tilde{z}\in\tilde{\Omega}\cup \partial\Omega,\\
{S({{z}})}\quad&\text{for }{z}\in {\Omega},\text{ close to }\partial\Omega.
\end{cases}
\end{equation} 
The two functions $\phi$ and $\tilde{\phi}$ then make up a complex analytic  atlas on the Schottky double, and using (\ref{SSz})
one finds that the transition function is the Schwarz function itself:
\begin{equation}\label{Sphi}
S=\tilde{\phi}\circ \phi^{-1}.
\end{equation}

When the metric on $M$ is expressed in the coordinate (\ref{extension}) it becomes
\begin{equation}\label{coordinate metric}
ds=
\begin{cases}
|dz|\quad &\text{for }z\in\Omega\cup \partial\Omega,\\
|S^\prime({z})||d{z}|\quad&\text{for }{z}\in \C\setminus\overline{\Omega},\text{ close to }\partial\Omega.
\end{cases}
\end{equation}
In the second case we have $\tilde{z}=\overline{S(z)}$, and so $|d\tilde{z}|=|S^\prime({z})||d{z}|$. 
Therefore (\ref{coordinate metric}) is consistent with (\ref{metricM}).
And we see from the coordinate representation (\ref{coordinate metric}) that the metric is only Lipschitz continuous across $\partial\Omega$. 
This is the best one can expect. A function $f$ in general is called {\it Lipschitz continuous} if an exact modulus of continuity
$$
|f(z)-f(w)|\leq C |z-w|
$$
holds for $z$ close to $w$ and for some fixed constant $C$.


\subsection{Electrostatic versus hydrodynamic Green functions}\label{sec:electro-hydro}

We shall now consider the Schottky double in hydrodynamical and electrostatic contexts.
Hodge decomposition of a two-form $\omega$, representing for example a vorticity distribution, 
becomes explicitly (see (\ref{omegaddG}))
\begin{equation*}\label{Hodgeomega1}
\omega =-d*dG^{\omega} + c\cdot {{\rm vol}}, \quad c=-\frac{1}{V}\int_M \omega.
\end{equation*}
If $\omega$ is antisymmetric with respect to $J$, then $c=0$, and in particular if 
$$
\omega=\delta_a -\delta_{J(a)} \quad (a\in \Omega),
$$
describing a vortex pair with vortices of opposite unit strengths at $a\in\Omega$ and $\tilde{a}=J(a)\in \tilde\Omega$,
then the corresponding Green function $G^\omega(z)$  is simply the odd
extension to the Schottky double of the ordinary Dirichlet Green function for  $\Omega$ with zero
boundary values. 
A suggestive terminology, inspired by  \cite{Cohn-1980},  is to call this Dirichlet Green function 
the {\it electrostatic Green function} and denote it $G_{\rm electro}(z,a)$. Thus
\begin{equation}\label{Gelectro}
G_{\rm electro}(z,a)=G^{\delta_a-\delta_{\tilde{a}}}(z)\quad (z\in\Omega).
\end{equation}

The monopole Green function (\ref{ddGdelta}) on the double, that is $G(z,a)=G^{\delta_a}(z)$, may in this context be denoted
$G_{\rm double}(z,a)$. This gives, as a reformulation of (\ref{Gelectro}),
\begin{equation}\label{GGG}
G_{\rm electro}(z,a)= G_{\rm double}(z,a)-G_{\rm double} (z,J(a)).
\end{equation}
As a further observation, we notice that the differential of the function in (\ref{Gelectro}), being exact on the Schottky double,
has no periods. It follows that the analytic completion of it has only imaginary periods. In addition it has
polar divisor $a-J(a)$, the same as the Abelian differential $\upsilon_{a-J(a)}$ introduced in (\ref{upsilon}). Therefore
\begin{equation}\label{dGdGnu}
dG_{\rm electro}(z,a)+\I*dG_{\rm electro}(z,a)=-\frac{1}{2\pi}\upsilon_{a-J(a)}(z).
\end{equation}

As previously mentioned, the differential  $\upsilon_{a-b}$ does not depend analytically
on $a$. Still the combination used in (\ref{dGdGnu}), namely $\upsilon_{a-J(a)}$, 
is analytically with respect $a$. In fact, it equals that fundamental Abelian differential
which has zero periods around those cycles on the Schottky double which go across the surface and
returns on the back-side (the ``$\alpha$-cycles'', see below).
The simple details of explanation can be found in  Lemma~2.3 of \cite{Gustafsson-Sebbar-2012}, and in the doubly connected 
case it becomes apparent, as is discussed after equation (\ref{exupsilonab}) below.

In hydrodynamic contexts the Green function has the role of being a stream function, like $\psi$  in (\ref{psi}), 
but in the case that the domain $\Omega$ is multiply connected it is actually another Green function which is more
relevant, namely the {\it hydrodynamic Green function}, to be denoted $G_{\rm hydro}(z,a)$,
again inspired by \cite{Cohn-1980}. Instead of having vanishing Dirichlet data on the boundary
it is only required to have locally constant boundary values, and is then determined by preassigned
circulations around the boundary components, this in order to 
automatically adapt to the Kelvin/Helmholtz law of conservation of circulations.
For this reason the hydrodynamic Green function actually depends on a number of parameters, the circulations.
At least in the case that all circulations around the inner boundary components are zero,
the hydrodynamic Green function can be traced back to \cite{Koebe-1916, Lin-1943}. 
See also \cite{Crowdy-Marshall-2005, Crowdy-Marshall-2006, Crowdy-Marshall-2007}, in particular for computational aspects.
More general circulations are treated in, for example, \cite{Cohn-1980, Gustafsson-1979, Flucher-Gustafsson-1997, Flucher-1999}.


\begin{figure}
\begin{center}
\includegraphics[scale=0.7]{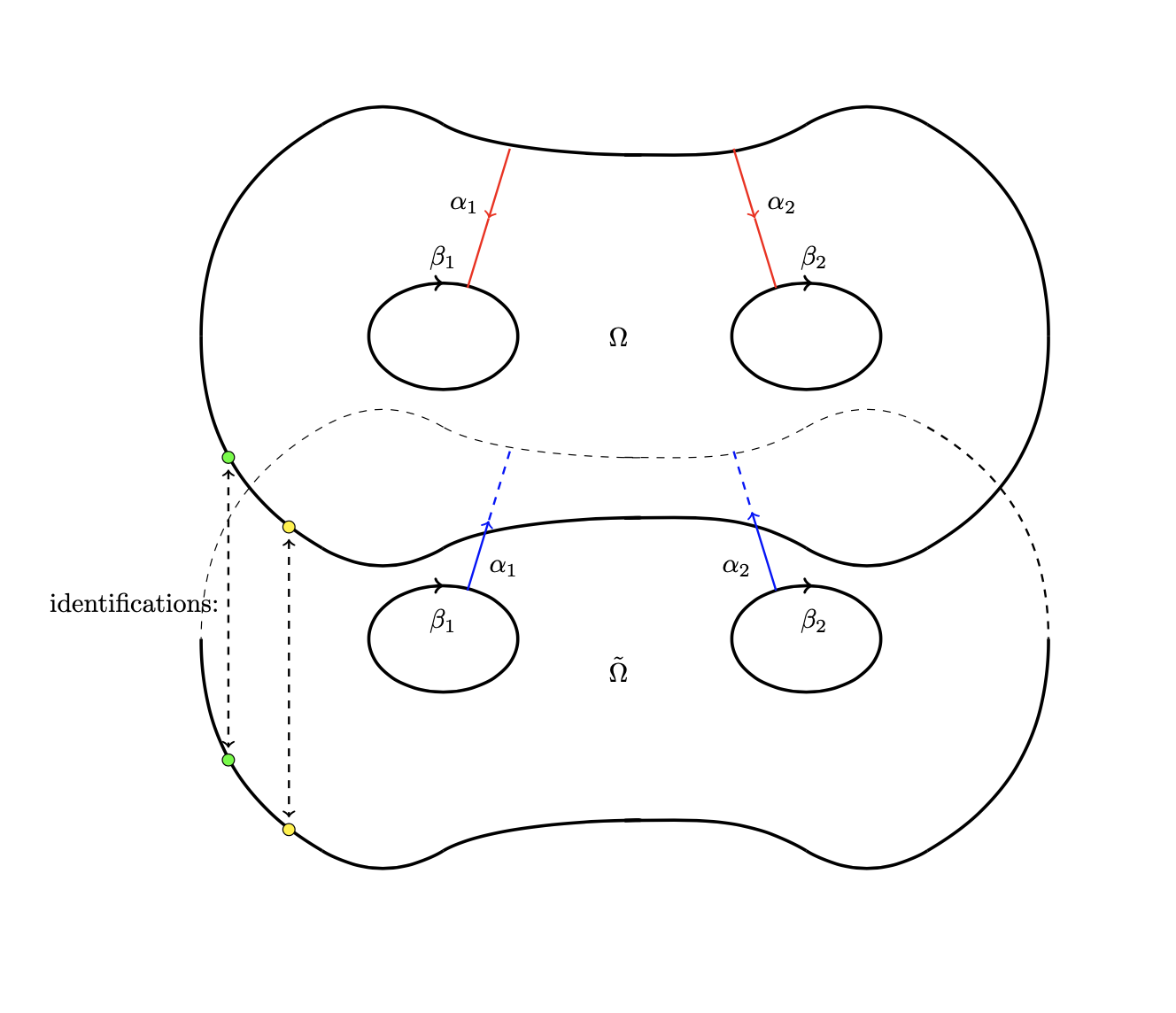}
\end{center}
\caption{Illustration of Schotttky double, with homology basis, for a domain of connectivity $\texttt{g}+1=3$.}
\label{fig:Schottky}
\end{figure} 


In order to describe the period requirements for the hydrodynamic Green function we first have to fix, 
for Schottky double $M=\hat{\Omega}$, an appropriate homology basis 
$\{\alpha_1, \dots, \alpha_{\texttt{g}},\beta_1,\dots, \beta_\texttt{g}\}$.
We shall choose it so that the curves $\beta_j$,
$j=1,\dots,\texttt{g}$, coincide with the inner components $(\partial\Omega)_j$ of $\partial\Omega$, 
and so that each curve $\alpha_j$ goes from the outer component $(\partial\Omega)_0$
through $\Omega$ to $(\partial\Omega)_j$, and then back again along the same track
on the backside. See Figure~\ref{fig:Schottky}. We also introduce the {\it harmonic measures}
$u_j$, $j=1,\dots,\texttt{g}$, here defined to be those harmonic functions in $\Omega$ which have boundary values 
$$
u_j(z)=
\begin{cases}
1, \quad z\in (\partial\Omega)_j,\\
0, \quad z\in \partial\Omega\setminus (\partial\Omega)_j.
\end{cases}
$$

These functions (or potentials) are traditionally called harmonic measures because their normal derivatives represent, 
after integration, natural measures (for example electrostatic charges) on the boundary, in the same sense as the normal 
derivative of the Green function represents a harmonic measure as in Definition~\ref{def:harmonic measure}
and equation (\ref{dGdn}).

In $\Omega$ we have
\begin{equation}\label{dudU}
du_j=-2dU_{\beta_j}=-2\eta_{\beta_j},
\end{equation}
where the $\eta_{\beta_j}$ is one of the basic harmonic forms on the Schottky double, see (\ref{deltaeta}), (\ref{etaalpha}), (\ref{etabeta}). 
(The one-form $\eta_{\beta_j}$ is not to be confused with the measure $\eta_a$ in equation (\ref{dGdn}), even though
the letter $\eta$ in both cases refer to something ``harmonic''.)
Since $u_j$ is single-valued on $\Omega$, so is $U_{\beta_j}$, and integration of (\ref{dudU}) gives
\begin{equation}\label{Uup}
2U_{\beta_j}=-u_j+p_j \quad \text{in \,}\Omega,
\end{equation}
where the $p_j$ are constants of integration. These will turn out to be the circulations.

Recall from (\ref{periods}) the period matrices $P=(P_{kj})$, $Q=(Q_{kj})$, $R=(R_{kj})$: 
$$
P_{kj}=-\oint_{\beta_j} *\eta_{\beta_k}, \quad Q_{kj}=-\oint_{\alpha_j}*\eta_{\alpha_k}, \quad   R_{kj}=\oint_{\alpha_j}*\eta_{\beta_k}.
$$
In the last integral the contributions from $\alpha_j\cap \Omega$ and $\alpha_j\cap \tilde{\Omega}$
cancel since the curve goes in opposite directions on the two sides of the double , while the integrand $*\eta_{\beta_k}$
itself is even. Thus $R=(R_{kj})=0$ as a consequence of $M=\Hat{\Omega}$ being a symmetric Riemann surface. Note also that 
\begin{equation}\label{PUU}
P_{kj}=\int_{\hat{\Omega}} \eta_{\beta_k}\wedge *\eta_{\beta_j}
=2\int_{{\Omega}} \eta_{\beta_k}\wedge *\eta_{\beta_j}
=\frac{1}{2}\int_\Omega du_k\wedge *du_j
\end{equation}
by (\ref{ointdU}), (\ref{dudU}). Thus the system (\ref{RQPR}) takes the form 
\begin{equation}\label{stardUqdU}
-*\eta_\beta=P\,\eta_\alpha, \quad *\eta_\alpha=Q\,\eta_\beta.
\end{equation}
These equations are consistent since (\ref{PQIR}) shows that $Q=P^{-1}$ in the present case.
 
Consider now a fluid flow in $\Omega$ governed by Euler's equation (see Section~\ref{sec:Bernoulli}),
and copy this flow symmetrically to $\tilde{\Omega}$ so that we on the double $M$ have a flow
satisfying $\nu=J^* (\nu)$, this star denoting pull-back of a one-form. In particular, the flow does not cross $\partial\Omega$.
The stream function $\psi$, which is single-valued on $\Omega$ and 
for which the fluid velocity field as a one-form is given by $\nu=-*d\psi$,
must, in the case of a single vortex of strength $\Gamma=1$ at $a\in\Omega$, have a local expansion as in 
(\ref{psi}) and so be of the form 
$$
\psi=G_{\rm electro}(\cdot,a)+\psi_0
$$
for some function $\psi_0$ which is harmonic in all $\Omega$ and satisfies $d\psi_0=0$ along $\partial \Omega$. The latter means
that $\psi_0$ is constant on each boundary component, and these constants may depend on $a$. 

We thus have, for the flow, 
$$
\nu=-*dG_{\rm electro}(\cdot,a)+\eta,
$$
where $\eta=-*d\psi_0$ is a harmonic flow. 
Since $\psi_0$ takes nonzero, and possibly different, constant values on the components of 
$\partial\Omega$ it does not, in the multiply connected case,
extend to the  Schottky double as a harmonic function. However, $\psi_0$ is single-valued
on $\Omega$ and $\eta$ extends as a harmonic one-form, in fact an ``even'' form (symmetric
with respect to the involution), hence it is a linear combination of only the $\eta_{\alpha_j}$ ($j=1,\dots, \texttt{g}$),
referring to the basic harmonic forms in (\ref{etaalpha}), (\ref{etabeta}).

The above analysis leads eventually to that
$$
\psi=G_{\rm hydro}(\cdot,a), \quad\nu=-*dG_{\rm hydro}(\cdot,a),
$$
where, given the periods, $G_{\rm hydro}(\cdot,a)$ is an instance of the hydrodynamic Green function,
defined as follows.
\begin{definition}
With $p_1,\dots,p_\texttt{g}$ any set of prescribed periods
we define the {\it hydrodynamic Green function} for $\Omega$ associated with the list $p_1, \dots, p_{\texttt{g}}$ by
\begin{equation}\label{Ghydrodef}
G_{\rm hydro}(z,a)
=G_{\rm electro} (z,a)+\frac{1}{2}\sum_{i,j=1}^{\texttt g}Q_{ij}(u_{i}(z)-p_i)(u_{j}(a)-p_j) 
\end{equation}
for $z,a\in\Omega$.
\end{definition} 

A more complete notation would be $G_{{\rm hydro}, p_1,\dots, p_\texttt{g}}(z,a)$, but we refrain from using that.
The crucial properties of $G_{\rm hydro}(z,a)$ are that
\begin{equation}\label{dGhydro}
dG_{\rm hydro}(\cdot,a)=0 \quad \text{along }\partial \Omega,
\end{equation}
and that the periods have the prescribed values independent of $a$:
\begin{equation}\label{ointdGp}
-\oint_{\beta_k} *dG_{\rm hydro}(\cdot,a)=p_k.
\end{equation}
Aside from the obvious symmetry
$$
G_{\rm hydro}(a,b)=G_{\rm hydro}(b,a),
$$
the normalization statement
\begin{equation}\label{hydrohydro}
\oint_{\partial\Omega}G_{\rm hydro}(\cdot,a)*dG_{\rm hydro}(\cdot,b)=0,
\end{equation}
valid for all $a,b\in\Omega$, comes out as a consequence of (\ref{Ghydrodef}).

Independent of the construction, $G_{\rm hydro}(z,a)$ is determined by the singularity structure
together with (\ref{dGhydro}), (\ref{ointdGp}), (\ref{hydrohydro}).
From these one can deduce (as in (\ref{mutual energyG}))
$$
G_{\rm hydro}(a,b)=\int_\Omega dG_{\rm hydro}(\cdot, a)\wedge *dG_{\rm hydro}(\cdot, b),
$$
from which the symmetry is a consequence. 
See also \cite{Grotta-Ragazzo-Gustafsson-Koiller-2024}.
The period property (\ref{ointdGp}) requires a short proof:
$$
\oint_{\beta_k}*dG_{\rm hydro}(\cdot, a)=\oint_{\beta_k}*dG_{\rm electro}(\cdot,a)
+\frac{1}{2}\sum_{i,j=1}^{\texttt g}Q_{ij}\big(u_j(a)-p_j\big)\oint_{\beta_k}*du_i 
$$
$$
=\oint_{\partial \Omega}{u_k}*dG_{\rm electro}(\cdot, a)
+\frac{1}{2}\sum_{i,j=1}^{\texttt g}Q_{ij}\big(u_j(a)-p_j\big)\oint_{\partial\Omega}u_k *du_i
$$
$$
=-u_k(a)+\sum_{i,j=1}^{\texttt g}Q_{ij}\big(u_j(a)-p_j\big))P_{ki}=-p_k.
$$
Here we used (\ref{Uup}), (\ref{PUU}) and, in the final step, the fact that $Q=P^{-1}$.


\subsection{The hydrodynamic Bergman kernel}\label{sec:hydrodynamic Bergman} 

The {Bergman kernel}  for the full space of square 
integrable analytic functions in the planar domain $\Omega$ was discussed in 
Section~\ref{sec:Bergman} and denoted there $K(z,a)$,
as is standard. We shall in this section denote it
$K_{\rm electro}(z,a)$ in order to conform with the notation for the  electrostatic Green function, the 
 regular part of which then becomes $H_{\rm electro}(z,a)$, defined like in (\ref{GlogH}). Thus
\begin{equation}\label{KHelectro}
K_{\rm electro}(z,a)dzd\bar{a}
=-\frac{2}{\pi}\frac{\partial^2 H_{\rm electro}(z,a)}{\partial z\partial\bar{a}}dzd\bar{a}, \quad z,a\in\Omega,
\end{equation}
and considering it as reproducing kernel for the square integrable holomorphic differentials on $\Omega$
we have, in similarity with (\ref{Kreproducing}),
\begin{equation}\label{Kelectroreproducing}
f(a)da=\frac{\I}{2}\int_\Omega f(z)\wedge \overline{K_{\rm electro}(z,a)dzd\bar{a}}.
\end{equation} 
  
Using instead the hydrodynamic Green function and its regular part $H_{\rm hydro}(z,a)$, for {any} choice of circulations, 
one gets what usually goes under names as the {\it reduced}, {\it exact} or {\it semiexact}, Bergman kernel. We shall call it
the {hydrodynamic} Bergman kernel:
\begin{definition}\label{def:hydrodynamic Bergman} 
The {\it hydrodynamic Bergman kernel} $K_{\rm hydro}(z,a)$ is the reproducing kernel for the Hilbert space of those 
square integrable analytic functions $f$ in $\Omega$ which have a single-valued integral in $\Omega$,
equivalently are of the form $f=F^\prime$ for some $F$ analytic in $\Omega$.
The defining property is that $K_{\rm hydro}(\cdot,a)$ shall itself belong to the space and that, for all $f$ in it, 
\begin{equation}\label{Khydroreproducing}
f(a)da=\frac{\I}{2}\int_\Omega f(z)\wedge \overline{K_{\rm hydro}(z,a)dzd\bar{a}}.
\end{equation}
In terms of the regular part of hydrodynamic Green function it is given by
\begin{equation}\label{Khydro}
K_{\rm hydro}(z,a)dzd\bar{a}
=-\frac{2}{\pi}\frac{\partial^2 H_{\rm hydro}(z,a)}{\partial z\partial\bar{a}}dzd\bar{a}. 
\end{equation}
\end{definition}

The proof that the kernel (\ref{Khydro}) has the stated reproducing property is straight-forward, and well-known,
but we shall still say a few words about it.
For the electrostatic Green function we had
$$
\int_{\partial\Omega}f(z)G_{\rm electro}(z,a)dz=0,
$$
simply because $G_{\rm electro}(z,a)$ vanishes on the boundary. In the derivation of (\ref{intfG}) we actually used only
that the left member above is independent of $a$, and then we differentiated with respect to $a$. In the case of 
$G_{\rm hydro}(z,a)$ we instead use partial integration: when $fdz=dF$ in $\Omega$,
$$
\int_{\partial \Omega}f(z)G_{\rm hydro}(z,a)dz=-\int_{\partial\Omega}F(z)dG_{\rm hydro}(z,a)=0
$$
because of (\ref{dGhydro}), and so by Stokes' theorem,
$$
\int_{ \Omega}f(z)\frac{\partial G_{\rm hydro}(z,a)}{\partial \bar{z}}d\bar{z}dz=0.
$$
Then the hydrodynamic counterpart of (\ref{intfG}) follows after taking a derivative
with respect to $a$. Again we see that there is some redundancy: we need only that the left member above is
independent of $a$. This opens up for replacing $G_{\rm hydro}(z,a)$ with other functions, for example a certain 
Neumann function, see Section~\ref{sec:Neumann}.

Part of the requirements on $K_{\rm hydro}(z,a)dz$ is that it, as a differential, is exact in $\Omega$.
Thus it needs to satisfy
\begin{equation}\label{betaKhydro}
\oint_{\beta_j} K_{\rm hydro}(z,a)dz=0, \quad j=1,\dots, \texttt{g}.
\end{equation}
Since $a\in\Omega$, thus $a\notin \beta_j$, the verification of (\ref{betaKhydro}) is straight-forward:
$$
\oint_{\beta_j} K_{\rm hydro}(z,a)dz=-\frac{2}{\pi}\frac{\partial}{\partial\bar{a}}\oint_{\beta_j}\frac{\partial H_{\rm hydro}(z,a)}{\partial z}dz
=-4\frac{\partial}{\partial\bar{a}}\oint_{\beta_j}\frac{\partial G_{\rm hydro}(z,a)}{\partial z}dz
$$
$$
=-2\frac{\partial}{\partial\bar{a}}\oint_{\beta_j}\big({dG_{\rm hydro}(z,a)}+\I *dG_{\rm hydro} \big)
=0+2\I\frac{\partial }{\partial\bar{a}}p_j =0.
$$
Here we used (\ref{dGhydro}), (\ref{ointdGp}) in the last step. 
Similarly, one shows that
\begin{equation}\label{alphaKelectro}
\oint_{\alpha_j} K_{\rm electro}(z,a)dz=0, \quad j=1,\dots, \texttt{g}.
\end{equation}

The space of the holomorphic functions of the form $f=F^\prime$ has complex  codimension 
$\texttt{g}$ in the full Bergman space 
because for $f$ to have a single-valued anti-derivative requires that the $\texttt{g}$ integrals $\oint_{\beta_j}fdz$
($j=1,\dots,\texttt{g}$) vanish.	
In view of (\ref{Ghydrodef}), (\ref{Uup}) it follows that, accordingly,  
\begin{equation}\label{KKQ} 
K_{\rm electro}(z,a) -K_{\rm hydro}(z,a)
=2\sum_{i,j=1}^\texttt{g} Q_{ij}\frac{\partial u_i}{\partial z}\frac{\partial u_j}{\partial \bar{a}}.
\end{equation}

At this stage the Bergman kernel of a closed surface in Definition~\ref{def:Bergman and Schiffer},
specifically (\ref{K}), comes into use, the closed surface now being the Schottky double $M=\hat{\Omega}$ of $\Omega$. 
Recall the monopole Green function $G_{\rm double}(z,a)$ introduced in the beginning of Section~\ref{sec:electro-hydro},
which we decompose as
\begin{equation}\label{Gdouble}
G_{\rm double}(z,a)=\frac{1}{2\pi}\big(-\log |z-a|+H_{\rm double}(z,a)\big).
\end{equation}
Here the last term gives the Bergman kernel for the double:
\begin{equation}
\label{KdoubleH}
K_{\rm double}(z,a)=-\frac{2}{\pi}\frac{\partial^2 H_{\rm double}(z,a)}{\partial z \partial\bar{a}}.
\end{equation}
Now we compare (\ref{KKQ}) with the expansion (\ref{KPQ}), where by Definition~\ref{def:omegagamma}
and (\ref{dudU}),
$$
\omega_{\beta_j}=dU_{\beta_j}+\I*dU_{\beta_j}=2\frac{\partial U_{\beta_j}}{\partial z}dz
=-\frac{\partial u_j}{\partial z}dz.
$$
Since $P^{-1}=Q$ (recall (\ref{PQIR}), where presently $R=0$) we find that
\begin{equation}\label{KQu}
K_{\rm double}(z,a)dzd\bar{a}
=\sum_{k,j=1}^\texttt{g} Q_{kj}\frac{\partial u_k}{\partial z}\frac{\partial u_j}{\partial \bar{a}}dzd\bar{a}.
\end{equation}
Thus (\ref{KKQ}) gives the fundamental identity
\begin{equation}\label{KKH}
K_{\rm electro}(z,a)-K_{\rm hydro}(z,a)=2K_{\rm double}(z,a).
\end{equation}

Here $K_{\rm electro}$ and $K_{\rm hydro}$ are reproducing kernels for the planar domain $\Omega$,
while $K_{\rm double}$ is a reproducing kernel for the Schottky double of $\Omega$.
This explains the factor $2$ in (\ref{KKH}): the
reproducing property of $K_{\rm double}$ is given in (\ref{Kreproducing}) and involves integration over all
of double $M=\hat{\Omega}$, while the reproducing properties (\ref{Kelectroreproducing}), (\ref{Khydroreproducing})
of  $K_{\rm electro}$ and $K_{\rm hydro}$ only involve integration over $\Omega$. 

Implicit above is that all kernels have analytic continuations to the double (via the Green functions),
and the relations above remain valid all over the surface. The restrictions of $K_{\rm electro}$ and $K_{\rm hydro}$ 
to the back side are, after a pull back to the front side, often called the {\it adjoint} kernels.
For example, the adjoint $L_{\rm hydro}(z,a)dz$ of $K_{\rm hydro}(z,a)dz$ may  on the level of differentials be defined by 
\begin{equation}\label{adjointBergmanhydro}
L_{\rm hydro}(z,a)dz+\overline{K_{\rm hydro}(J(z),a)dJ(z)}=0, \quad z\in\Omega.
\end{equation}
This means that the kernel and its adjoint are glued by the boundary relation (recall that $J(z)=z$ on $\partial\Omega$)
\begin{equation}\label{LKboundary}
L_{\rm hydro}(z,a)dz+\overline{K_{\rm hydro}(z,a)dz}=0, 
\end{equation}
holding along $\partial\Omega$ (that is, with $dz$ aligned with the boundary). 
The same holds for the electrostatic kernels, and compare also with the
relation (\ref{gammaLK}) connecting the general Bergman and Schiffer kernels.

The relation (\ref{KKH}) expresses $K_{\rm double}$ in terms of $K_{\rm electro}$ and $K_{\rm hydro}$,
but we also want to have relations in the other direction. Such statements are contained in the following
proposition.

\begin{proposition}\label{prop:KKL}
The Bergman kernels on a planar domain are given in terms of the Bergman and Schiffer kernel on the double by
$$
K_{\rm electro}(z,a)dzd\bar{a}=K_{\rm double}(z,a)dzd\bar{a}-L_{\rm double}(z,J(a))dzdJ(a),
$$
$$
K_{\rm hydro}(z,a)dzd\bar{a}=-K_{\rm double}(z,a)dzd\bar{a}-L_{\rm double}(z,J(a))dzdJ(a).
$$
In the other direction,
$$
K_{\rm double}(z,a)dzd\bar{a}=\frac{1}{2}\big(K_{\rm electro}(z,a)dzd\bar{a}-K_{\rm hydro}(z,a)dzd\bar{a}\big),
$$
$$
-L_{\rm double}(z,J(a))dzdJ(a)=\frac{1}{2}\big(K_{\rm electro}(z,a)dzd\bar{a}+K_{\rm hydro}(z,a)dzd\bar{a}\big).
$$
When extended as differentials to the double by the above formulas, the Bergman kernels for $\Omega$
satisfy the period relations (\ref{betaKhydro}) and (\ref{alphaKelectro}).
\end{proposition}

\begin{proof}
The first equation, for $K_{\rm electro}$, follows from (\ref{GGG}), (\ref{Gdouble}), (\ref{Ghydrodef}) 
along with the definitions (\ref{K}), (\ref{L}) of the kernels. For $K_{\rm hydro}$ we combine this with (\ref{KKH}).
The relations in the other direction are immediate consequences, and the period relations have already been discussed.
\end{proof}

Returning briefly to the Robin function, 
the electrostatic and hydrodynamic Green functions have their own Robin functions, denoted
$\gamma_{\rm electro}$ and $\gamma_{\rm hydro}$, and the Laplacian of these are the Bergman kernels
specialized to the diagonal:
\begin{equation}\label{Deltagamma}
-\Delta \gamma_{\rm electro}(z)=4\pi K_{\rm electro}(z,z), \quad -\Delta \gamma_{\rm hydro}(z)=4\pi K_{\rm hydro}(z,z).
 \end{equation}
For $\gamma_{\rm electro}$ this equation was mentioned already in (\ref{Deltah0K}).
Comparing with (\ref{Deltagammadouble}) we see that there is no contribution from any metric 
here. The reason is that the contributions from the two sides of the Schottky double
cancel, and $G_{\rm electro}(z,a)$, $G_{\rm hydro}(z,a)$ are indeed purely harmonic quantities.
 
For $\gamma_{\rm double}$ we have on the other hand, by (\ref{Deltagammadouble}),
$$
-\Delta\gamma_{\rm double}(z)=4\pi \big(K_{\rm double}(z,z)-\frac{1}{V}\big).
$$
The additional term in the right member enters here is because $\gamma_{\rm double}$ is 
based on the monopole Green function, which depends on the metric. In (\ref{DeltaR}), we have $\kappa =0$
(except on $\partial\Omega$) since we are using the ``pancake'' metric. 

\begin{example}\label{ex:kernels}
As an illustration of the use of differentials we have, for the unit disk $\D$, for which
$J(a)={1}/{\bar{a}}, \quad dJ(a)=-{d\bar{a}}/{\bar{a}^2}$,
$$
K_{\rm electro}(z,a)=\frac{1}{\pi(1-z\bar{a})^2}, \quad L_{\rm electro}(z,a)=\frac{1}{\pi(z-a)^2},
$$
$$
 K_{\rm double}(z,a)=0, \quad L_{\rm double}(z,J(a))=\frac{1}{\pi(z-J(a))^2}.
$$
Since $K_{\rm hydro}=K_{\rm electro}$, $L_{\rm hydro}=L_{\rm electro}$ in the simply connected case
these formulas are in full agreement with (\ref{LKboundary}) and the statements in Proposition~\ref{prop:KKL}.

The monopole Green function $G_{\rm double}(z,a)$ for the unit disk with the Euclidean metric (pancake on the double)
was computed in Proposition~14 in \cite{Grotta-Ragazzo-Gustafsson-Koiller-2024}. It seems to be difficult to find it for
more general domains, even when $G_{\rm electro}(z,a)$ is known (in the opposite direction we have (\ref{GGG})).
\end{example}


\begin{figure}
\begin{center}
\includegraphics[scale=0.7]{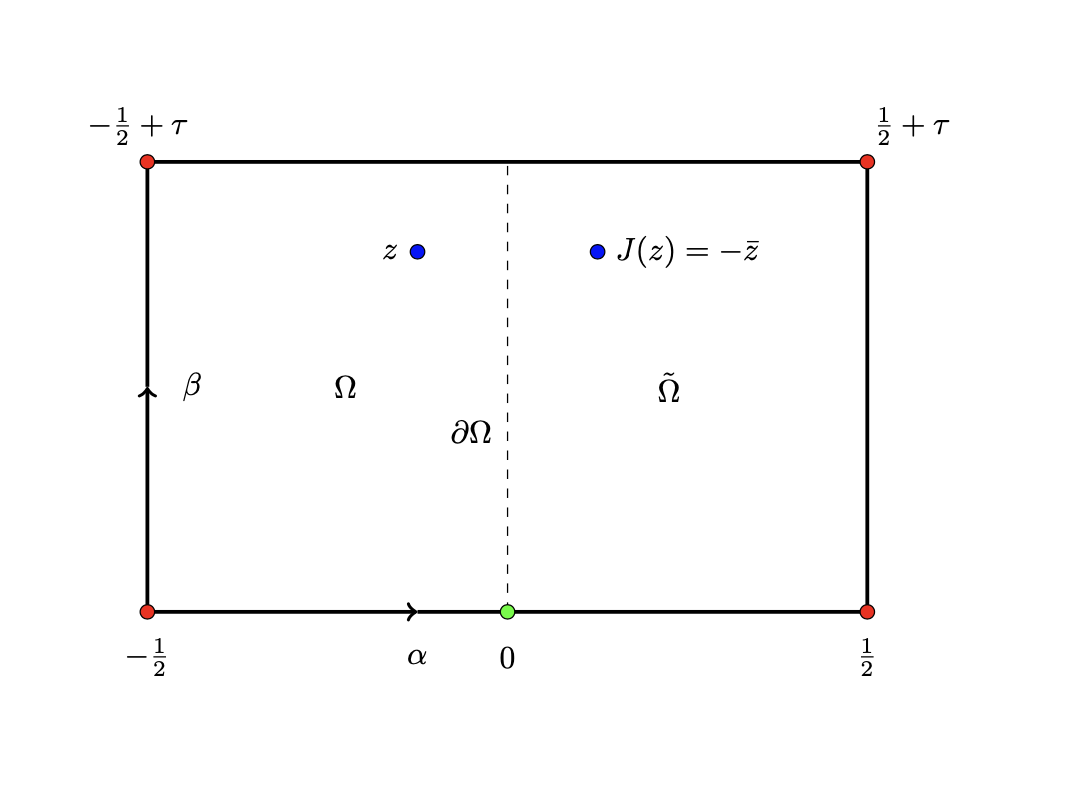}
\end{center}
\caption{Period perallelogram of a symmetric torus, the Schottky double of a doubly connected planar domain $\Omega$,
with backside $\tilde{\Omega}$ and involution $J$.}
\label{fig:torus2}
\end{figure}


\begin{example}\label{ex:kernels2}
In the case of a torus, and using notations as in Section~\ref{sec:torus}, 
the period parallelogram $F$ is a rectangle with right angles in the case of a double of a planar domain, 
thus the ``modulus'' $\tau$ is purely imaginary.
We need to identify the involution $J$ and the parts of the fundamental domain $F$ which correspond to 
$\Omega$ and $\tilde{\Omega}$, respectively. 
First of all, to simplify formulas we move $F$ leftwards so that it becomes symmetric with respect
to the $y$-axis, which then can be identified with the symmetry line of the Schottky double.

The involution then becomes reflection in the $y$-axis, $x\leftrightarrow -x$. Thus $J(z)=-\bar{z}$,
and we have (see Figure~\ref{fig:torus2})
$$
F=\{ -\frac{1}{2}<x<\frac{1}{2}, \,0<y<\im\tau\} \quad (\re \tau=0),  
$$
$$
{\Omega} =F\cap \{x<0\}, \quad \tilde{\Omega}=F\cap \{x>0\},
$$
$$
\alpha= \text{positively oriented }\,x\text{-axis},
$$
$$
\beta= \text{positively oriented }\,y\text{-axis}.
$$
We recall from Section~\ref{sec:torus} that, in present notation,
$$
K_{\rm double}(z,a)=\frac{1}{\im \tau}, \quad L_{\rm double}(z,a)=\frac{1}{\pi}\big(\wp (z-a)+\eta_1\big)- \frac{1}{\im \tau}.
$$
Since $dJ(a)=-\bar{a}$ this gives
$$
L_{\rm double}(z,J(a))dJ(a)=-\frac{1}{\pi}\big(\wp (z+\bar{a})+\eta_1\big)d\bar{a}+ \frac{1}{\im \tau}d\bar{a}.
$$

Thus, by Proposition~\ref{prop:KKL},
$$
K_{\rm electro}(z,a)=\frac{1}{\pi}\big(\wp (z+\bar{a})+\eta_1\big),
$$
$$
K_{\rm hydro}(z,a)=\frac{1}{\pi}\big(\wp (z+\bar{a})+\eta_1\big)-\frac{2}{\im\tau}.
$$
From this one confirms that (integrations along $\alpha$ and $\beta$) 
$$
\int_{-\frac{1}{2}}^\frac{1}{2} K_{\rm electro}(x,a)dx=0, \quad
\int_0^{\im \tau} K_{\rm electro}(\I y,a)\I dy=2\I,
$$
$$
\int_{-\frac{1}{2}}^\frac{1}{2} K_{\rm hydro}(x,a)dx=-\frac{2}{\im \tau}, \quad
\int_0^{\im \tau} K_{\rm hydro}(\I y,a)\I dy=0.
$$

The Abelian differential of the third kind with purely imaginary periods, see (\ref{upsilon}), is
in the present example
\begin{equation}\label{exupsilonab}
\upsilon_{a-b}(z)=\big(\zeta(z-a)-\zeta(z-b)+\eta_1 (a-b)+\frac{2\pi}{\tau} \im (a-b)\big)dz
\end{equation}
in terms of the $\zeta$-function used in Section~\ref{sec:torus}. 
Compare equations (6.5)-(6.8) in \cite{Gustafsson-Sebbar-2012}, and note that the last
term above makes the dependence on $a$ and $b$ to be only harmonic (not holomorphic).
However, when $b=J(a)=-\bar{a}$ this last term disappears, clarifying a previous remark after (\ref{dGdGnu}).

Further analysis related to doubly connected planar domains is available in \cite{Aboudi-2005}.
\end{example}


\subsection{A Neumann function}\label{sec:Neumann}

As mentioned, the space of the holomorphic functions of the form $f=F^\prime$ has complex  codimension 
$\texttt{g}$ in the full Bergman space (consisting of all square integrable analytic functions), and for the Bergman
kernels we had the relation (\ref{KKH}), also contained in Proposition~\ref{prop:KKL}. However, this relation cannot 
be raised to the level of potentials. We only have 
(\ref{GGG}), saying that $G_{\rm electro}(z,a)$ is twice the odd part $G_{\rm double}(z,a)$, while $G_{\rm hydro}(z,a)$
is not the corresponding even part. Indeed, the even part is not even harmonic outside the singularity.

Still we now introduce this even part, as a kind of Neumann function: 
\begin{definition}
The {\it Neumann function} $N(z,a)$ relevant in our context is
\begin{equation}\label{Neumann}
N(z,a)=G_{\rm double}(z,a)+G_{\rm double}(z,J(a)). 
\end{equation}
\end{definition}

Adding (\ref{GGG}) and (\ref{Neumann})  gives, in the other direction
\begin{equation}\label{GGN}
G_{\rm double}(z,a)=\frac{1}{2}\big(G_{\rm electro}(z,a)+N(z,a)\big).
\end{equation}

The above function $N(z,a)$ is not exactly the usual harmonic Neumann function for a planar domain (as in \cite{Nehari-1952},
for example).  It has the right singularity in $\Omega$ but it is not harmonic. 
The restriction to $\Omega$ satisfies instead the Poisson equation
$$
-4\frac{\partial^2 N(z,a)}{\partial z\partial\bar{z}} =\delta_a(z)-\frac{1}{{\rm area}(\Omega)}, \quad z\in\Omega.
$$
On the boundary $\partial\Omega$, the normal derivative vanishes because the function is an even function on the double.
In general, the term ``Neumann function'' usually refers to a Green type function having  constant normal derivative on the
boundary. 

Returning to the Bergman kernels in (\ref{KKH}) and comparing with (\ref{GGN}) we now have
$$
\frac{\partial^2 G_{\rm double}(z,w)}{\partial z \partial\bar{a}}
=\frac{1}{2}\Big(\frac{\partial^2 G_{\rm electro}(z,a)}{\partial z \partial\bar{a}}
+\frac{\partial^2 N(z,a)}{\partial z \partial\bar{a}}\Big).
$$
Since on the other hand,
$$
\frac{\partial^2 G_{\rm double}(z,a)}{\partial z \partial\bar{a}}
=\frac{1}{2}\Big(\frac{\partial^2 G_{\rm electro}(z,a)}{\partial z \partial\bar{a}}
-\frac{\partial^2 G_{\rm hydro}(z,a)}{\partial z \partial\bar{a}}\Big)
$$
we see that
\begin{equation}\label{NG}
\frac{\partial^2 N(z,a)}{\partial z \partial\bar{a}}=-\frac{\partial^2 G_{\rm hydro}(z,a)}{\partial z \partial\bar{a}}.
\end{equation}
This identity is well-known in classical contexts, namely when $N(z,a)$ is harmonic in $\Omega$ and has a constant,
but non-zero, normal derivative on $\partial\Omega$.
The identity might seem remarkable because $G_{\rm hydro}(z,a)$
depends on given circulations, while $N(z,a)$ has no parameters in it. However, the dependence 
on the circulation disappears under the differentiations. 


\subsection{Orthogonal decomposition of Bergman space}\label{sec:orthogonal}

The space of the square integrable holomorphic one-forms on the planar domain $\Omega$ 
is infinite dimensional, while those which extend holomorphically to the double make up a subspace of finite
dimension $\texttt{g}$. No independent metric is needed since we work with one-forms, for which there is the  intrinsic 
inner product such as (\ref{mutual energy}), with a conjugation of the second factor in the case of complex
Hilbert spaces.  Choosing an orthonormal basis $\{e_k\}_{k=1}^\texttt{g}$ for the latter space,
for example by orthonormalizing one of the bases in Proposition~\ref{prop:GUU}
as after (\ref{Kee}), we have
\begin{equation}\label{Kdoubleee}
2K_{\rm double}(z,a)dzd\bar{a}=\sum_{k=1}^\texttt{g} e_k(z)\overline{e_k(a)}.
\end{equation}
Here a factor $2$ is inserted in the left member because the kernel $K_{\rm double}$ is adapted to a Hilbert
on the Schottky double of $\Omega$, while we now are orthonormalizing only on $\Omega$ itself.
Compare discussions after (\ref{KKH}). In brief, it is $2K_{\rm double}$ that is comparable with 
$K_{\rm electro}$ and $K_{\rm hydro}$.

We may extend the basis $\{e_k\}_{k=1}^\texttt{g}$ to $\{e_k\}_{k=1}^\infty$ (still orthonormal) to write
\begin{equation}\label{Kelectroee}
K_{\rm electro}(z,a)dzd\bar{a}=\sum_{k=1}^\infty e_k(z)\overline{e_k(a)}.
\end{equation}
From (\ref{KKH}), (\ref{Kdoubleee}), (\ref{Kelectroee}) it is obvious that $K_{\rm hydro}(z,a)$
now is given by
\begin{equation}\label{Khydroee}
K_{\rm hydro}(z,a)dzd\bar{a}=\sum_{k=\texttt{g}+1}^\infty e_k(z)\overline{e_k(a)}.
\end{equation}
Thus, in some sense, 
$$
2\cdot {\rm double}+{\rm hydrodynamic}= {\rm electrostatic},
$$
and this corresponds to an orthogonal decomposition of the Bergman space.

The splitting 
$$
{\rm electrostatic}\leftrightarrow {\rm hydrodynamic}
$$
is closely related to the dichotomy of having two kinds of cycles on the Schottky double,
$$
\alpha-{\rm cycles} \leftrightarrow \beta-{\rm cycles}.
$$
Recall that the two Bergman kernels can be characterized as being 
Abelian integrals of the second kind with single second order poles on the backside of the 
surface and determined by having vanishing $\alpha$-periods (for $K_{\rm electro}$), respectively vanishing
$\beta$-periods (for $K_{\rm hydro}$), see (\ref{betaKhydro}), (\ref{alphaKelectro}) and Example~\ref{ex:kernels}. 

It may appear surprising that one of the kernels ($K_{\rm electro}$) is bigger than the other ($K_{\rm hydro}$)
in the sense that the difference ($2K_{\rm double}$) is positive definite, while the choice between $\alpha$ and
$\beta$ cycles seems completely symmetric. However, it is not really symmetric because the involution $J$
involved in the formation of the Schottky double does not treat the two types of cycles equally.
Positive definiteness questions for Bergman kernels are discussed in depth in \cite{Hejhal-1972}, along with  the
Szeg\"o kernel (see Section~\ref{sec:Szego}). The latter cuts through  $K_{\rm double}$
in a way which is (perhaps) not fully understood at present. See further below. 

The orthogonality implicit in the expansions (\ref{Kdoubleee}), (\ref{Kelectroee}), (\ref{Khydroee})
can also be described as follows (immediate consequence of (\ref{Kdoubleee}) and (\ref{Khydroee})):
\begin{proposition}\label{prop:orthogonal}
For any $a,b\in\Omega$,
$$
\int_\Omega K_{\rm double}(z,a)dz\wedge  \overline{K_{\rm hydro}(z,b)dz}=0.
$$
\end{proposition}

\begin{remark}\label{rem:intermediate}
There are certainly intermediate cases between the pure electrostatic and hydrodynamic Bergman kernels, corresponding
to mixing the $\alpha$- and $\beta$-cycles. Such intermediate cases for the Bergman kernels are discussed in
\cite{Schiffer-Spencer-1954, Ahlfors-Sario-1960}. These texts also discuss doubles not only of planar domains,
but also of bordered Riemann surfaces in general, and of non-orientable surfaces.
\end{remark}


\subsection{The Szeg\"o kernel}\label{sec:Szego}

In addition to the Bergman kernels $K_{\rm hydro}(z,a)$ and $K_{\rm electro}(z,a)$, there is the
{\it Szeg\"o kernel}, $K_\text{Szeg\"o}(z,a)$. This is the reproducing kernel for the analytic functions in $\Omega$
with respect to the arc-length metric, i.e. for the {\it Hardy space}, or more exactly {\it Smirnov space}, 
see \cite{Duren-1970, Bell-2016}. The inner product is in this case
\begin{equation}\label{Smirnov}
(f,g)=\int_{\partial \Omega} f(z)\overline{g(z)}|dz|,
\end{equation}
and $K_\text{Szeg\"o}(z,a)$ has the defining property that $f(a)=(f,K_\text{Szeg\"o}(\cdot ,a))$ for all $f$ in the space. It turns out
that the square of the Szeg\"o kernel is intermediate between the two Bergman kernels, see \cite{Hejhal-1972}
along with some discussion below.

Like the Bergman kernels, the Szeg\"o kernel extends to the Schottky double of the domain, however not as an ordinary differential,
but as a half-order differential in each variable, $K_\text{Szeg\"o}(z,a)\sqrt{dz}\sqrt{d\bar{a}}$. See below for the meaning of this.
The representation of the Szeg\"o kernel on the back-side of the Schottky double is known as the {\it Garabedian kernel}, see \cite{Bell-2016}.
It has a singularity of the form
\begin{equation}\label{LSzego}
L_\text{Szeg\"o}(z,a)=\frac{1}{2\pi(z-a)}+\text{regular},
\end{equation}
and it matches $K_{\text{Szeg\"o}}(z,a)$ along the boundary $\partial\Omega$ as
\begin{equation}\label{Szegoboundary}
L_\text{Szeg\"o}(z,a)\sqrt{dz}=\I \overline{K_\text{Szeg\"o}(z,a)\sqrt{d{z}}}.
\end{equation} 
In addition to the pole of $L_\text{Szeg\"o}(z,a)$ seen in (\ref{LSzego}),  $K_\text{Szeg\"o}(z,a)$ has $\texttt{g}$ zeros in $\Omega$.
See \cite{Bell-2016} and compare the discussion of the Ahlfors map in (\ref{Ahlfors}).

On a general Riemann surface the choice of square roots for half order differentials, like the Szeg\"o  kernel, is not a trivial matter,
see \cite{Hawley-Schiffer-1966, Gunning-1966}.
Making consistent choices all over the surface amounts to a choice of a {\it spin structure} of the surface (compare
\cite{Atiyah-1971}). However,  in the case of a Schottky double of a planar domain there is a canonical choice 
\cite{Hejhal-1972, Gustafsson-1987}: in (\ref{Szegoboundary}) one may simply, on augmenting the relation by another factor
$\sqrt{dz}$, express it as $L_\text{Szeg\"o}(z,a)T(z)=\I \overline{K_\text{Szeg\"o}(z,a)}$ 
where $T(z)=dz/|dz|$ is the positively oriented unit tangent vector on $\partial\Omega$.
 
The singularity structure (\ref{LSzego}) shows that $(2\pi L_\text{Szeg\"o}(z,a))^{-1}$ behaves
like the coordinate function $z-a$ for $z$ close to $a$. The expression
\begin{equation}\label{E}
\frac{1}{2\pi L_\text{Szeg\"o}(z,a){\sqrt{dz}\sqrt{da}}}=\frac{z-a}{\sqrt{dz}\sqrt{da}}+\dots
\end{equation}
is a degree $(-\frac{1}{2}, -\frac{1}{2})$ double differential on the Schottky double,
and besides the zero at $z=a$ there are in total $\texttt{g}$ poles on the double 
(located on the backside, so these are the zeros of $K_\text{Szeg\"o}(z,a)$).
It is possible to remove these poles at the price of introducing some multiplicative multi-valuedness.
This leads to the {\it Schottky-Klein prime function} (or {\it prime form}) ${E_{\rm prime}(z,a)}$,
which is usually constructed via theta functions.
We refer to \cite{Hejhal-1972, Fay-1973, Bogatyrev-2009, Crowdy-2010} for details on these matters.
In \cite{Hoker-2024} ${E_{\rm prime}(z,a)}$ is used to give an individual meaning to each of the two terms in the
decomposition (\ref{GHza}) of the Schiffer kernel.

\begin{example}
In the notations of Example~\ref{ex:kernels2} we have, in the genus one case,
$$
L_\text{Szeg\"o}(z,a)=\frac{1}{2\pi}\sqrt{\wp(z-a)-\wp(\frac{1}{2}(1+\tau))}.
$$
The single-valuedness of the square root is guaranteed by the fact that $\wp(z)$ has a double zero at $z=\frac{1}{2}(1+\tau)$.
Recall that, in present notation and referring to a variable $t=z-a$,
$$
\wp'(t)^2=4\big(\wp(t)-\wp(\frac{1}{2})\big)\big(\wp(t)-\wp(\frac{\tau}{2})\big)\big(\wp(t)-\wp(\frac{1}{2}(1+\tau))\big).
$$
See in general \cite{Ahlfors-1966}. 
\end{example}

As noticed in \cite{Hawley-Schiffer-1966}, the Garabedian kernel allows for the introduction of a canonical decomposable
Abelian differential of the third kind,
\begin{equation}\label{LLL}
\frac{L_\text{Szeg\"o}(z,a)L_\text{Szeg\"o}(z,b)}{L_\text{Szeg\"o}(a,b)}dz,
\end{equation}
defined independently of any period requirements.
The pole structure of the expression (\ref{LLL}) imitates that of the identity
\begin{equation*}\label{elementary1}
\frac{1}{z-a}-\frac{1}{z-b}=\frac{a-b}{(z-a)(z-b)}.
\end{equation*}
The famous {\it Fay trisecant identity} \cite{Fay-1973, Raina-1989} can similarly be viewed as a version of the identity
\begin{equation*}\label{elementary2}
\frac{1}{z-a}\cdot \frac{1}{w-b}-\frac{1}{z-b}\cdot \frac{1}{w-a}
=-\frac{(z-w)(a-b)}{(z-a)(z-b)(w-a)(w-b)},
\end{equation*}
or more generally
$$
\det \Big( \big(\frac{1}{z_i-a_j}\big)_{i,j}\Big)=\frac{\prod_{i<j}(z_i-z_j)(a_j-a_i)}{\prod_{i,j}(z_i-a_j)},
$$
when this is formulated in terms of $E_{\rm prime}(z,a)$.

Related to the above, and what we shall actually need about the Szeg\"o kernel,
 is that its square is comparable with the Bergman kernel. This means for example that
\begin{equation}\label{SG}
4\pi K_\text{Szeg\"o}(z,a)^2= -\frac{2}{\pi}\frac{\partial^2 H_\text{Szeg\"o}(z,a)}{\partial z \partial \bar{a}}
\end{equation}
in terms of the regular part $H_{\text{Szeg\"o}}(z,a)$ of a certain {\it Szeg\"o Green function} (our terminology), $G_\text{Szeg\"o}(z,a)$, which has the same singularity 
as the other Green functions, 
\begin{equation}\label{GSzegoH}
G_\text{Szeg\"o}(z,a)=\frac{1}{2\pi}\big(-\log |z-a| + H_\text{Szeg\"o}(z,a)\big),
\end{equation}
and which is constant on each boundary component. 
This Green function seems not to have been studied in any depth, but it is mentioned in \cite{Garabedian-1949}.
Presumably, the property (\ref{SG}) that the mixed second derivative is a perfect square characterizes the
Szeg\"o Green function. 

In any case it follows that, in analogy with (\ref{Ghydrodef}),  
$$
{G_{\text{Szeg\"o}}}(z,a)=G_{\rm electro}(z,a)+ \frac{1}{2} \sum_{k,j=1}^{\texttt{g}}c_{kj} u_k(z){u_j(a)},
$$
$$
4\pi K_\text{Szeg\"o}(z,a)^2=K_{\rm electro}(z,a)-2\sum_{k,j=1}^{\texttt{g}}c_{kj} \frac{\partial u_k}{\partial z}\frac{\partial u_j}{\partial \bar{a}},
$$
for suitable coefficients $c_{kj}$, the same in both equations. Comparing instead with the hydrodynamic Bergman kernel we can write
$$
4\pi K_\text{Szeg\"o}(z,a)^2=K_{\rm hydro}(z,a)+2\sum_{k,j=1}^{\texttt{g}}C_{kj} \frac{\partial u_k}{\partial z}\frac{\partial u_j}{\partial \bar{a}},
$$
It is a deep fact, indeed the main result (Theorem~39) in \cite{Hejhal-1972}, that both matrices $(c_{kj})$ and $(C_{kj})$
are positive definite. See \cite{Bell-Gustafsson-2022} for some recent discussions. 

\begin{remark}\label{rem:weighted Bergman}
It is possible to embed both the Bergman space and the Hardy (Smirnov) space in a single series of weighted Bergman spaces.
In the case of the unit disk these spaces are defined by the inner products
\begin{equation}\label{weighted Bergman}
(f,g)_\alpha=(\alpha+1)\int_\D f(z)\overline{g(z)}(1-|z|^2)^{\alpha}\, dxdy,
\end{equation}  
where $-1< \alpha<\infty$ is a real parameter. Choosing $\alpha=0$ one gets the Bergman inner product, while letting
$\alpha\to -1$ gives the Hardy/Smirnov inner product, up to a factor. The above can be generalized to arbitrary domains
by replacing $1-|z|^2$ with $1/\lambda(z)$, referring to  the Poincar\'{e} metric written as $ds=\lambda(z) |dz|$. 
The Poincar\'{e} metric for a  general domain is obtained from the Green function for the universal covering surface, 
see \cite{Ahlfors-1973} and after (\ref{Poincare}) in the present paper. 

For integer values of $\alpha$ the orthogonal decomposition in Section~\ref{sec:orthogonal}  generalizes to corresponding results for
the weighted spaces, see  \cite{Gustafsson-Peetre-1990, Gustafsson-Sebbar-2012}.
\end{remark}


\subsection{Extremal problems and conformal mapping}\label{sec:capacity functions} 

For a planar domain $\Omega\subset\C$ of connectivity $\texttt{g}+1$ we define,
in the spirit of \cite{Sario-Oikawa-1969, Hejhal-1972}, the {\it capacity functions} (or {\it capacities})
$$
c_D(z)=\sqrt{\pi K_{\rm hydro}(z,z)},
$$
$$
c_B(z)=2\pi K_\text{Szeg\"o}(z,z),
$$
$$
M(z)=\pi K_{\rm electro}(z,z),
$$
$$
c_\beta(z)=e^{-\gamma_{\rm electro}(z)},
$$
$$
c_1(z)=e^{-\gamma_{\rm hydro}(z)}.
$$
The notations in the left members have historical explanations. For example, $D$ stands for Dirichlet (norm),
$B$ stands for bounded ($L^\infty$-norm), $\beta$ is a traditional notation for the ``ideal boundary'' (all of the boundary,
in an abstract sense).
The origin for the notation $M$ is somewhat less clear, but a hint is that \cite{Schiffer-Spencer-1954} 
({Chapter~4}) uses a gothic $M$ for the manifold (Riemann surface) and then uses the fat ${\bf M}$ for the Hilbert
space of holomorphic differentials on it, which directly leads to the Bergman kernel.

The function $c_1(z)$ actually depends on the choice of circulations for the underlying hydrodynamic Green function. For the
purpose of the present section we may take $p_1=\dots=p_\texttt{g}=0$ in (\ref{ointdGp}), i.e. all circulation is taken up by  
the outer component of $\partial\Omega$. It then follows from the normalization (\ref{hydrohydro}) that 
$G_{\rm hydro}(\cdot,a)=0$ on that component.

The above capacity functions are solutions of the following extremal problems for analytic functions (citing from \cite{Sario-Oikawa-1969}):
\begin{align*} 
\frac{1}{c_1(z)}&= \min \sup_\Omega |f| \,: \quad&& f \text{ univalent, \,} f(z)=0,\,\, f^\prime (z)=1,\\
&= \min\, \sqrt{\frac{1}{\pi} \int_\Omega |f^\prime |^2\,dxdy }\,: \quad&& f \text{ univalent, \,}  f^\prime (z)=1,\\
\frac{1}{c_D(z)}& = \min \sqrt{\frac{1}{\pi} \int_\Omega |f^\prime|^2\,dxdy }\,: \quad&& f \text{ analytic, \,\,} f^\prime (z)=1,\\
\frac{1}{c_B(z)}&= \min \sup_\Omega |f|\,: \quad&&f \text{ analytic, \,\,} f(z)=0,\,\, f^\prime (z)=1, \\
&=\min \, \frac{1}{2\pi} \int_{\partial\Omega} |f|^2 \,ds \,: \quad&& f \text{ analytic, \,\,} f(z)=1,\\
\frac{1}{c_\beta(z)}&= \min \sup_\Omega |f|\,:\quad&&   f \text{ analytic*, \,} f(z)=0,\,\, f^\prime (z)=1 ,\\
\frac{1}{\sqrt{M(z)}}& = \min \sqrt{\frac{1}{\pi} \int_\Omega |f|^2\,dxdy } \,: \quad&& f \text{ analytic, \,\,} f(z)=1.\\
\end{align*}
Here $f$ analytic* means that $f$ is (possibly) multiple-valued analytic, but with $|f|$ single-valued. 
The side condition is to hold for one of the branches. 

In the above notation,
\begin{equation}\label{inequalities0}
c_1(z)\leq c_D(z)\leq c_B(z)\leq c_\beta (z)\leq \sqrt{M(z)},
\end{equation}
and in our more explicit notation this becomes (after taking squares)
\begin{equation}\label{inequalities}
e^{-2\gamma_{\rm hydro}(z)}\leq \pi {K_{\rm hydro}}(z,z)\leq 4\pi^2 K_\text{Szeg\"o}(z,z)^2\leq e^{-2\gamma_{\rm electro}(z)}\leq\pi K_{\rm electro}(z,z).
\end{equation}
All the above quantities can be viewed as coefficients for natural metrics.
It is not clear where ${\gamma_\text{Szeg\"o}}(a)=H_{\text{Szeg\"{o}}}(a,a)$ fits into the picture, except that
$$
e^{-2\gamma_{\rm hydro}(z)}\leq e^{-2{\gamma_\text{Szeg\"o}}(a)}\leq e^{-2\gamma_{\rm electro}(z)}. 
$$

All inequalities (\ref{inequalities0}), (\ref{inequalities}) become equality in the simply connected case. 
This means, on comparison with (\ref{Deltagamma}), 
that the Robin functions satisfy Liouville type equations: with 
$u=\gamma_{\rm hydro}=\gamma_\text{Szeg\"o}=\gamma_{\rm electro}$ in the simply connected case
we have $\Delta u=4e^{-2u}$, expressing that the metric $ds=e^{-\gamma(z)}|dz|$ has constant curvature. 
In the multiply connected case, all inequalities in (\ref{inequalities}) are strict  (under our assumptions on $\partial\Omega$). 
The first and last inequalities in (\ref{inequalities}) say that the Gaussian curvatures  of the corresponding metrics satisfy 
(see \cite{Sario-Oikawa-1969})
\begin{equation}\label{estimateskappa}
\kappa_{\rm electro}\leq -4, \quad {\kappa_{\rm hydro}}\leq -4.
\end{equation}
We do not know of any estimate for ${\kappa_\text{Szeg\"o}}$. 

The extremal functions for the minimization problems above are related to canonical mapping functions, briefly as follows:
\begin{itemize}

\item The two problems for $c_1(z)$ have the same extremal function $f(z)$. Such a function maps $\Omega$ 
onto a circular slit disk centered at the origin and having radius $\gamma_{\rm hydro}(z)$. More precisely, $f$ is given by
\begin{equation}\label{fGG} 
f(z)=\exp\big(\gamma_{\rm hydro}(a) -{G}_{\rm hydro}(z,a)-\I {G}^*_{\rm hydro}(z,a)  \big). 
\end{equation}

\item The extremal function for the problem for $c_D(z)$ is a normalized version of the integral of the 
hydrodynamic Bergman kernel:
\begin{equation}\label{fKK}
f(z)=\int_a^z \frac{{K_{\rm hydro}}(\zeta,a)}{{K_{\rm hydro}}(a,a)}\,d\zeta.
\end{equation}

\item 
Any primitive function of $L_{\rm hydro}(z,a)$, 
\begin{equation}\label{fintL}
f(z)=\int_b^z L_{\rm hydro}(\zeta, a)d\zeta
\end{equation}
where $b\ne a$,
is univalent and maps $\Omega$ onto the exterior of $\texttt{g}+1$ compact convex sets in the plane
(then with $f(a)=\infty$, $f(b)=0$). See \cite{Schiffer-1943, Sario-Oikawa-1969}.

\item Mixing $K_{\rm hydro}$ and $L_{\rm hydro}$ in linear combinations
gives canonical maps onto slit regions. For example, the function
$$
f(z)=\int_b^z \big(L_{\rm hydro}(\zeta,a)- K_{\rm hydro}(\zeta,a)\big) d\zeta
$$
maps $\Omega$ onto a horizontal slit domain. Indeed, as a consequence of 
(\ref{LKboundary}), the differential of $f$ equals $f^\prime (z)dz= -2\re (K_{\rm hydro}(z,a)dz)$, 
hence is real along the boundary. 

\item The extremal function for the first problem for $c_B(z)$ is a suitably scaled {\it Ahlfors map}.
In terms of the Szeg\"o and Garabedian kernels, see (\ref{LSzego}), (\ref{Szegoboundary}), this is
\begin{equation}\label{Ahlfors}
f(z)= \frac{K_\text{Szeg\"o}(z,a)}{K_\text{Szeg\"o}(a,a)L_\text{Szeg\"o}(z,a)},
\end{equation}
and after normalization it maps $\Omega$ onto the unit disk covered $\texttt{g}+1$ times.
  
\item The extremal function for the second problem for $c_B(z)$ is the (normalized) Szeg\"o kernel:
\begin{equation}\label{KLSzego}
f(z)=\frac{K_\text{Szeg\"o}(z,a)}{K_\text{Szeg\"o}(a,a)}
\end{equation}
\item For $c_\beta$, the extremal function (multiple-valued) is the exponential of the analytic completion of the Green function. More precisely:
\begin{equation}\label{fKSzego}
f(z)=\exp\big(\gamma_{\rm electro}(a) -G_{\rm electro}(z,a)-\I G_{\rm electro}^*(z,a)  \big). 
\end{equation}
\item For $M(z)$, the extremal function is the normalized Bergman kernel:
\begin{equation}\label{fKelectro}
f(z)=\frac{K_{\rm electro}(z,a)}{K_{\rm electro}(a,a)}.
\end{equation}
\end{itemize}

References for the inequalities (\ref{inequalities0}), (\ref{inequalities}) are as follows.
\begin{enumerate}

\item The first inequality is proved in \cite{Sario-Oikawa-1969} (p.172-179). It is actually obvious in view of the extremal problems
because a univalent function is by definition analytic.

\item The second inequality (related to Ahlfors-Beurling \cite{Ahlfors-Beurling-1950}) also appears in \cite{Sario-Oikawa-1969},
and in \cite{Sakai-1969} (see also \cite{Sakai-1982}) it is shown to be strict in the multiply connected case (in \cite{Sario-Oikawa-1969} 
this is stated as an open question).

\item The third inequality, in the form $\leq$, is obvious from the extremal problems. Also the strict version follows easily, and is proved explicitly in \cite{Sario-Oikawa-1969}, p.178-179.

\item The fourth inequality was conjectured in \cite{Suita-1972} and proved in \cite{Berndtsson-Lempert-2016}. 
Suita \cite{Suita-1972} also has a simple proof of the inequality $4\pi K_\text{Szeg\"o}(a,a)^2\leq K_{\rm electro}(a,a)$.
Hejhal \cite{Hejhal-1972} shows strict inequality in the multiply connected case.

\end{enumerate}



\bibliography{bibliography_gbjorn.bib}

\def\cprime{$'$} \def\cprime{$'$} \def\cprime{$'$} \def\cprime{$'$}
  \def\cprime{$'$} \def\cprime{$'$}
\begin{thebibliography}{100}

\bibitem{Aboudi-2005}
{\sc N.~Aboudi}, {\em Geodesics for the capacity metric in doubly connected
  domains}, Complex Var. Theory Appl., 50 (2005), pp.~7--22.

\bibitem{Ahlfors-Beurling-1950}
{\sc L.~Ahlfors and A.~Beurling}, {\em Conformal invariants and
  function-theoretic null-sets}, Acta Math., 83 (1950), pp.~101--129.

\bibitem{Ahlfors-1966}
{\sc L.~V. Ahlfors}, {\em Complex analysis: {A}n introduction of the theory of
  analytic functions of one complex variable}, Second edition, McGraw-Hill Book
  Co., New York, 1966.

\bibitem{Ahlfors-1973}
\leavevmode\vrule height 2pt depth -1.6pt width 23pt, {\em Conformal
  invariants: topics in geometric function theory}, McGraw-Hill Series in
  Higher Mathematics, McGraw-Hill Book Co., New
  York-D\"{u}sseldorf-Johannesburg, 1973.

\bibitem{Ahlfors-Sario-1960}
{\sc L.~V. Ahlfors and L.~Sario}, {\em Riemann surfaces}, Princeton
  Mathematical Series, No. 26, Princeton University Press, Princeton, N.J.,
  1960.

\bibitem{Alekseev-Mineev-2017}
{\sc O.~Alekseev and M.~Mineev-Weinstein}, {\em Theory of stochastic
  {L}aplacian growth}, J. Stat. Phys., 168 (2017), pp.~68--91.

\bibitem{Alling-Greenleaf-1971}
{\sc N.~L. Alling and N.~Greenleaf}, {\em Foundations of the theory of {K}lein
  surfaces}, Lecture Notes in Mathematics, Vol. 219, Springer-Verlag,
  Berlin-New York, 1971.

\bibitem{Alvarez-1987}
{\sc L.~Alvarez-Gaum\'e, J.-B. Bost, G.~Moore, P.~Nelson, and C.~Vafa}, {\em
  Bosonization on higher genus {R}iemann surfaces}, Comm. Math. Phys., 112
  (1987), pp.~503--552.

\bibitem{Armitage-Gardiner-2001}
{\sc D.~H. Armitage and S.~J. Gardiner}, {\em Classical potential theory},
  Springer Monographs in Mathematics, Springer-Verlag London, Ltd., London,
  2001.

\bibitem{Arnold-1978}
{\sc V.~I. Arnold}, {\em Mathematical methods of classical mechanics},
  Springer-Verlag, New York, 1978.
\newblock Translated from the Russian by K. Vogtmann and A. Weinstein, Graduate
  Texts in Mathematics, 60.

\bibitem{Arnold-Khesin-1998}
{\sc V.~I. Arnold and B.~A. Khesin}, {\em Topological methods in
  hydrodynamics}, vol.~125 of Applied Mathematical Sciences, Springer-Verlag,
  New York, 1998.

\bibitem{Atiyah-1971}
{\sc M.~F. Atiyah}, {\em Riemann surfaces and spin structures}, Ann. Sci.
  \'Ecole Norm. Sup. (4), 4 (1971), pp.~47--62.

\bibitem{Balabanova-Montaldi-2022}
{\sc N.~Balabanova and J.~Montaldi}, {\em A hamiltonian approach for point
  vortices on non-orientable surfaces i: the m\"obius band}, arXiv:2022.06160
  [math.DS], 3 (2022), pp.~1--31.

\bibitem{Bandle-Flucher-1998}
{\sc C.~Bandle and M.~Flucher}, {\em Harmonic radius and concentration of
  energy; hyperbolic radius and {L}iouville's equations {$\Delta U=e^U$} and
  {$\Delta U=U^{(n+2)/(n-2)}$}}, SIAM Rev., 38 (1996), pp.~191--238.

\bibitem{Bass-1995}
{\sc R.~F. Bass}, {\em Probabilistic techniques in analysis}, Probability and
  its Applications (New York), Springer-Verlag, New York, 1995.

\bibitem{Beardon-Minda-2007}
{\sc A.~F. Beardon and D.~Minda}, {\em The hyperbolic metric and geometric
  function theory}, in Quasiconformal mappings and their applications, Narosa,
  New Delhi, 2007, pp.~9--56.

\bibitem{Bell-2016}
{\sc S.~R. Bell}, {\em The {C}auchy transform, potential theory and conformal
  mapping}, Chapman \& Hall/CRC, Boca Raton, FL, second~ed., 2016.

\bibitem{Bell-Gustafsson-2022}
{\sc S.~R. Bell and B.~Gustafsson}, {\em Ruminations on {H}ejhal's theorem
  about the {B}ergman and {S}zeg\"{o} kernels}, Anal. Math. Phys., 12 (2022),
  pp.~Paper No. 24, 15.

\bibitem{Bergman-1970}
{\sc S.~Bergman}, {\em The kernel function and conformal mapping}, Mathematical
  Surveys, No. V, American Mathematical Society, Providence, R.I., revised~ed.,
  1970.

\bibitem{Berndtsson-Lempert-2016}
{\sc B.~Berndtsson and L.~Lempert}, {\em A proof of the {O}hsawa-{T}akegoshi
  theorem with sharp estimates}, J. Math. Soc. Japan, 68 (2016),
  pp.~1461--1472.

\bibitem{Boatto-Koiller-2013}
{\sc S.~Boatto and J.~Koiller}, {\em Vortices on closed surfaces}, in Geometry,
  mechanics, and dynamics, vol.~73 of Fields Inst. Commun., Springer, New York,
  2015, pp.~185--237.

\bibitem{Bogatyrev-2009}
{\sc A.~B. Bogatyr\"{e}v}, {\em Prime form and {S}chottky model}, Comput.
  Methods Funct. Theory, 9 (2009), pp.~47--55.

\bibitem{Burke-1983}
{\sc W.~L. Burke}, {\em Manifestly parity invariant electromagnetic theory and
  twisted tensors}, J. Math. Phys., 24 (1983), pp.~65--69.

\bibitem{Burke-1985}
\leavevmode\vrule height 2pt depth -1.6pt width 23pt, {\em Applied differential
  geometry}, Cambridge University Press, Cambridge, 1985.

\bibitem{Caffarelli-Friedman-1985}
{\sc L.~A. Caffarelli and A.~Friedman}, {\em Convexity of solutions of
  semilinear elliptic equations}, Duke Math. J., 52 (1985), pp.~431--456.

\bibitem{Cohn-1980}
{\sc H.~Cohn}, {\em Conformal mapping on {R}iemann surfaces}, Dover
  Publications, Inc., New York, 1980.
\newblock Reprint of the 1967 edition, Dover Books on Advanced Mathematics.

\bibitem{Courant-Hilbert-1943}
{\sc R.~Courant and D.~Hilbert}, {\em Methoden der {M}athematischen {P}hysik.
  {V}ols. {I}, {II}}, Interscience Publishers, Inc., New York, 1943.

\bibitem{Courant-Hilbert-1968}
\leavevmode\vrule height 2pt depth -1.6pt width 23pt, {\em Methoden der
  mathematischen {P}hysik. {I}}, Heidelberger Taschenb\"{u}cher, Band 30,
  Springer-Verlag, Berlin-New York, 1968.
\newblock Dritte Auflage.

\bibitem{Crowdy-2008}
{\sc D.~Crowdy}, {\em Geometric function theory: a modern view of a classical
  subject}, Nonlinearity, 21 (2008), pp.~T205--T219.

\bibitem{Crowdy-2010}
\leavevmode\vrule height 2pt depth -1.6pt width 23pt, {\em The
  {S}chottky-{K}lein prime function on the {S}chottky double of planar
  domains}, Comput. Methods Funct. Theory, 10 (2010), pp.~501--517.

\bibitem{Crowdy-Marshall-2006}
{\sc D.~Crowdy and J.~Marshall}, {\em Conformal mappings between canonical
  multiply connected domains}, Comput. Methods Funct. Theory, 6 (2006),
  pp.~59--76.

\bibitem{Crowdy-Marshall-2007}
\leavevmode\vrule height 2pt depth -1.6pt width 23pt, {\em Green's functions
  for {L}aplace's equation in multiply connected domains}, IMA J. Appl. Math.,
  72 (2007), pp.~278--301.

\bibitem{Crowdy-Marshall-2005}
{\sc D.~G. Crowdy and J.~S. Marshall}, {\em The motion of a point vortex around
  multiple circular islands}, Phys. Fluids, 17 (2005), pp.~056602, 13.

\bibitem{Davis-1974}
{\sc P.~J. Davis}, {\em The {S}chwarz function and its applications}, The
  Mathematical Association of America, Buffalo, N. Y., 1974.
\newblock The Carus Mathematical Monographs, No. 17.

\bibitem{deBranges-1985}
{\sc L.~de~Branges}, {\em A proof of the {B}ieberbach conjecture}, Acta Math.,
  154 (1985), pp.~137--152.

\bibitem{deRham-1984}
{\sc G.~de~Rham}, {\em Differentiable manifolds}, vol.~266 of Grundlehren der
  Mathematischen Wissenschaften [Fundamental Principles of Mathematical
  Sciences], Springer-Verlag, Berlin, 1984.
\newblock Forms, currents, harmonic forms, Translated from the French by F. R.
  Smith, With an introduction by S. S. Chern.

\bibitem{Hoker-2024}
{\sc E.~D'Hoker, M.~Hidding, and O.~Schlotterer}, {\em Cyclic products of
  higher-genus {S}zeg\"{o} kernels, modular tensors, and polylogarithms}, Phys.
  Rev. Lett., 133 (2024), pp.~Paper No. 021602, 7.

\bibitem{Doob-1984}
{\sc J.~L. Doob}, {\em Classical potential theory and its probabilistic
  counterpart}, vol.~262 of Grundlehren der Mathematischen Wissenschaften
  [Fundamental Principles of Mathematical Sciences], Springer-Verlag, New York,
  1984.

\bibitem{Duren-Schuster-2004}
{\sc P.~Duren and A.~Schuster}, {\em Bergman spaces}, vol.~100 of Mathematical
  Surveys and Monographs, American Mathematical Society, Providence, RI, 2004.

\bibitem{Duren-1970}
{\sc P.~L. Duren}, {\em Theory of {$H^{p}$} spaces}, Pure and Applied
  Mathematics, Vol. 38, Academic Press, New York-London, 1970.

\bibitem{Duren-1983}
\leavevmode\vrule height 2pt depth -1.6pt width 23pt, {\em Univalent
  functions}, vol.~259 of Grundlehren der mathematischen Wissenschaften
  [Fundamental Principles of Mathematical Sciences], Springer-Verlag, New York,
  1983.

\bibitem{Duren-Schiffer-1991}
{\sc P.~L. Duren and M.~M. Schiffer}, {\em Robin functions and energy
  functionals of multiply connected domains}, Pacific J. Math., 148 (1991),
  pp.~251--273.

\bibitem{Epstein-1965}
{\sc B.~Epstein}, {\em Orthogonal families of analytic functions}, The
  Macmillan Company, New York; Collier Macmillan Ltd., London, 1965.

\bibitem{Faltings-1984}
{\sc G.~Faltings}, {\em Calculus on arithmetic surfaces}, Ann. of Math. (2),
  119 (1984), pp.~387--424.

\bibitem{Farkas-Kra-1992}
{\sc H.~M. Farkas and I.~Kra}, {\em Riemann surfaces}, vol.~71 of Graduate
  Texts in Mathematics, Springer-Verlag, New York, second~ed., 1992.

\bibitem{Fay-1973}
{\sc J.~D. Fay}, {\em Theta functions on {R}iemann surfaces}, Lecture Notes in
  Mathematics, Vol. 352, Springer-Verlag, Berlin, 1973.

\bibitem{Federer-1969}
{\sc H.~Federer}, {\em Geometric measure theory}, Die Grundlehren der
  mathematischen Wissenschaften, Band 153, Springer-Verlag New York Inc., New
  York, 1969.

\bibitem{Fekete-1923}
{\sc M.~Fekete}, {\em \"{U}ber die {V}erteilung der {W}urzeln bei gewissen
  algebraischen {G}leichungen mit ganzzahligen {K}oeffizienten}, Math. Z., 17
  (1923), pp.~228--249.

\bibitem{Flucher-1999}
{\sc M.~Flucher}, {\em Variational problems with concentration}, vol.~36 of
  Progress in Nonlinear Differential Equations and their Applications,
  Birkh\"{a}user Verlag, Basel, 1999.

\bibitem{Flucher-Gustafsson-1997}
{\sc M.~Flucher and B.~Gustafsson}, {\em Vortex motion in two-dimensional
  hydrodynamics}, Royal Instituste of Technology Research Bulletins,
  TRITA-MAT-1997-MA-02 (1997), pp.~1--24.

\bibitem{Forster-1981}
{\sc O.~Forster}, {\em Lectures on {R}iemann surfaces}, vol.~81 of Graduate
  Texts in Mathematics, Springer-Verlag, New York-Berlin, 1981.
\newblock Translated from the German by Bruce Gilligan.

\bibitem{Frankel-2012}
{\sc T.~Frankel}, {\em The {G}eometry of {P}hysics}, Cambridge University
  Press, Cambridge, third~ed., 2012.
\newblock An introduction.

\bibitem{Fredholm-1903}
{\sc I.~Fredholm}, {\em Sur une classe d'\'equations fonctionnelles}, Acta
  Math., 27 (1903), pp.~365--390.

\bibitem{Frostman-1935}
{\sc O.~Frostman}, {\em Potentiel d'\'{e}quilibre et capacit\'{e} des ensembles
  avec quelques applications \`{a} la th\'{e}orie des fonctions}, Lunds
  Universitet, 1935.
\newblock Doctoral thesis.

\bibitem{Garabedian-1949}
{\sc P.~R. Garabedian}, {\em Schwarz's lemma and the {S}zeg\"{o} kernel
  function}, Trans. Amer. Math. Soc., 67 (1949), pp.~1--35.

\bibitem{Garabedian-1964}
\leavevmode\vrule height 2pt depth -1.6pt width 23pt, {\em Partial differential
  equations}, John Wiley \& Sons, Inc., New York-London-Sydney, 1964.

\bibitem{Green-1828}
{\sc G.~Green}, {\em Mathematical papers}, Chelsea Publishing Co., Bronx, NY,
  1970.

\bibitem{Griffiths-Harris-1978}
{\sc P.~Griffiths and J.~Harris}, {\em Principles of algebraic geometry},
  Wiley-Interscience [John Wiley \& Sons], New York, 1978.
\newblock Pure and Applied Mathematics.

\bibitem{Grotta-Ragazzo-2024}
{\sc C.~Grotta-Ragazzo}, {\em Vortex on {S}urfaces and {B}rownian {M}otion in
  {H}igher {D}imensions: {S}pecial {M}etrics}, J. Nonlinear Sci., 34 (2024),
  p.~Paper no. 31.

\bibitem{Grotta-Ragazzo-Gustafsson-Koiller-2024}
{\sc C.~Grotta-Ragazzo, B.~Gustafsson, and J.~Koiller}, {\em On the interplay
  between vortices and harmonic flows: {H}odge decomposition of {E}uler's
  equations in 2d}, Regul. Chaotic Dyn., 29 (2024), pp.~241--303.

\bibitem{Grotta-Koiller-Oliva-1994}
{\sc C.~Grotta~Ragazzo, J.~Koiller, and W.~M. Oliva}, {\em On the motion of
  two-dimensional vortices with mass}, J. Nonlinear Sci., 4 (1994),
  pp.~375--418.

\bibitem{Grotta-Ragazzo-Barros-Viglioni-2017}
{\sc C.~Grotta~Ragazzo and H.~H. d.~B. Viglioni}, {\em Hydrodynamic vortex on
  surfaces}, J. Nonlinear Sci., 27 (2017), pp.~1609--1640.

\bibitem{Gunning-1966}
{\sc R.~C. Gunning}, {\em Lectures on {R}iemann surfaces}, Princeton
  Mathematical Notes, Princeton University Press, Princeton, N.J., 1966.

\bibitem{Gunning-1967}
\leavevmode\vrule height 2pt depth -1.6pt width 23pt, {\em Special coordinate
  coverings of {R}iemann surfaces}, Math. Ann., 170 (1967), pp.~67--86.

\bibitem{Gustafsson-1979}
{\sc B.~Gustafsson}, {\em On the motion of a vortex in two-dimensional flow of
  an ideal fluid in simply and multiply connected domains}, Royal Instituste of
  Technology Research Bulletins, TRITA-MAT-1979-7 (1979), pp.~1--109.

\bibitem{Gustafsson-1987}
{\sc B.~Gustafsson}, {\em Application of half-order differentials on {R}iemann
  surfaces to quadrature identities for arc-length}, J. Analyse Math., 49
  (1987), pp.~54--89.

\bibitem{Gustafsson-1990a}
\leavevmode\vrule height 2pt depth -1.6pt width 23pt, {\em On the convexity of
  a solution of {L}iouville's equation}, Duke Math. J., 60 (1990),
  pp.~303--311.

\bibitem{Gustafsson-2004}
\leavevmode\vrule height 2pt depth -1.6pt width 23pt, {\em Lectures on
  balayage}, in Clifford algebras and potential theory, vol.~7 of Univ. Joensuu
  Dept. Math. Rep. Ser., Univ. Joensuu, Joensuu, 2004, pp.~17--63.

\bibitem{Gustafsson-2019}
{\sc B.~Gustafsson}, {\em Vortex motion and geometric function theory: the role
  of connections}, Philos. Trans. Roy. Soc. A, 377 (2019), pp.~20180341, 27.

\bibitem{Gustafsson-2022}
\leavevmode\vrule height 2pt depth -1.6pt width 23pt, {\em Harold {S}.
  {S}hapiro at {KTH}: some personal memories}, Anal. Math. Phys., 12 (2022),
  pp.~Paper No. 41, 4.

\bibitem{Gustafsson-2022a}
\leavevmode\vrule height 2pt depth -1.6pt width 23pt, {\em Vortex pairs and
  dipoles on closed surfaces}, J. Nonlinear Sci., 32 (2022), pp.~38, Paper No.
  62.

\bibitem{Gustafsson-Peetre-1990}
{\sc B.~Gustafsson and J.~Peetre}, {\em Hankel forms on multiply connected
  plane domains. {II}. {T}he case of higher connectivity}, Complex Variables
  Theory Appl., 13 (1990), pp.~239--250.

\bibitem{Gustafsson-Roos-2018}
{\sc B.~Gustafsson and J.~Roos}, {\em Partial balayage on {R}iemannian
  manifolds}, J. Math. Pures Appl. (9), 118 (2018), pp.~82--127.

\bibitem{Gustafsson-Sebbar-2012}
{\sc B.~Gustafsson and A.~Sebbar}, {\em Critical points of {G}reen's function
  and geometric function theory}, Indiana Univ. Math. J., 61 (2012),
  pp.~939--1017.

\bibitem{Gustafsson-Sebbar-2022}
{\sc B.~Gustafsson and A.~Sebbar}, {\em Hadamard's variational formula in terms
  of stress and strain tensors}, Anal. Math. Phys., 12 (2022), pp.~Paper No.
  29, 12.

\bibitem{Gustafsson-Teodorescu-Vasiliev-2014}
{\sc B.~Gustafsson, R.~Teoderscu, and A.~Vasil{\cprime}ev}, {\em Classical and
  stochastic {L}aplacian growth}, Advances in Mathematical Fluid Mechanics,
  Birkh\"auser Verlag, Basel, 2014.

\bibitem{Hawley-Schiffer-1966}
{\sc N.~S. Hawley and M.~Schiffer}, {\em Half-order differentials on {R}iemann
  surfaces}, Acta Math., 115 (1966), pp.~199--236.

\bibitem{Hedenmalm-Korenblum-Zhu-2000}
{\sc H.~Hedenmalm, B.~Korenblum, and K.~Zhu}, {\em Theory of {B}ergman spaces},
  vol.~199 of Graduate Texts in Mathematics, Springer-Verlag, New York, 2000.

\bibitem{Hejhal-1972}
{\sc D.~A. Hejhal}, {\em Theta functions, kernel functions, and {A}belian
  integrals}, American Mathematical Society, Providence, R.I., 1972.
\newblock Memoirs of the American Mathematical Society, No. 129.

\bibitem{Helms-2014}
{\sc L.~L. Helms}, {\em Potential theory}, Universitext, Springer, London,
  second~ed., 2014.

\bibitem{Jost-2001}
{\sc J.~Jost}, {\em Bosonic strings: a mathematical treatment}, vol.~21 of
  AMS/IP Studies in Advanced Mathematics, American Mathematical Society,
  Providence, RI; International Press, Somerville, MA, 2001.

\bibitem{Kang-Makarov-2013}
{\sc N.-G. Kang and N.~G. Makarov}, {\em Gaussian free field and conformal
  field theory}, Ast\'{e}risque,  (2013), pp.~viii+136.

\bibitem{Kawohl-1985}
{\sc B.~Kawohl}, {\em Rearrangements and convexity of level sets in {PDE}},
  vol.~1150 of Lecture Notes in Mathematics, Springer-Verlag, Berlin, 1985.

\bibitem{Kellogg-1967}
{\sc O.~D. Kellogg}, {\em Foundations of potential theory}, Die Grundlehren der
  mathematischen Wissenschaften, Band 31, Springer-Verlag, Berlin-New York,
  1967.
\newblock Reprint from the first edition of 1929.

\bibitem{Kinderlehrer-Stampacchia-1980}
{\sc D.~Kinderlehrer and G.~Stampacchia}, {\em An introduction to variational
  inequalities and their applications}, vol.~88 of Pure and Applied
  Mathematics, Academic Press, Inc. [Harcourt Brace Jovanovich, Publishers],
  New York-London, 1980.

\bibitem{Koebe-1916}
{\sc P.~Koebe}, {\em Abhandlungen zur {T}heorie der konformen {A}bbildung},
  Acta Math., 41 (1916), pp.~305--344.
\newblock IV. Abbildung mehrfach zusammenh\"{a}ngender schlichter Bereiche auf
  Schlitzbereiche.

\bibitem{Koebe-1918}
\leavevmode\vrule height 2pt depth -1.6pt width 23pt, {\em Abhandlungen zur
  {T}heorie der konformen {A}bbildung}, Math. Z., 2 (1918), pp.~198--236.

\bibitem{Koiller-Boatto-2009}
{\sc J.~Koiller and S.~Boatto}, {\em Vortex paires on surfaces}, AIP Conference
  Proceedings 1130; https://doi.org/10.1063/1.3146241, 77 (2009).

\bibitem{Kostov-Krichever-Mineev-Wiegmann-Zabrodin-2001}
{\sc I.~K. Kostov, I.~Krichever, M.~Mineev-Weinstein, P.~B. Wiegmann, and
  A.~Zabrodin}, {\em The {$\tau$}-function for analytic curves}, in Random
  matrix models and their applications, vol.~40 of Math. Sci. Res. Inst. Publ.,
  Cambridge Univ. Press, Cambridge, 2001, pp.~285--299.

\bibitem{Krantz-2004}
{\sc S.~G. Krantz}, {\em Complex analysis: the geometric viewpoint}, vol.~23 of
  Carus Mathematical Monographs, Mathematical Association of America,
  Washington, DC, second~ed., 2004.

\bibitem{Krichever-Marshakov-Zabrodin-2005}
{\sc I.~Krichever, A.~Marshakov, and A.~Zabrodin}, {\em Integrable structure of
  the {D}irichlet boundary problem in multiply-connected domains}, Comm. Math.
  Phys., 259 (2005), pp.~1--44.

\bibitem{Lamb-1993}
{\sc H.~Lamb}, {\em Hydrodynamics}, Cambridge Mathematical Library, Cambridge
  University Press, Cambridge, sixth~ed., 1993.
\newblock With a foreword by R. A. Caflisch [Russel E. Caflisch].

\bibitem{Landkof-1972}
{\sc N.~S. Landkof}, {\em Foundations of modern potential theory}, vol.~Band
  180 of Die Grundlehren der mathematischen Wissenschaften, Springer-Verlag,
  New York-Heidelberg, 1972.
\newblock Translated from the Russian by A. P. Doohovskoy.

\bibitem{Lang-1988}
{\sc S.~Lang}, {\em Introduction to {A}rakelov theory}, Springer-Verlag, New
  York, 1988.

\bibitem{Lin-1943}
{\sc C.~C. Lin}, {\em On the {M}otion of {V}ortices in {T}wo {D}imensions},
  University of Toronto Studies, Applied Mathematics Series, no. 5, University
  of Toronto Press, Toronto, Ont., 1943.

\bibitem{Marchioro-Pulvirenti-1994}
{\sc C.~Marchioro and M.~Pulvirenti}, {\em Mathematical theory of
  incompressible nonviscous fluids}, vol.~96 of Applied Mathematical Sciences,
  Springer-Verlag, New York, 1994.

\bibitem{Minda-Wright-1982}
{\sc D.~Minda and D.~J. Wright}, {\em Univalence criteria and the hyperbolic
  metric}, Rocky Mountain J. Math., 12 (1982), pp.~471--479.

\bibitem{Mineev-1990}
{\sc M.~B. Mineev}, {\em A finite polynomial solution of the two-dimensional
  interface dynamics}, Phys. D, 43 (1990), pp.~288--292.

\bibitem{Moreau-1962}
{\sc J.-J. Moreau}, {\em Fonctions convexes duales et points proximaux dans un
  espace hilbertien}, C. R. Acad. Sci. Paris, 255 (1962), pp.~2897--2899.

\bibitem{Nehari-1952}
{\sc Z.~Nehari}, {\em Conformal mapping}, McGraw-Hill Book Co., Inc., New York,
  Toronto, London, 1952.

\bibitem{Nevanlinna-1953}
{\sc R.~Nevanlinna}, {\em Uniformisierung}, vol.~Band LXIV of Die Grundlehren
  der mathematischen Wissenschaften in Einzeldarstellungen mit besonderer
  Ber\"ucksichtigung der Anwendungsgebiete, Springer-Verlag,
  Berlin-G\"ottingen-Heidelberg, 1953.

\bibitem{Newton-1687}
{\sc I.~S. Newton}, {\em Philosophiae naturalis principia mathematica}, William
  Dawson \& Sons, Ltd., London, 1952 (Facsimile reproduction of the version of
  1686).

\bibitem{Okikiolu-2009}
{\sc K.~Okikiolu}, {\em A negative mass theorem for surfaces of positive
  genus}, Comm. Math. Phys., 290 (2009), pp.~1025--1031.

\bibitem{Pommerenke-1975}
{\sc C.~Pommerenke}, {\em Univalent functions}, Vandenhoeck \& Ruprecht,
  G\"ottingen, 1975.
\newblock With a chapter on quadratic differentials by Gerd Jensen, Studia
  Mathematica/Mathematische Lehrb\"ucher, Band XXV.

\bibitem{Raina-1989}
{\sc A.~K. Raina}, {\em Fay's trisecant identity and conformal field theory},
  Comm. Math. Phys., 122 (1989), pp.~625--641.

\bibitem{Ransford-1995}
{\sc T.~Ransford}, {\em Potential theory in the complex plane}, vol.~28 of
  London Mathematical Society Student Texts, Cambridge University Press,
  Cambridge, 1995.

\bibitem{Richardson-1980}
{\sc S.~Richardson}, {\em Vortices, {L}iouville's equation and the {B}ergman
  kernel function}, Mathematika, 27 (1980), pp.~321--334.

\bibitem{Saff-Totik-1997}
{\sc E.~B. Saff and V.~Totik}, {\em Logarithmic potentials with external
  fields}, vol.~316 of Grundlehren der Mathematischen Wissenschaften
  [Fundamental Principles of Mathematical Sciences], Springer-Verlag, Berlin,
  1997.
\newblock Appendix B by Thomas Bloom.

\bibitem{Sakai-1969}
{\sc M.~Sakai}, {\em On constants in extremal problems of analytic functions},
  K\=odai Math. Sem. Rep., 21 (1969), pp.~223--225.

\bibitem{Sakai-1982}
\leavevmode\vrule height 2pt depth -1.6pt width 23pt, {\em Quadrature domains},
  vol.~934 of Lecture Notes in Mathematics, Springer-Verlag, Berlin, 1982.

\bibitem{Sario-Oikawa-1969}
{\sc L.~Sario and K.~Oikawa}, {\em Capacity functions}, Die Grundlehren der
  mathematischen Wissenschaften, Band 149, Springer-Verlag New York Inc., New
  York, 1969.

\bibitem{Schiffer-1943}
{\sc M.~Schiffer}, {\em The span of multiply connected domains}, Duke Math. J.,
  10 (1943), pp.~209--216.

\bibitem{Schiffer-Spencer-1954}
{\sc M.~Schiffer and D.~C. Spencer}, {\em Functionals of finite {R}iemann
  surfaces}, Princeton University Press, Princeton, N. J., 1954.

\bibitem{Schottky-1877}
{\sc F.~Schottky}, {\em Ueber die conforme {A}bbildung mehrfach
  zusammenh\"angender ebener {F}l\"achen}, J. Reine Angew. Math., 83 (1877),
  pp.~300--351.

\bibitem{Shapiro-1992}
{\sc H.~S. Shapiro}, {\em The {S}chwarz function and its generalization to
  higher dimensions}, University of Arkansas Lecture Notes in the Mathematical
  Sciences, 9, John Wiley \& Sons Inc., New York, 1992.
\newblock A Wiley-Interscience Publication.

\bibitem{Steiner-2005}
{\sc J.~Steiner}, {\em A geometrical mass and its extremal properties for
  metrics on {$S^2$}}, Duke Math. J., 129 (2005), pp.~63--86.

\bibitem{Suita-1972}
{\sc N.~Suita}, {\em Capacities and kernels on {R}iemann surfaces}, Arch.
  Rational Mech. Anal., 46 (1972), pp.~212--217.

\bibitem{Takhtajan-2001}
{\sc L.~A. Takhtajan}, {\em Free bosons and tau-functions for compact {R}iemann
  surfaces and closed smooth {J}ordan curves. {C}urrent correlation functions},
  vol.~56, 2001, pp.~181--228.
\newblock EuroConf\'erence Mosh\'e{} Flato 2000, Part III (Dijon).

\bibitem{Takhtajan-Teo-2006}
{\sc L.~A. Takhtajan and L.-P. Teo}, {\em Quantum {L}iouville theory in the
  background field formalism. {I}. {C}ompact {R}iemann surfaces}, Comm. Math.
  Phys., 268 (2006), pp.~135--197.

\bibitem{Treves-1975}
{\sc F.~Tr{\`e}ves}, {\em Basic linear partial differential equations},
  Academic Press [A subsidiary of Harcourt Brace Jovanovich, Publishers], New
  York-London, 1975.
\newblock Pure and Applied Mathematics, Vol. 62.

\bibitem{Tsuji-1975}
{\sc M.~Tsuji}, {\em Potential theory in modern function theory}, Chelsea
  Publishing Co., New York, 1975.
\newblock Reprinting of the 1959 original.

\bibitem{Warner-1983}
{\sc F.~W. Warner}, {\em Foundations of differentiable manifolds and {L}ie
  groups}, vol.~94 of Graduate Texts in Mathematics, Springer-Verlag, New
  York-Berlin, 1983.
\newblock Corrected reprint of the 1971 edition.

\bibitem{Weyl-1964}
{\sc H.~Weyl}, {\em Die {I}dee der {R}iemannschen {F}l\"ache}, Vierte Auflage.
  Unver\"anderter Nachdruck der dritten, Vollst\"andig umgearbeiteten Auflage,
  B. G. Teubner Verlagsgesellschaft, Stuttgart, 1964.

\bibitem{Wiegmann-Zabrodin-2000}
{\sc P.~B. Wiegmann and A.~Zabrodin}, {\em Conformal maps and integrable
  hierarchies}, Comm. Math. Phys., 213 (2000), pp.~523--538.

\bibitem{Wolpert-2018}
{\sc S.~A. Wolpert}, {\em Schiffer variations and {A}belian differentials},
  Adv. Math., 333 (2018), pp.~497--522.

\bibitem{Zamolodchikov-2005}
{\sc A.~B. Zamolodchikov}, {\em A three-point function of minimal {L}iouville
  gravitation}, Teoret. Mat. Fiz., 142 (2005), pp.~218--234.

\end{thebibliography}

\end{document}